\documentclass[a4paper,leqno,11pt,twoside]{amsart}
\usepackage{amsmath,amsfonts,amssymb,amsthm,amscd}
\usepackage[utf8]{inputenc}
\usepackage[english]{babel}
\usepackage[colorlinks=true,citecolor=blue, urlcolor=blue, linkcolor=blue,pagebackref]{hyperref}
\usepackage[top=1.in,bottom=1.in,left=1.in,right=1.in]{geometry} 
\usepackage{graphicx}
\usepackage{wrapfig}
\usepackage{paralist}
\usepackage{tabto}
\usepackage{standalone}
\usepackage{xfrac}
\usepackage{float}

\usepackage{pgf,tikz}
\usetikzlibrary{arrows,chains,
 decorations.pathreplacing,
shapes,%
}

\usepackage[all]{xy} 

\usepackage{enumerate}
\usepackage{xcolor}
\usepackage{aliascnt}

\theoremstyle{plain}
\newtheorem{thm}{Theorem}[section]

\newtheorem{thmIntr}{Theorem}

\newaliascnt{propIntr}{thmIntr}
\newtheorem{propIntr}[propIntr]{Proposition}

\newaliascnt{lem}{thm}
\newtheorem{lem}[lem]{Lemma}
\aliascntresetthe{lem}

\newaliascnt{cor}{thm}
\newtheorem{cor}[cor]{Corol\-lary}
\aliascntresetthe{cor}

\newaliascnt{prop}{thm}
\newtheorem{prop}[prop]{Proposition}
\aliascntresetthe{prop}

\theoremstyle{definition}

\newaliascnt{rem}{thm}
\newtheorem{rem}[rem]{Remark}
\aliascntresetthe{rem}

\newaliascnt{defn}{thm}
\newtheorem{defn}[defn]{Definition}
\aliascntresetthe{defn}

\newaliascnt{ex}{thm}
\newtheorem{ex}[ex]{Ex\-ample}
\aliascntresetthe{ex}

\numberwithin{equation}{section}

\def\bK{\ensuremath{\mathbb{K}}}
\def\bN{\ensuremath{\mathbb{N}}}

\def\bQ{\ensuremath{\mathbb{Q}}}
\def\bR{\ensuremath{\mathbb{R}}}
\def\bZ{\ensuremath{\mathbb{Z}}}
\def\bC{\ensuremath{\mathbb{C}}}

\def\cA{\ensuremath{\mathcal{A}}}

\def\cC{\ensuremath{\mathcal{C}}}

\def\cF{\ensuremath{\mathcal{F}}}

\def\cH{\ensuremath{\mathcal{H}}}
\def\cI{\ensuremath{\mathcal{I}}}
\def\cO{\ensuremath{\mathcal{O}}}
\def\cP{\ensuremath{\mathcal{P}}}
\def\cT{\ensuremath{\mathcal{T}}}

\def\wbeta{\ensuremath{\widetilde{\beta}}}

\def\wF{\ensuremath{\widetilde{F}}}

\def\Db{\mathop{\mathrm{D}^{\mathrm{b}}}\nolimits}

\DeclareMathOperator{\Pic}{Pic}

\DeclareMathOperator{\Ext}{Ext}

\DeclareMathOperator{\Hom}{Hom}
\DeclareMathOperator{\sHom}{\mathcal{H}\emph{om}}
\DeclareMathOperator{\NS}{NS}
\DeclareMathOperator{\rk}{rk}

\DeclareMathOperator{\ch}{ch}
\DeclareMathOperator{\Coh}{Coh}
\DeclareMathOperator{\Hilb}{Hilb}

\DeclareMathOperator{\odisc}{\overline{\Delta}}

\DeclareMathOperator{\Stab}{Stab}
\def\blank{\underline{\hphantom{A}}}
\DeclareMathOperator{\Char}{char}

\newcommand{\hyp}[1]{\ensuremath{H}_{#1}}
\def\rf{\ensuremath{\operatorname{chd}}}

\definecolor{applegreen}{rgb}{0.55, 0.71, 0.0}

\newcommand{\set}[1]{\left\{#1\right\}}

\title[Chern degree functions]{Chern degree functions}
\author[M.~Lahoz and A.~Rojas]{Mart\'{\i} Lahoz and Andr\'es Rojas}

\address{M.L.: Departament de Matem\`atiques i Inform\`atica,
Universitat de Barcelona, Gran Via de les Corts Catalanes, 585, 08007 Barcelona, Spain}
\email{marti.lahoz@ub.edu}
\urladdr{\url{http://www.ub.edu/geomap/lahoz/}}

\address{A.R.: BGSMath -- Departament de Matem\`atiques i Inform\`atica,
Universitat de Barcelona, Gran Via de les Corts Catalanes, 585, 08007 Barcelona, Spain}
\curraddr{Mathematisches Institut, Universität Bonn, Endenicher
Allee 60, 53115 Bonn, Germany}
\email{arojas@math.uni-bonn.de} 

\begin{document}

\begin{abstract}
We introduce Chern degree functions for complexes of coherent sheaves on a polarized surface, which encode information given by stability conditions on the boundary of the $(\alpha,\beta)$-plane.
We prove that these functions extend to continuous real valued functions and we study their differentiability in terms of stability.
For abelian surfaces, Chern degree functions coincide with the cohomological rank functions defined by Jiang--Pareschi.
We illustrate in some geometrical situations a general strategy to compute these functions.
\end{abstract}


\thanks{M.~L.~was supported by a Ram\'on y Cajal fellowship and partially by the Spanish MINECO research project PID2019-104047GB-I00.
A.~R.~was partially supported by the Spanish MINECO grants MDM-2014-0445, RYC-2015-19175, and PID2019-104047GB-I00.}

\maketitle

\setcounter{tocdepth}{1}
\tableofcontents

\section{Introduction}

In the context of irregular varieties, Barja, Pardini and Stoppino \cite{BPS:LinearSystems} introduced the \emph{continuous rank function} associated to a line bundle, a continuous function defined on a line in the space of $\bR$-divisor classes with similar properties to those of the volume function.
Shortly after, in \cite{JP} Jiang and Pareschi generalized this notion to the \emph{cohomological rank functions} $h^i_{F,L}$ associated to a coherent sheaf (or more generally, a bounded complex of coherent sheaves) $F$ on a polarized abelian variety $(A,L)$.

These functions have received several applications.
For instance, they have been used to prove new Clifford–Severi inequalities, i.e.~geographical lower bounds of the volume of a line bundle and characterize the polarized varieties where the bound is attained (\cite{BPS:Clifford-Severi,Jiang:Severi}).
They have been also applied to the study of syzygies of abelian varieties (see \cite{Caucci,Jiang:Syzygies,Ito3} and references therein).

Given $x\in\bQ$, Jiang and Pareschi use the multiplication maps on $A$ to make sense of the (generic) $i$-th cohomological rank of $F$ twisted with the rational polarization $xL$.
This defines the number $h^i_{F,L}(x)$.
Their two main results about the general structure of these functions, a priori only defined over the rational numbers, can be summarized as follows:

\vspace{-.2cm}
\begin{enumerate}[{\rm (1)}]
\item \cite[Corollary~2.6]{JP} Every $x_0\in\bQ$ admits a left (resp.~right) neighborhood where the function $h^i_{F,L}$ is given by an explicit polynomial $P^-$ (resp.~$P^+$) depending on $x_0$.

\item \cite[Theorem~3.2]{JP} The functions extend to continuous real functions of real variable.
\end{enumerate}

These results are proved via an extensive use of the Fourier-Mukai transform on the abelian variety, which also justifies the need to extend the definition of the cohomological rank functions to arbitrary objects in the derived category of coherent sheaves.

In the case of elliptic curves, it is well known that the Harder-Narasimhan filtration of a coherent sheaf $F$ provides a precise description of its cohomological rank functions (see \autoref{CRFelliptic} for details). 
Nevertheless for higher-dimensional abelian varieties only a few concrete examples of functions are known, and a general structure is far from being understood.
For instance, it is not known whether cohomological rank functions are always piecewise polynomial (\cite[Remark~2.8]{JP}).

In the present paper we investigate the relation between the functions and stability in the case of surfaces, from a twofold perspective: not only that of obtaining a better understanding of cohomological rank functions on abelian surfaces, but also that of proposing similar invariants on arbitrary (i.e.~not necessarily abelian) polarized surfaces.
This can be done in a unified way, by means of the \emph{Chern degree functions} that we attach to objects in the derived category $\Db(X)$ of any smooth polarized surface $(X,L)$. 
These functions are our main object of study, and their definition relies on a certain set of (weak) stability conditions on $\Db(X)$.

More precisely, let $(X,L)$ be a smooth polarized surface.
Since the seminal work of Bridgeland \cite{Bridgeland:K3} (for K3 surfaces), then extended by Arcara--Bertram \cite{Aaron-Daniele} to arbitrary smooth surfaces, one can consider a certain Bridgeland stability condition $\sigma_{\alpha,\beta}=(Z_{\alpha,\beta},\Coh^\beta(X))$ for every $(\alpha,\beta)\in\bR_{>0}\times\bR$, where $\Coh^\beta(X)$ is a heart of a bounded t-structure on $\Db(X)$.
In the corresponding \textit{$(\alpha,\beta)$-plane} of stability conditions, the wall-crossing behaviour is very well understood.
A similar construction holds in positive characteristic as well, thanks to work of Koseki \cite{Koseki} (even though the cases we will consider were already covered by the work of Langer \cite{Langer}).

When $\alpha=0$, one cannot ensure in general that $\sigma_{0,\beta}$ is a Bridgeland stability condition.
Roughly speaking, this depends on whether the classical Bogomolov inequality for slope-semistable sheaves of slope $\beta$ is sharp or not; in case of non-sharpness, the construction of Bridgeland stability conditions can be extended to a larger region containing $\sigma_{0,\beta}$ in its interior, as explained in the recent work \cite[section~3]{FLZ}.

Regardless of whether $\sigma_{0,\beta}$ is a Bridgeland stability condition or not, for every $\beta\in\bQ$ one can at least ensure the existence of Harder-Narasimhan filtrations with respect to the tilt slope $\nu_{0,\beta}$ induced by $\sigma_{0,\beta}$: this defines a  \textit{weak stability condition}.
For instance, in the literature $\sigma_{0,0}$ has also been called \emph{Brill-Noether stability condition} (see \cite[Definition~2.11]{LiInv}).
The Chern degree functions may be thought of as a way of classifying the information provided by these (weak) stability conditions.

Essentially, to every object $E\in\Db(X)$ we associate nonnegative \emph{Chern degree functions}
\[
\rf^k_{E,L}:\bQ\to\bQ_{\geq0}
\]
for every $k\in\bZ$, related by the identity $\sum (-1)^k\cdot \rf^k_{E,L}(x)=\ch_2^{-x}(E)$, where $\ch_2^t$ is the second  twisted Chern character defined by $\ch_2^{t}=\ch_2-tL\cdot\ch_1+\frac{t^2}2L^2\cdot\operatorname{rk}$.
For instance, if $F\in\Coh^\beta(X)$ only $\rf^0_{F,L}$ and $\rf^1_{F,L}$ are nonzero at $-\beta$, and 
\[\rf^0_{F,L}(-\beta)=\ch_2^{\beta}(\wF),\]
where $\wF$ is the maximal subobject $\wF\subset F$ with the property that any quotient of $\wF$ has positive tilt slope (see \autoref{sec:Chern_deg} for the general definition).

Our first main result about these functions is:

\begin{thmIntr}\label{main}
Let $(X,L)$ be a smooth polarized projective surface over an algebraically closed field $\bK$ of arbitrary characteristic.
If $\Char \bK>0$, assume that $X$ is neither of general type nor quasi-elliptic with $\kappa(X)=1$.

Then, for every object $E\in\Db(X)$ and $k\in\bZ$ the following hold:
\begin{enumerate}[{\rm (1)}]
 \item\label{item:main1} Every rational number $x_0\in\bQ$ admits a left (resp.~right) neighborhood where the function $\rf^k_{E,L}$ is given by an explicit polynomial $P^-$ (resp.~$P^+$) depending on $x_0$, satisfying $P^-(x_0)=\rf^k_{E,L}(x_0)=P^+(x_0)$.
 \item\label{item:main2} The function $\rf^k_{E,L}$ extends to a continuous real function of real variable.
\end{enumerate}
\end{thmIntr}

Let us specify that, even in the case where $F\in\Coh^{\beta_0}(X)$, where $\rf^0_{F,L}(-\beta_0)=\ch_2^{\beta_0}(\wF)$,
the polynomial $\ch_2^{-x}(\wF)$ does not necessarily provide a local polynomial expression of $\rf^0_{F,L}(x)$ for $x$ in a neighborhood of $-\beta_0$.
For instance, the polynomial expression in a left neighborhood of $-\beta_0$ is given by $\ch_2^{-x}(G)$, where $G\subset \wF$ is a certain subobject satisfying $Z_{0,\beta_0}(\wF/G)=0$ and appearing in the Harder-Narasimhan filtration of $F$ with respect to $\sigma_{0,\beta}$ for all small enough $\beta>\beta_0$.

In other words, to control the local polynomial expressions one is led to determine the Harder-Narasimhan filtrations of $F$ with respect to $\sigma_{0,\beta}$ as $\beta$ tends to $\beta_0$;
we call such filtrations the \emph{weak limit filtrations} of $F$ at $\beta_0$ (see \autoref{weakLimitHN} for a precise definition).
Their study requires a good understanding of \emph{Bridgeland limit filtrations}, namely Harder-Narasimhan filtrations with respect to the (honest) Bridgeland stability conditions $\sigma_{\alpha,\beta}$ as $\alpha$ tends to $0$ (see \autoref{BridLimitHN}).

These auxiliary notions may be of independent interest to study the boundary of the stability manifold (see also \autoref{variant});
as a tool for \autoref{main}.\eqref{item:main1}, we prove their existence for objects without Harder-Narasimhan factors of vanishing tilt slope.

\begin{thmIntr}\label{existenceIntr}
Let $\beta_0\in \bQ$ and $F\in\Coh^{\beta_0}(X)$ be an object having no Harder-Narasimhan factor with respect to $\sigma_{0,\beta_0}$ of vanishing $\nu_{0,\beta_0}$.
Then,
\begin{enumerate}[{\rm (1)}]
\item\label{exist:Br} $F$ admits a Bridgeland limit Harder-Narasimhan filtration at $\beta_0$.
\item\label{exist:Wk} $F$ admits a weak limit Harder-Narasimhan filtration at $\beta_0$.
\end{enumerate}
\end{thmIntr}

The extension of the Chern degree functions to continuous real functions is obtained by integration after using the local polynomial expression to bound the derivative, following the approach of \cite{JP} for the cohomological rank functions. 

Once the continuity is settled, it is natural to study the rational \emph{critical points} (i.e.~rational points where the Chern degree functions are not of class $\cC^\infty$).
In particular we obtain the following result characterizing the rational points where the functions are not $\cC^1$.
This result is a corollary of our characterization of the critical points in terms of stability (see \autoref{critical}, as well as the beginning of section~\ref{sec:critical} for a discussion when $F\not\in \Coh^{\beta}(X)$).
\begin{propIntr}\label{mainCritical}
Let $\beta\in \bQ$.
If $F\in\Coh^{\beta}(X)$, then the Chern degree functions of $F$ are not differentiable at $-\beta$ if and only if $F$ has a Harder-Narasimhan factor (with respect to $\sigma_{0,\beta}$) of vanishing tilt slope.
\end{propIntr}

We must point out that the Chern degree functions seem to be especially nontrivial for polarized surfaces on which $\sigma_{0,\beta}$ fails to be a Bridgeland stability condition at every  $\beta\in\bQ$.
Indeed, in those cases we do not know how to conclude from \autoref{main} that the Chern degree functions are piecewise polynomial, as one would expect.
Among these polarized surfaces, it is well-known that one finds abelian surfaces.
As we explain in \autoref{exGiesekerkernel}, this is also the case for surfaces with finite Albanese map (and polarization pulled back from the Albanese variety), which in particular covers polarized irregular surfaces with Picard number 1\footnote{This provides new examples $(X,L)$ where the Le Potier function $\Phi_{X,L}$ is explicitly known, which may be of independent interest since allows to describe their Bridgeland stability manifold (cf. \cite[Section 3]{FLZ}).}.
In such cases of irregular (non-abelian) surfaces, the relation of the Chern degree functions to the original continuous rank functions, or more generally to the cohomological rank functions of the push-forwarded object via the Albanese map, remains mysterious.

On the opposite side, for regular surfaces such as K3 surfaces, the Chern degree functions are clearly piecewise polynomial, since their stability manifold is certainly larger than the $(\alpha,\beta)$-plane. 
This indicates that these functions are probably not the correct notion to deal with regular surfaces; in \autoref{variant} we propose a variant (explicitly for K3 surfaces of Picard number 1).

While for arbitrary polarized surfaces we plan to study in future work these functions, in the case of abelian surfaces the Chern degree functions recover the cohomological rank functions of Jiang and Pareschi.
This establishes a natural analogue with the structure of cohomological rank functions on elliptic curves:

\begin{thmIntr}\label{mainCRF} 
Let $(X,L)$ be a a polarized abelian surface.
Then, the Chern degree function $\rf^k_{E,L}$ equals the cohomological rank function $h^k_{E,L}$ for every object $E\in\Db(X)$ and $k\in\bZ$.
\end{thmIntr}

Note that this shows, in particular, that the cohomological rank functions of an object at $-\beta$ split into simpler pieces, corresponding to its Harder-Narasimhan factors with respect to $\sigma_{0,\beta}$.
Conversely, knowing the cohomological rank functions of an object one can deduce constraints on its behaviour with respect to stability (the easiest case being an object that never destabilizes, see \autoref{trivial}).

This new presentation turns out to be significantly useful to understand cohomological rank functions.
For instance, \autoref{mainCritical} (and more generally \autoref{critical}) characterizes their differentiability at rational critical points, which in \cite[Proposition~4.4]{JP} was settled a sufficient condition to control the dimension of certain jump loci. 
Observe also that, thanks to \autoref{main}.\eqref{item:main2}, we obtain that cohomological rank functions extend to continuous real functions in positive characteristic as well (a case not covered by Jiang and Pareschi).

Furthermore, the stability viewpoint seems to be fruitful for the computation of particular examples, if one exploits the wall-crossing behaviour in the $(\alpha,\beta)$-plane.
We end the paper by focusing on Gieseker semistable sheaves, which are well-known to be semistable for large values of $\alpha$ (at the so-called \emph{Gieseker chamber}).

For them, we propose a method based on successive destabilizations along the $(\alpha,\beta)$-plane; in many concrete situations, this method gives a explicit description of $\rf^0_{F,L}$ as a piecewise polynomial function.
We illustrate this phenomenon with the ideal sheaves of 0-dimensional subschemes of low length on principally polarized abelian surfaces.

Another concrete example where this method works is the case of the ideal sheaf of one point, for abelian surfaces endowed with a polarization of arbitrary type. 
As explained before, such functions have recently received considerable attention as a tool to understand syzygies of abelian varieties, thanks to the results of Caucci \cite{Caucci}.
A computation of this function for abelian surfaces appears in \cite{Rojas}, where
the second author improves the bound given in \cite{Ito} for the maximal critical point in the case of arbitrarily polarized abelian surfaces.

\subsection*{Plan of the paper} 
After some preliminaries on the $(\alpha,\beta)$-plane of stability conditions given in \autoref{sec:prelim}, we define the Chern degree functions in \autoref{sec:Chern_deg} and prove some elementary properties of cohomological nature.
We also discuss how one could extend these functions to the boundary of the manifold of geometric Bridgeland stability conditions.

Section \ref{sec:LocalExpr} constitutes the technical core of the paper.
First, we introduce in \autoref{subsec:Bridgelandlimit} Bridgeland limit filtrations and prove \autoref{existenceIntr}.\eqref{exist:Br}.
In \autoref{subsec:LocalExpr2} we study the weak limit filtrations and we prove \autoref{existenceIntr}.\eqref{exist:Wk}, by means of Bridgeland limit filtrations and a systematic exploitation of wall-crossing. 
This almost gives \autoref{main}.\eqref{item:main1}; the proof is completed in \autoref{subsec:LocalExpr3} by considering the exceptional situations.

The proof of \autoref{main}.\eqref{item:main2} occupies the first part of \autoref{sec:ExtensionToR}.
The rest of \autoref{sec:ExtensionToR} addresses the characterization of rational critical points and their differentiability, which leads to \autoref{critical} (and in particular, to \autoref{mainCritical}).

In \autoref{sec:CaseCRF} we briefly review cohomological rank functions on abelian varieties, and then give a proof of \autoref{mainCRF}.
Finally, \autoref{sec:chdGieseker} deals with Chern degree functions of (twisted) Gieseker semistable sheaves, including a general discussion (\autoref{subsec:Trivialchd} and \autoref{subsec:lowDisc}) illustrated on principally polarized abelian surfaces (\autoref{subsec:CRFGiesekerPPAS}).

\subsection*{Acknowledgements.} 
We would like to thank specially Joan Carles Naranjo, who started this project with us. 
We also thank Miguel \'Angel Barja and Zhi Jiang whose questions and ideas were the starting point of this article, and Federico Caucci, Chunyi Li, Emanuele Macr\`{\i} and Quim Ortega for useful conversations and comments, and an anonymous
referee for her or his suggestions.

\section{Preliminaries on stability conditions}\label{sec:prelim}

In this section, we review the definitions and basic properties of (possibly weak) stability conditions, with a special view towards the $(\alpha,\beta)$-plane defined by a polarization on a smooth projective surface as developed in \cite{Bridgeland:Stab,Bridgeland:K3,Aaron-Daniele}.
We follow the notations of the excellent survey \cite{MS}.

\subsection{Weak and Bridgeland stability conditions}
Let $\cA$ be an abelian category with Grothen\-dieck group $K_0(\cA)$.
\begin{defn}
 A \emph{stability function} (resp.~\emph{weak stability function}) on $\cA$ is a group homomorphism $Z:K_0(\cA)\to\bC$ such that every $F\in\cA\setminus\{0\}$ satisfies
\[
\Im Z(F)\geq0,\;\;\text{and }\;\Im Z(F)=0\Longrightarrow \Re Z(F)<0 \text{ (resp.~}\Im Z(F)=0\Longrightarrow \Re Z(F)\leq0)
\]
\end{defn}

Given a (possibly weak) stability function $Z$ on $\cA$, one can define a slope for objects $F\in\cA$ as $\mu(F)=\frac{-\Re Z(F)}{\Im Z(F)}$, with the convention $\mu(F)=+\infty$ whenever $\Im Z(F)=0$.

\begin{defn}
An object $F\in\cA\setminus\{0\}$ is called \emph{(semi)stable} if, for every nonzero subobject $E\subsetneq F$, the inequality $\mu(E)<(\leq)\;\mu(F/E)$ holds.
\end{defn}

\begin{rem}
Equivalently, $F$ is semistable if and only if $\mu(E)\leq\mu(F)$ for every nonzero $E\subsetneq F$.
The same holds (with strict inequality) for stability of $F$, whenever $Z$ is a stability function.
But if $Z$ is strictly weak, then $F$ may be stable and admit subobjects with $\mu(E)=\mu(F)$; indeed, if $Z(F/E)=0$ then $\mu(E)=\mu(F)<+\infty=\mu(F/E)$ does not contradict stability for $F$.
\end{rem}

Now we give the definition of stability conditions on the bounded derived category $\Db(X)$ of a smooth projective variety $X$.
To this end, we fix a finite rank lattice $\Lambda$ together with a group epimorphism $v:K_0(X)\twoheadrightarrow\Lambda$, where $K_0(X)$ denotes the Grothendieck group of $\Db(X)$.

\begin{defn}
A \emph{Bridgeland stability condition} (resp.~\emph{weak stability condition}) on $\Db(X)$ is a pair $\sigma=(Z,\cA)$, where:
\begin{enumerate}[{\rm (1)}]
    \item $\cA$ is the heart of a bounded t-structure on $\Db(X)$, and $Z:K_0(X)\to\bC$ (\emph{central charge}) is a stability function (resp.~weak stability function) on $\cA$ factoring through $v$.
    \item Every object $F\in\cA\setminus\{0\}$ has a \emph{Harder-Narasimhan} (\emph{HN} for short) \emph{filtration}: there exists a (unique) chain of subobjects
    \[
    0=F_0\subset F_1\subset\ldots\subset F_{r-1}\subset F_r=F
    \]
    with the \emph{HN factors} $F_i/F_{i-1}$ being semistable and $\mu(F_1)>\mu(F_2/F_1)>\ldots>\mu(F/F_{r-1})$.
    \item The \emph{support property} is satisfied: Let $\Lambda_0\subset\Lambda$ be the saturation of the subgroup generated by classes $v(F)$ of objects $F\in\cA$ such that $Z(v(F))=0$.
    Then there exists a quadratic form $Q$ on $(\Lambda/\Lambda_0)\otimes\bR$, 
    such that $Q(v,v)<0$ for every nonzero $v\in(\Lambda/\Lambda_0)\otimes\bR$ with $Z(v)=0$, and $Q(v(F),v(F))\geq0$ for every semistable object $F\in\cA$.
\end{enumerate}
\end{defn}

For $(X,L)$ a polarized smooth projective variety of dimension $n$, the pair $\sigma=(Z,\Coh(X))$ (where $Z=-L^{n-1}\cdot\ch_1+iL^n\cdot\ch_0$)
defines a weak stability condition on $\Db(X)$ (the usual \emph{$\mu_L$-stability}) with respect to the epimorphism $v:K_0(X)\twoheadrightarrow\bZ^2$ given by $v(E)=(L^n\cdot\ch_0(E),L^{n-1}\cdot\ch_1(E))$.
In this case, the support property is guaranteed by the quadratic form $Q=0$.

Let $\Stab_\Lambda(X)$ be the set of Bridgeland stability conditions on $\Db(X)$ (with respect to $\Lambda$ and $v$).
The support property may be thought of as a technical condition allowing to give $\Stab_\Lambda(X)$, when endowed with a certain topology, the structure of a complex manifold.
This result is usually known as \emph{Bridgeland's deformation theorem} (\cite{Bridgeland:Stab}, see also \cite{BayerDef}).
The main consequence is that the regions where objects of a fixed class are (semi)stable behave following a wall and chamber structure.

\subsection{The $(\alpha,\beta)$-plane of a polarized surface}\label{ab-plane}
Let $X$ be a smooth projective surface over an algebraically closed field $\bK$; if $\Char\bK>0$, we will assume that $X$ is neither of general type nor quasi-elliptic with $\kappa(X)=1$.
For a fixed polarization $L\in\NS(X)$ on $X$, we review the construction of the stability conditions forming the so-called \emph{$(\alpha,\beta)$-plane of $L$}.
This is a slice of the connected component of the stability manifold $\Stab(X)$ constructed by Bridgeland (\cite{Bridgeland:K3}) in the case of K3 surfaces, generalized to arbitrary surfaces in \cite{Aaron-Daniele}.

For the rest of this section, we fix the lattice $\Lambda=\bZ^2\oplus\frac{1}{2}\bZ$ together with the surjective map $v:K_0(X)\to\Lambda$ defined by
\[
v(E)=(v_0(E),v_1(E),v_2(E))=(L^2\cdot\ch_0(E),L\cdot\ch_1(E),\ch_2(E))
\]

Given $\beta\in\bR$, consider the full subcategories
\[
\cF_{\beta}:=\{E\in\Coh(X)\mid\mu_L^+(E)\leq\beta\},\;\;
\cT_{\beta}:=\{E\in\Coh(X)\mid\mu_L^-(E)>\beta\} 
\]
of $\Coh(X)$, where $\mu_L^+$ (resp.~$\mu_L^-$) denotes the maximum (resp.~minimum) slope of a HN factor in $\mu_L$-stability.
They form a torsion pair, and thus according to \cite{HRS} their tilt
\[
\Coh^{\beta}(X):=\set{E\in\Db(X)\mid\cH^{-1}(E)\in\cF_{\beta}, \; \cH^{0}(E)\in\cT_{\beta},\; \cH^i(E)=0 \text{ for }i\neq0,-1}
\]
is the heart of a bounded t-structure on $\Db(X)$.


\begin{rem}\label{heart}
If an object $E\in\Db(X)$ satisfies $\cH^i(E)=0$ for $i\neq0,-1$, then $E\in\Coh^\beta(X)$ is equivalent to $\mu_L^+(\cH^{-1}(E))\leq\beta<\mu_L^-(\cH^0(E))$.
In particular, as a condition on $\beta\in\bR$, $E\in\Coh^\beta(X)$ is open on the right.
\end{rem}

For every $(\alpha,\beta)\in\bR_{>0}\times\bR$, we define a central charge $Z_{\alpha,\beta}:K_0(X)\to\bC$ by
\[
Z_{\alpha,\beta}(E)=-\left(\ch_2^{\beta}(E)-\frac{\alpha^2}{2}L^2\cdot\ch_0^{\beta}(E)
\right)+i\left(L\cdot\ch_1^{\beta}(E)\right)
\]
where $\ch^\beta=e^{-\beta L}\cdot\ch$ is the Chern character twisted by $\beta L$.
Note that $Z_{\alpha,\beta}$ factors through $v$.
We will denote by $\nu_{\alpha,\beta}$ the \emph{tilt slope} defined by the central charge $Z_{\alpha,\beta}$.

The main result of this part, for which we adopt the version in \cite[Theorems 6.10 and 6.13]{MS}, strongly relies on the classical Bogomolov inequality for $\mu_L$-semistable sheaves (see \cite[Theorem 1.3]{Langer} for positive characteristic):

\begin{thm}[{\cite{Bridgeland:K3,Aaron-Daniele}}]\label{alphabetaplane}
For every $(\alpha,\beta)\in\bR_{>0}\times\bR$, the pair $\sigma_{\alpha,\beta}=(\Coh^\beta(X),Z_{\alpha,\beta})$ is a Bridgeland stability condition on $\Db(X)$, satisfying the support property with respect to the quadratic form $\odisc:=(L\cdot\ch_1)^2-2(L^2\cdot\ch_0)\ch_2$.
\end{thm}

\begin{rem}
In the recent work \cite{Koseki} Koseki proved a version of Bogomolov inequality in positive characteristic, for surfaces of general type and quasi-elliptic surfaces with $\kappa=1$.
This enables him to construct (a smaller region of) Bridgeland stability conditions in such cases, that we will not consider.
\end{rem}

Given a class $v\in\Lambda$, a \emph{numerical wall} for $v$ is the region of $\bR_{>0}\times\bR$ determined by an equation of the form $\nu_{\alpha,\beta}(v)=\nu_{\alpha,\beta}(w)$, where $w\in\Lambda$ is a class non-proportional to $v$.
An \emph{actual wall} for $v$ is a subset of a numerical wall, at which the set of semistable objects of class $v$ changes.
If $\sigma_{\alpha,\beta}$ lies in an actual wall for $v$, then there is a short exact sequence of semistable objects $0 \to E \to F \to Q \to 0$ in $\Coh^\beta (X)$, with $v(F)=v$, $v(E)=w$ and $\nu_{\alpha,\beta}(E)=\nu_{\alpha,\beta}(F)$ (see \cite[Proposition 3.3]{BM:localP2} or \cite[Section 5.5]{MS}).
We say that the short exact sequence defines the corresponding actual wall.


The structure of the walls in this $(\alpha,\beta)$-plane is well understood.
Items \eqref{item:numerical}--\eqref{item:allactual} of the following theorem are called \emph{Bertram's Nested Wall Theorem}, and were proved in \cite{Maciocia}.
The last item is a consequence of \cite[Lemma~A.7]{BMS}, as part of a systematic study of the support property in terms of the quadratic form.

\begin{thm} \label{structurewalls}
Let $v\in\Lambda$ be a class with $\odisc(v)\geq0$.
\begin{enumerate}[{\rm (1)}]
 \item\label{item:numerical} All numerical walls for $v$ are either semicircles centered on the $\beta$-axis or lines parallel to the $\alpha$-axis.
 \item The numerical walls defined by classes $u,w\in\Lambda$ intersect if and only if $v,u,w$ are linearly dependent.
 In such a case, the two walls are identical.
 \item If $v_0\neq0$, there is a unique vertical wall with equation $\beta=\frac{v_1}{v_0}$.
 At each side of this vertical wall, all semicircular walls are strictly nested (see \autoref{fig:vneq0}).
 \item If $v_0=0$, there is no vertical wall and all numerical walls are strictly nested semicircles (see \autoref{fig:v=0}).
 \item The curve $H_v:\nu_{\alpha,\beta}(v)=0$ intersects every semicircular wall at its top point.
 This curve is an hyperbola (if $v_0\neq0$ and $\odisc(v)>0$), a pair of lines (if $v_0\neq0$ and $\odisc(v)=0$) or a single vertical line (if $v_0=0$).
 
 \item\label{item:allactual} If a numerical wall is an actual wall at some of its points, then it is an actual wall at all of its points.
 
 \item\label{item:Disc0nodest} If an actual wall is defined by a short exact sequence $0\to E\to F\to Q\to 0$ of semistable objects (with $v(F)=v$), then $\odisc(E)+\odisc(Q)<\odisc(F)$.
 In particular, if $\odisc(v)=0$ then semistable objects $F$ with $v(F)=v$ can only be destabilized at the vertical wall.
\end{enumerate}
\end{thm}

\begin{figure}[H]
\definecolor{ffffff}{rgb}{1,0,0}
\definecolor{sexdts}{rgb}{0.1803921568627451,0.49019607843137253,0.19607843137254902}
\definecolor{wrwrwr}{rgb}{0.3803921568627451,0.3803921568627451,0.3803921568627451}
\begin{tikzpicture}[line cap=round,line join=round,>=triangle 45,x=1.15cm,y=1.15cm]

\draw [line width=0.5pt,color=wrwrwr] (-3,0) -- (3,0);
\draw [line width=1pt,color=wrwrwr] (0,0) -- (0,1.8);

\draw [shift={(1.05,0)},line width=1pt,color=wrwrwr] plot[domain=0:3.141592653589793,variable=\t]({1*0.32*cos(\t r)+0*0.32*sin(\t r)},{0*0.32*cos(\t r)+1*0.32*sin(\t r)});
\draw [shift={(1.2,0)},line width=1pt,color=wrwrwr] plot[domain=0:3.141592653589793,variable=\t]({1*0.66*cos(\t r)+0*0.66*sin(\t r)},{0*0.66*cos(\t r)+1*0.66*sin(\t r)});
\draw [shift={(1.4,0)},line width=1pt,color=wrwrwr] plot[domain=0:3.141592653589793,variable=\t]({1*0.98*cos(\t r)+0*0.98*sin(\t r)},{0*0.98*cos(\t r)+1*0.98*sin(\t r)});
\draw [shift={(-1.05,0)},line width=1pt,color=wrwrwr] plot[domain=0:3.141592653589793,variable=\t]({1*0.32*cos(\t r)+0*0.32*sin(\t r)},{0*0.32*cos(\t r)+1*0.32*sin(\t r)});
\draw [shift={(-1.2,0)},line width=1pt,color=wrwrwr] plot[domain=0:3.141592653589793,variable=\t]({1*0.66*cos(\t r)+0*0.66*sin(\t r)},{0*0.66*cos(\t r)+1*0.66*sin(\t r)});
\draw [shift={(-1.4,0)},line width=1pt,color=wrwrwr] plot[domain=0:3.141592653589793,variable=\t]({1*0.98*cos(\t r)+0*0.98*sin(\t r)},{0*0.98*cos(\t r)+1*0.98*sin(\t r)});

\draw[line width=1pt,dash pattern=on 4pt off 3pt,color=sexdts, smooth,samples=100,domain=0:10] 
plot[domain=0:1.37,variable=\t]({-(-exp(\t)-exp(-\t))/2},{(exp(\t)-exp(-\t))/2});

\draw[line width=1pt,dash pattern=on 4pt off 3pt,color=sexdts, smooth,samples=100,domain=0:10] 
plot[domain=0:1.37,variable=\t]({(-exp(\t)-exp(-\t))/2},{(exp(\t)-exp(-\t))/2});

\draw [line width=0.5pt,color=wrwrwr] (4.3,0) -- (8.4,0);
\draw [line width=1pt,color=wrwrwr] (6.4,0) -- (6.4,1.8);
\draw[line width=1pt,dash pattern=on 4pt off 3pt,color=sexdts, smooth,samples=100,domain=0:10] (4.6,1.8) -- (6.4,0);
\draw[line width=1pt,dash pattern=on 4pt off 3pt,color=sexdts, smooth,samples=100,domain=0:10] (6.4,0) -- (8.2,1.8);

\draw [shift={(5.45,0)},line width=1pt,color=wrwrwr] plot[domain=0:3.141592653589793,variable=\t]({0.95*cos(\t r)},{0.95*sin(\t r)});
\draw [shift={(6.1,0)},line width=1pt,color=wrwrwr] plot[domain=0:3.141592653589793,variable=\t]({0.3*cos(\t r)},{0.3*sin(\t r)});
\draw [shift={(5.8,0)},line width=1pt,color=wrwrwr] plot[domain=0:3.141592653589793,variable=\t]({0.6*cos(\t r)},{0.6*sin(\t r)});

\draw [shift={(7.2,0)},line width=1pt,color=wrwrwr] plot[domain=0:3.141592653589793,variable=\t]({0.8*cos(\t r)},{0.8*sin(\t r)});

\draw [shift={(6.9,0)},line width=1pt,color=wrwrwr] plot[domain=0:3.141592653589793,variable=\t]({0.5*cos(\t r)},{0.5*sin(\t r)});

\begin{footnotesize}
\draw[color=sexdts] (-2.32,1.82) node {$\hyp{v}$};
\draw[color=sexdts] (2.32,1.82) node {$\hyp{v}$};
\draw[color=wrwrwr] (0,2.1) node {$\beta=\frac{v_1}{v_0}$};

\draw[color=sexdts] (4.32,1.82) node {$\hyp{v}$};
\draw[color=sexdts] (8.4,1.8) node {$\hyp{v}$};
\draw[color=wrwrwr] (6.4,2.1) node {$\beta=\frac{v_1}{v_0}$};

\end{footnotesize}
\end{tikzpicture}

\caption{Numerical walls for $v$ when $v_0\neq0$}\label{fig:vneq0}

\vspace{.3cm}
\definecolor{ffffff}{rgb}{1,0,0}
\definecolor{sexdts}{rgb}{0.1803921568627451,0.49019607843137253,0.19607843137254902}
\definecolor{wrwrwr}{rgb}{0.3803921568627451,0.3803921568627451,0.3803921568627451}
\begin{tikzpicture}[line cap=round,line join=round,>=triangle 45,x=1.15cm,y=1.15cm]

\draw [line width=0.5pt,color=wrwrwr] (-2.5,0) -- (2.5,0);
\draw [line width=1pt,dash pattern=on 4pt off 3pt,color=sexdts, smooth,samples=100,domain=0:10] (0,0) -- (0,1.8);

\draw [line width=1pt,color=wrwrwr] plot[domain=0:3.141592653589793,variable=\t]({1*0.98*cos(\t r)},{1*0.98*sin(\t r)});

\draw [line width=1pt,color=wrwrwr] plot[domain=0:3.141592653589793,variable=\t]({1.3*cos(\t r)},{1.3*sin(\t r)});

\draw [line width=1pt,color=wrwrwr] plot[domain=0:3.141592653589793,variable=\t]({0.5*cos(\t r)},{0.5*sin(\t r)});

\begin{footnotesize}
\draw[color=sexdts] (0,2.1) node {$\hyp{v}:\beta=\frac{v_2}{v_1}$};
\end{footnotesize}
\end{tikzpicture}

\caption{Numerical walls for $v$ when $v_0=0$}\label{fig:v=0}
\end{figure}

\begin{rem}\label{hyperbola}
If $F\in\Coh^\beta(X)$ for some $\beta\in\bR$, then $L\cdot\ch_1^\beta(F)\geq0$ gives $\beta\leq \mu_L(F)$ (resp.~$\beta\geq\mu_L(F)$) if $\ch_0(F)>0$ (resp.~$\ch_0(F)<0$).
There is no condition on $\beta$, if $\ch_0(F)=0$.

Thus in the study of a particular object $F$, one is led to consider just one of the regions separated by the vertical wall (if exists).
Such a region is divided into two parts by a component of $H_{v(F)}$;
abusing of notation, this component will be called \emph{hyperbola of $F$} and denoted by $H_F$.
Note that at the left-hand (resp.~right-hand) side of $H_F$, $F$ has positive (resp.~negative) tilt slope.
The $\beta$-coordinate of the intersection point of $H_F$ with $\alpha=0$ will be denoted by $p_F$.
\end{rem}

The next result, originally due to Bridgeland, motivates the name of \emph{Gieseker chamber} for the chamber above the largest wall, in the case of a class with positive rank. The reader is referred to \cite{MW} or \cite[Definition~14.1]{Bridgeland:K3} for the definition of twisted Gieseker semistability.

\begin{prop}[{\cite{Bridgeland:K3}, \cite[Exercise 6.27]{MS}}]\label{Giesekerchamber}
Let $v\in\Lambda$ be a class with $\odisc(v)\geq0$ and $v_0>0$, and let $\beta<\frac{v_1}{v_0}$.
Then an object $F\in\Coh^\beta(X)$ of class $v(F)=v$ is $\sigma_{\alpha,\beta}$-semistable for every $\alpha\gg0$ if, and only if, $F$ is a twisted $(L,-\frac{1}{2}K_X)$-Gieseker semistable sheaf.
\end{prop}

When $\alpha=0$ one cannot ensure in general that this construction gives Bridgeland stability conditions.
This is due to the fact that for $\beta\in\bQ$ the central charge $Z_{0,\beta}$ may send to $0$ certain objects of $\Coh^\beta(X)$, as we describe in the following proposition:

\begin{prop}\label{Giesekerkernel}
If $\beta\in\bQ$, then an object $F\in\Coh^\beta(X)$ satisfies:
\begin{enumerate}[{\rm (1)}]
    \item\label{item:rk0} $L\cdot\ch_1^\beta(F)=0$ if and only if the following hold:
    \begin{itemize}[\textbullet]
        \item $\cH^{-1}(F)$ is either $0$ or a $\mu_L$-semistable torsion-free sheaf with $\mu_L=\beta$.
        \item $\cH^0(F)$ is either $0$ or a sheaf supported in dimension $0$.
    \end{itemize}
    \item\label{item:KernelZ} $Z_{0,\beta}(F)=0$ if and only if $F=S[1]$, where $S$ is a twisted $(L,-\frac{1}{2}K_X)$-Gieseker semistable vector bundle with $\mu_L(S)=\beta$ and $\odisc(S)=0$.
\end{enumerate} 
\end{prop}
\begin{proof}
The first item is an easy consequence of $\cH^{-1}(F)\in\cF_\beta$ and $\cH^0(F)\in\cT_\beta$.
We explain how to check the ``only if'' part of \eqref{item:KernelZ}, the converse implication being immediate.

Note that $Z_{0,\beta}(F)=0$ implies $L\cdot\ch_1^\beta(F)=0$, so $\cH^{-1}(F)$ and $\cH^0(F)$ must be as stated in \eqref{item:rk0}.
If it were $\cH^0(F)\neq0$, then $Z_{0,\beta}(\cH^0(F))\in\bR_{<0}$; 
this would force $Z_{0,\beta}(\cH^{-1}(F)[1])\in\bR_{>0}$, which contradicts that $Z_{0,\beta}$ is a weak stability function on $\Coh^\beta(X)$.

Therefore $F=\cH^{-1}(F)[1]$.
If $\cH^{-1}(F)$ were not locally free, then $\cH^{-1}(F)^{**}/\cH^{-1}(F)$ would be a nontrivial subobject of $F$ (in $\Coh^\beta(X)$) satisfying $Z_{0,\beta}(\cH^{-1}(F)^{**}/\cH^{-1}(F))\in\bR_{<0}$, a contradiction.
The maximal destabilizing subobject of $\cH^{-1}(F)$ with respect to twisted $(L,-\frac{1}{2}K_X)$-Gieseker stability, if different from $\cH^{-1}(F)$, would give a similar contradiction.

Finally, $\odisc(\cH^{-1}(F))=0$ immediately follows from $\ch_2^{\beta}(\cH^{-1}(F))=0=L\cdot\ch_1^{\beta}(\cH^{-1}(F))$.
\end{proof}

\begin{ex}\label{exGiesekerkernel}
For $\beta\in\bQ$, the fact that $\sigma_{0,\beta}$ is a Bridgeland stability condition depends on the sharpness of the Bogomolov inequality for $\mu_L$-semistable sheaves of slope $\beta$, which is encoded by the Le Potier function
\[
\Phi_{X,L}(\beta):=\sup \set{\frac{\ch_2(E)}{L^2\cdot\ch_0(E)} \mid E \text{ is $\mu_L$-semistable with }\mu_L(E)=\beta},
\]
see \cite[Definition~3.1]{FLZ} for its extension as a real function.
This function depends on the particular geometry of the surface:
\begin{enumerate}[{\rm (1)}]
    \item If $X$ is a complex abelian surface, the vector bundles $S$ of \autoref{Giesekerkernel} are semihomogeneous.
    This is a consequence of \cite[Theorem~IV.4.7]{kobayashi} and \cite[Theorem~5.12]{yang}.
    
    \noindent Conversely, for $X$ an arbitrary abelian surface and 
    for every $\beta\in\bQ$ there exist semihomogeneous vector bundles $S$ with $\frac{\ch_1(S)}{\rk(S)}=\beta L\in\NS(X)_\bQ$ by \cite{Muk}; for such bundles $S$, one has $S[1]\in\Coh^\beta(X)$ and $Z_{0,\beta}(S[1])=0$.
    Therefore, $\sigma_{0,\beta}$ fails to be a Bridgeland stability condition for every $\beta\in\bQ$.
    
    \item More generally, let $X$ be a surface whose Albanese map $a:X\to A$ is finite onto its image, endowed with a polarization $L=a^*\widetilde{L}$ pulled back from a polarization $\widetilde{L}$ on $A$.
    Then, for every $\beta\in\bQ$ one can find objects $F\in\Coh^\beta(X)$ with $Z_{0,\beta}(F)=0$, for instance $F=a^*S[1]$ for simple semihomogeneous vector bundles $S$ with $\frac{\ch_1(S)}{\rk(S)}=\beta\widetilde{L}\in\NS(A)_\bQ$.
    
    \noindent Let us check that $a^*S[1]\in\Coh^\beta(X)$, the equality $Z_{0,\beta}(a^*S[1])=0$ being straightforward.
    Since $\mu_L(a^*S)=\beta$, we need to prove that $a^*S$ is $\mu_L$-semistable; by \cite[Lemma~3.2.2]{HuybrechtsLehn}, it suffices to check that $S|_{a(X)}$ is slope semistable with respect to $\widetilde{L}|_{a(X)}$.
    
    \noindent And indeed, according to \cite[Proposition~7.3]{Muk}, there exists an isogeny $\pi:B\to A$ and a line bundle $M$ on $B$ with $\pi^*S=M^{\oplus \rk S}$.
    This description implies that $\pi^*S|_{\pi^{-1}(a(X))}$ is slope semistable with respect to $\pi^*\widetilde{L}|_{\pi^{-1}(a(X))}$, and therefore the semistability of $S|_{a(X)}$ again follows from \cite[Lemma~3.2.2]{HuybrechtsLehn}.
    
    \noindent In particular, this shows that $\Phi_{X,L}(x)=\frac{x^2}{2}$ for every $x\in\bR$ adding an instance to \cite[Remark~3.3]{FLZ}.
    
    \item If $X$ is a simply connected complex surface and $Z_{0,\beta}(F)=0$ for $F\in\Coh^\beta(X)$, then $\beta\in\bZ$ and the vector bundle $S$ of \autoref{Giesekerkernel} satisfies $S=(L^\beta)^{\oplus \rk S}$; this follows from \cite[Theorem~IV.4.7 and Corollary~I.2.7]{kobayashi}.
    
    \noindent In terms of Le Potier function, this means that $\Phi_{X,L}(x)=\frac{x^2}{2}$ if and only if $x\in\bZ$.
\end{enumerate}

\end{ex}

In any case, $Z_{0,\beta}$ is a (possibly weak) stability function on $\Coh^\beta(X)$ for every $\beta\in\bQ$; since $\Coh^\beta(X)$ is Noetherian, 
the existence of HN filtrations with respect to the tilt slope $\nu_{0,\beta}$ is guaranteed (see \cite[Lemma~2.4]{Bridgeland:Stab} and \cite[Proposition~B.2]{BM:localP2}).
Therefore, $\sigma_{0,\beta}=(\Coh^\beta(X),Z_{0,\beta})$ is a (possibly weak) stability condition on $\Db(X)$\footnote{If $Z_{0,\beta}$ is a stability function, the support property for the Bridgeland stability condition $\sigma_{0,\beta}$ holds as explained in \cite[Remark~3.5]{FLZ}; otherwise, the quadratic form $Q=0$ guarantees the support property for the weak stability condition $\sigma_{0,\beta}$}.


We finish this section describing the behaviour of semistable objects under the derived dual.
This is certainly well-known to the experts but we include it for easy reference.
For an object $E\in\Db(X)$, we write $E^\vee=R\sHom(E,\cO_X)$.

\begin{prop}\label{dualstab}
For $\beta\in\bR$, let $F\in\Coh^\beta(X)$ be an object such that $\nu_{\alpha,\beta}^+(E)<+\infty$ for every $\alpha\geq0$ (i.e.~$F$ contains no subobject with $L\cdot\ch_1^\beta=0$).
Then:
\begin{enumerate}[{\rm (1)}]
 \item $F^\vee[1]\in\Coh^{-\beta}(X)$.
 \item\label{item:dual2} For every $\alpha\geq0$, $F$ is $\sigma_{\alpha,\beta}$-(semi)stable if and only if $F^\vee[1]$ is $\sigma_{\alpha,-\beta}$-(semi)stable.
\end{enumerate}
\end{prop}
\begin{proof}
The first item is the particular case for surfaces of the more general result \cite[Lemma~2.19.a]{BLMS}.
As a consequence of it, the contravariant functor $\blank^\vee[1]$ induces a bijection between subobjects of $F$ (in $\Coh^\beta(X)$) and quotients of $F^\vee[1]$ (in $\Coh^{-\beta}(X)$).
Taking into account the Chern character of the derived dual as well, item \eqref{item:dual2} is immediately checked.
\end{proof}

\begin{rem}
A combination of \autoref{Giesekerchamber} and \autoref{dualstab} describes the semistable objects in the large volume limit for classes with negative rank.
\end{rem}

\section{Chern degree functions}
\label{sec:Chern_deg}

In the next three sections, $(X,L)$ will be a fixed polarized smooth projective surface over $\bK$; in positive characteristic, we assume $X$ is neither of general type nor quasi-elliptic with $\kappa(X)=1$.
We present now the main objects of this article.

\begin{defn}\label{maindef}
Let $\beta\in\bQ$ be a rational number.
\begin{enumerate}[{\rm (1)}]
 \item \label{item:maindef1} If $F\in\Coh^\beta(X)$ is an object with HN filtration
 $0=F_0\hookrightarrow F_1\hookrightarrow\ldots\hookrightarrow F_r=F$
 with respect to $\sigma_{0,\beta}$, we define 
 \begin{align*}
 \rf^0_{F,L}(-\beta)&:=\displaystyle\sum_{\nu_{0,\beta}({F_i/F_{i-1}})\geq0}\ch_2^{\beta}({F_i/F_{i-1}})\\
 \rf^1_{F,L}(-\beta)&:=\displaystyle\sum_{\nu_{0,\beta}({F_i/F_{i-1}})<0}-\ch_2^{\beta}({F_i/F_{i-1}})
 \end{align*}
 \item \label{item:maindef2} More generally, for an arbitrary object $E\in\Db(X)$ and any integer $k\in\bZ$, we define the number $\rf^k_{E,L}(-\beta)$ using the cohomologies of $E$ with respect to $\Coh^\beta(X)$:
 \[
 \rf^k_{E,L}(-\beta):=\rf^0_{\cH_\beta^k(E),L}(-\beta)+\rf^1_{\cH_\beta^{k-1}(E),L}(-\beta)
 \]
\end{enumerate}
\end{defn}

This rule defines, given $E\in\Db(X)$ and $k\in\bZ$, a function $\rf^k_{E,L}:\bQ\to\bQ_{\geq0}$ that we will call the \emph{$k$-th Chern degree function of $E$}.
From the definition, it directly follows that
\[\displaystyle\sum_{k\in\bZ}(-1)^k\cdot \rf^k_{E,L}(x)=\ch_2^{-x}(E).\]
so the Chern degree functions are a positive alternate decomposition of the second twisted Chern character.

\begin{rem}
Thinking of $\sigma_{0,\beta}$ in terms of slicings, the definition of $\rf^k_{E,L}(-\beta)$ involves all the objects in the HN filtration of $E\in\Db(X)$ with phase in the interval $\left[\frac{1}{2}-k,\frac{3}{2}-k\right)$.
\end{rem}

\begin{ex}
To determine the Chern degree functions of $\cO_X$, observe that for every $\beta<0$ $\cO_X\in\Coh^\beta(X)$; moreover, combining \autoref{Giesekerchamber} and \autoref{structurewalls}.\eqref{item:Disc0nodest} we have that $\cO_X$ is $\sigma_{\alpha,\beta}$-semistable for every $\alpha>0$, hence for every $\alpha\geq0$.
A similar situation holds for $\cO_X[1]$ when $\beta\geq0$, so the nonzero Chern degree functions of $\cO_X$ are
\[
\rf^0_{\cO_X,L}(x)=\left\{
 \begin{array}{c l}
 0 & x\leq 0\\
 \frac{L^2}{2}x^2 & x\geq 0\\
 \end{array}
 \right.
\qquad\text{ and }\qquad \rf^2_{\cO_X,L}(x)=\left\{
 \begin{array}{c l}
 \frac{L^2}{2}x^2 & x\leq 0\\
 0 & x\geq 0.\\
 \end{array}
 \right.
\]
\end{ex}

\vspace{1mm}

Given $F\in\Coh^{\beta}(X)$, we can also define $\rf^0_{F,L}(-\beta)$ in terms of a unique subobject.
Namely, if 
\[
0=F_0\hookrightarrow F_1 \hookrightarrow \ldots \hookrightarrow F_{s-1} \hookrightarrow F_s \hookrightarrow F_{s+1} \hookrightarrow \ldots \hookrightarrow F_r=F
\]
is the HN filtration of $F$ with respect to $\sigma_{0,\beta_0}$, with the inequalities
\[
\nu_{0,\beta_0}(F_1)>\ldots>\nu_{0,\beta_0}(F_s/F_{s-1})>0\geq\nu_{0,\beta_0}(F_{s+1}/F_s)>\ldots>\nu_{0,\beta_0}(F/F_{r-1}),
\]
then observe that $\rf^0_{F,L}(-\beta)=\ch_2^\beta(F_s)$.
\begin{defn}
The index $s=s(F)$ is called the \emph{switching index of $F$}.
\end{defn}

Note that $\nu_{0,\beta}^-(F)>0$ is equivalent to $F_s=F$, and $\nu_{0,\beta}^+(F)\leq0$ is equivalent to $F_s=0$.

These functions satisfy the following properties, analogous to those that are natural from the viewpoint of cohomology:

\begin{prop}\label{rk:BasicProp}
If $E\in\Db(X)$ and $x\in\bQ$, the following properties hold:
\begin{enumerate}[{\rm (1)}]
\item\label{enum:SerreVan} (Serre vanishing) If $E\in\Coh(X)$, then for $x\gg0$ one has $\rf^k_{E,L}(x)=0$ for every $k\neq0$.
 \item \label{enum:SerreDua}(Serre duality) We have $\rf^k_{E,L}(x)=\rf^{2-k}_{E^\vee,L}(-x)$ for every $k\in\bZ$.
\end{enumerate}
\end{prop}
\begin{proof}
To prove the first item we note that any sheaf $E$ satisfies $E\in\Coh^\beta(X)$ for every $\beta<\mu_-(E)$;
in particular, $\rf^k_{E,L}(x)=0$ for every $x>-\mu_-(E)$ and $k\neq0,1$.
Thus it only remains to show vanishing for $\rf^1_{E,L}$.
 
To this end, we consider the HN filtration of $E$ with respect to twisted $(L,-\frac{1}{2}K_X)$-Gieseker stability (if $E$ is not itself torsion-free, we create this filtration using the torsion filtration).
By \autoref{Giesekerchamber}, this is the HN filtration of $E\in\Coh^\beta(X)$ with respect to $\nu_{0,\beta}$-stability, for $\beta\ll0$.
Therefore, for $\beta\ll0$ all the HN factors of $E$ have positive slope $\nu_{0,\beta}$, which proves that $\rf^1_{E,L}(x)=0$ for $x\gg0$.
 
For the second item, we will check the equality assuming that $E\in\Coh^{-x}(X)$; the general statement follows from considering the cohomologies of $E$ with respect to the heart $\Coh^{-x}(X)$.
 
For simplicity, we write $\beta=-x$.
Let
\[
0=E_0\hookrightarrow E_1 \hookrightarrow \ldots \hookrightarrow E_{r-1} \hookrightarrow E_r=E
\]
be the HN filtration of $E$ with respect to $\nu_{0,\beta}$, where the first HN factor $E_1$ is assumed to have slope $\nu_{0,\beta}=+\infty$ (if this does not occur, simply write $E_1=0$ and let the rest of the filtration be that of $E$).
By definition:
\[
 \rf^0_{E,L}(-\beta)=\ch_2^\beta(E_1)+\rf^0_{E/E_1,L}(-\beta), \;\;\;\;\rf^1_{E,L}(-\beta)=\rf^1_{E/E_1,L}(-\beta)
\]
Moreover, we have a triangle $0\to (E/E_1)^\vee[1]\to E^\vee[1]\to E_1^\vee[1]\to0$ in $\Db(X)$, where:
 
\begin{itemize}[\textbullet]
  \item $(E/E_1)^\vee[1]\in\Coh^{-\beta}(X)$, having $(E/E_{r-1})^\vee[1]$,\ldots,$(E_2/E_1)^\vee[1]$ as HN factors with respect to $\nu_{0,-\beta}$.
  This is a consequence of \autoref{dualstab}.
  
  \item It is not difficult to check that $E_1^\vee[2]\in\Coh^{-\beta}(X)$ and it is $\nu_{0,-\beta}$-semistable (with slope $+\infty$).
\end{itemize}
 
It follows that $E^\vee[1]$ has two cohomologies with respect to $\Coh^{-\beta}(X)$, namely
\[
 \cH^0_{-\beta}(E^\vee[1])=(E/E_1)^\vee[1], \;\;\;\cH^1_{-\beta}(E^\vee[1])=E_1^\vee[2]
\]
 whose HN factors are known in terms of those of $E$.
 This gives the desired relations
 \[
 \rf^0_{E,L}(-\beta)=\rf^1_{E^\vee[1],L}(\beta), \;\;\;\rf^1_{E,L}(-\beta)=\rf^0_{E^\vee[1],L}(\beta).\qedhere
 \]
\end{proof}

The following technical lemma, which will be useful later on, is also natural from the same cohomological viewpoint:

\begin{lem}\label{sesequence}
Let $0\to E\to F\to Q\to 0$ be a short exact sequence in $\Coh^\beta(X)$ ($\beta\in\bQ$).
If $\rf^0_{Q,L}(-\beta)=0$ and $Q$ has no subobject $\widetilde{Q}\subset Q$ in $\Coh^\beta(X)$ with $\widetilde{Q}\in\ker(Z_{0,\beta})$, then the equality $\rf^0_{E,L}(-\beta)=\rf^0_{F,L}(-\beta)$ holds.
\end{lem}
\begin{proof}
Let $0=F_0\hookrightarrow F_1\hookrightarrow\ldots\hookrightarrow F_s\hookrightarrow F_{s+1}\hookrightarrow\ldots\hookrightarrow F$ be the HN filtration of $F$ with respect to $\sigma_{0,\beta}$, so that $\nu_{0,\beta}(F_s/F_{s-1})>0\geq\nu_{0,\beta}(F_{s+1}/F_s)$.
That is, $\rf^0_{F,L}(-\beta)=\ch_2^\beta(F_s)$.

By $\rf^0_{Q,L}(-\beta)=0$ and our extra assumption on $\ker(Z_{0,\beta})$, we know that every subobject of $Q$ has $\nu_{0,\beta_0}\leq0$.
This implies that the morphism $F_s\to F\to Q$ must be $0$, and thus $F_s\subset E$.

It turns out that $F_s\subset E$  is the part of the HN filtration of $E$ corresponding to HN factors of positive slope.
Therefore, $\rf^0_{E,L}(-\beta)=\ch_2^\beta(F_s)=\rf^0_{F,L}(-\beta)$.
\end{proof}

\begin{rem}\label{variant}
In view of the results in \cite[Section~3]{FLZ} extending the construction of (geometric) Bridgeland stability conditions to a region enlarging the $(\alpha,\beta)$-plane, it would be interesting to consider functions defined via weak stability conditions on the ``boundary'' of this bigger region.
More precisely, Fu--Li--Zhao construct a Bridgeland stability condition
\[
\widetilde{\sigma}_{a,\beta}=\left(\Coh^\beta(X),\widetilde{Z}_{a,\beta}=(-\ch_2^\beta+aL^2\cdot\ch_0)+i(L\cdot\ch_1^\beta)\right)
\]
for every $(a,\beta)\in\bR^2$ with $a>\Phi_{X,L}(\beta)-\frac{\beta^2}{2}$, where $\Phi_{X,L}(\beta)$ is the Le Potier function (see \cite[Definition~3.1]{FLZ} and \autoref{exGiesekerkernel}).
In particular, the Bridgeland stability conditions in the $(\alpha,\beta)$-plane are recovered as $\widetilde{\sigma}_{\frac{\alpha^2}{2},\beta}=\sigma_{\alpha,\beta}$ for every $\alpha>0$.

The function $f(x):=\Phi_{X,L}(x)-\frac{x^2}{2}$ being upper-semicontinuous, its discontinuities form a meagre set.
Since the complement of a meagre set is dense thanks to the Baire category theorem, it turns out that the points where $f$ is continuous form a dense subset $A_{X,L}$ of $\bR$.

Henceforth, one could define functions on $A_{X,L}\cap\bQ$ via the HN filtrations with respect to the weak stability conditions $\set{\widetilde{\sigma}_{f(\beta),\beta}\mid \beta\in A_{X,L}\cap\bQ}$.
Clearly, in the cases where $f\equiv0$ (e.g. surfaces with finite Albanese map, as seen in \autoref{exGiesekerkernel}) this is nothing but our Chern degree functions.
In general, to extend these functions to the whole $\bR$, one could try to follow the same approach of \autoref{sec:LocalExpr};
however, while we expect that the existence of \emph{Bridgeland limit filtrations} (i.e.~ \autoref{existenceIntr}.\eqref{exist:Br}) could be proven following similar arguments, the existence of \emph{weak limit filtrations} (i.e.~ \autoref{existenceIntr}.\eqref{exist:Wk}) seems a much more obscure problem.

Consider for instance a polarized K3 surface $(X,L)$ with $\Pic(X)=\bZ\cdot L$ and $L^2=2e$.
In that case, the existence of spherical objects shows that 
\[
A_{X,L}\cap\bQ=\bQ\setminus\set{\frac{c}{r}\in\bQ:\;r|e(c^2+1)}
\]
and $f\equiv-\frac{1}{2e}$ on this subset.
For $\beta\in A_{X,L}\cap\bQ$, one can consider the central charge of $\widetilde{\sigma}_{-\frac{1}{2e},\beta}$ given by 
$
\widetilde{Z}_{-\frac{1}{2e},\beta}=-v_2^\beta+i(L\cdot v_1^\beta)
$
for $(v_0^\beta,v_1^\beta,v_2^\beta)=v\cdot e^{-\beta L}$ the twisted Mukai vector.
Accordingly, if $F\in\Coh^\beta(X)$ is an object with HN filtration $0=F_0\hookrightarrow F_1\hookrightarrow\ldots\hookrightarrow F_r=F$ with respect to $\sigma_{-\frac{1}{2e},\beta}$, we may define for example
\begin{align*}
\operatorname{vdeg}^0_{F,L}(-\beta)&:=\displaystyle\sum_{\nu_{-\frac{1}{2e},\beta}({F_i/F_{i-1}})\geq0}v_2^{\beta}({F_i/F_{i-1}}).
\end{align*}
In this particular case we could call such functions \emph{Mukai degree functions}.

Observe that in general one cannot expect to extend the Mukai degree functions to continuous functions in the whole $\bR$ mimicking \autoref{main}.\eqref{item:main2}, since discontinuities may arise at certain points of $\bQ\setminus A_{X,L}$, as one easily sees with the function of $\cO_X$.

In any case, it would be interesting to know if, fixed an object of $\Db(X)$, there are finitely many such discontinuities and the Mukai degree functions encode information of geometrical or cohomological type.
\end{rem}


\section{Local expressions for the Chern degree functions}
\label{sec:LocalExpr}

This section is devoted to prove that, in a neighborhood of every rational number, the Chern degree functions are piecewise polynomial. 
The result is analoguous to \cite[Corollary~2.6]{JP}, where it follows for cohomological rank functions from a transformation formula with respect to the Fourier-Mukai transform.
In our case, the proof follows a completely different path, by describing the behavior of HN filtrations around weak stability conditions $\sigma_{0,\beta_0}$ ($\beta_0\in\bQ$).
This description may be of independent interest, especially when $\sigma_{0,\beta_0}$ lies in the boundary of the (geometric) stability manifold (see \autoref{exGiesekerkernel} and \autoref{variant}).

Along this section, we keep fixed a rational number $\beta_0=\frac{a}{b}$ with $a$ and $b$ coprime integers ($b>0$).

\subsection{Bridgeland limit HN filtrations}\label{subsec:Bridgelandlimit}
Our first goal is to control HN filtrations with respect to the Bridgeland stability conditions $\sigma_{\alpha,\beta_0}$, for small values of $\alpha>0$.
In particular, we want to understand whether these filtrations remain constant:

\begin{defn}\label{BridLimitHN}
Given $F\in \Coh^{\beta_0}(X)$, if there exists $\alpha_0>0$ such that $F$ has the same HN filtration with respect to all the Bridgeland stability conditions $\sigma_{\alpha,\beta_0}$ with $\alpha\in(0,\alpha_0)$, we will call this HN filtration the \emph{Bridgeland limit HN filtration} of $F$ at $\beta_0$.
\end{defn}

In case the Bridgeland limit HN filtration exists, the HN filtration of $F$ at $\sigma_{0,\beta_0}$ can be recovered by identifying those limit HN factors with the same slope at $\alpha=0$.

The first result of this section is that Bridgeland limit HN filtrations exist for $\sigma_{0,\beta_0}$-semistable objects with nonzero tilt slope:

\begin{prop}\label{HNalpha2}
Any $\sigma_{0,\beta_0}$-semistable object $F\in\Coh^{\beta_0}(X)$ with $\nu_{0,\beta_0}(F)\neq0$ admits a Bridgeland limit HN filtration at $\beta_0$.
\end{prop}

The case $L\cdot\ch_1^{\beta_0}(F)=0$ being trivial (in this case $F$ is semistable along the whole line $\beta=\beta_0$), we will assume that $L\cdot\ch_1^{\beta_0}(F)>0$ (i.e.~ $\nu_{0,\beta_0}(F)\neq+\infty$). 
Thanks to the duality functor $\blank^\vee[1]$ and \autoref{dualstab}, we may restrict ourselves to the case $\nu_{0,\beta_0}(F)<0$.

The proof is then based on two lemmas:

\begin{lem}\label{SF}
Let $F\in\Coh^{\beta_0}(X)$ be $\sigma_{0,\beta_0}$-semistable, with $L\cdot\ch_1^{\beta_0}(F)>0$ and $\nu_{0,\beta_0}(F)<0$.
\begin{enumerate}[{\rm (1)}]
 \item\label{enum:SFNonEmpty} The set of subobjects
 \begin{equation*}\label{defn:SF}
 S_F=\set{E\in\Coh^{\beta_0}(X):E\subseteq F,\;\nu_{0,\beta_0}(E)=\nu_{0,\beta_0}(F),\;\odisc(E)\geq0}
 \end{equation*}
 with the same slope and non-negative discriminant is non-empty.
 
 \item\label{enum:BoundedBelow} The expression $\frac{L^2\cdot\ch_0}{L\cdot\ch_1^{\beta_0}}$ is bounded from below on $S_F$.
\end{enumerate}
\end{lem}

Before proving this first lemma, we note the following:

\begin{rem}\label{v0discrete}
By the assumption $\beta_0=\frac{a}{b}$, every $F\in\Coh^{\beta_0}(X)$ with $\nu_{0,\beta_0}^+(F)<+\infty$ satisfies
\[
\set{L\cdot\ch_1^{\beta_0}(G):G\subset F\text{ in }\Coh^{\beta_0}(X)}\subseteq \frac{1}{b}\cdot\bZ_{>0}, \;\;
 \set{\nu_{0,\beta_0}(G):G\subset F\text{ in }\Coh^{\beta_0}(X)}\subseteq \frac{1}{k_F}\cdot\bZ
\]
where $k_F=2b\left(bL\cdot\ch_1^{\beta_0}(F)\right)!\in\bZ_{>0}$ (it depends only on $F$).
Indeed, since $2b^2\ch_2^{\beta_0}(G)\in \bZ$ and $bL\cdot\ch_1^{\beta_0}(G)\in (0,bL\cdot\ch_1^{\beta_0}(F)]\cap\bZ$, then
$\nu_{0,\beta_0}(G)=\frac{2b^2\ch_2^{\beta_0}(G)}{2b\left(bL\cdot\ch_1^{\beta_0}(G)\right)}\in \frac{1}{k_F}\cdot\bZ$.
\end{rem}

\begin{proof}[Proof of \autoref{SF}]
To prove \eqref{enum:SFNonEmpty}, assume that the set $S_F$ is empty; in particular $\odisc(F)<0$, so $F$ is nonsemistable for every Bridgeland stability condition $\sigma_{\alpha,\beta_0}$ with $\alpha>0$.
If $n\in\bZ_{>0}$, let $G_n$ be the maximal destabilizing subobject of $F$ with respect to $\sigma_{\frac{1}{n},\beta_0}$.

Observe that $\nu_{0,\beta_0}(G_n)\leq\nu_{0,\beta_0}(F)$, by the $\sigma_{0,\beta_0}$-semistability of $F$; the emptiness of $S_F$ guarantees a strict inequality.
Therefore, $\nu_{0,\beta_0}(G_n)\leq\nu_{0,\beta_0}(F)-\frac{1}{k_F}$ (by \autoref{v0discrete}).
Combining with $\nu_{\frac{1}{n},\beta_0}(G_n)>\nu_{\frac{1}{n},\beta_0}(F)$, we get
\[
\nu_{0,\beta_0}(F)-\frac{1}{2n^2}\cdot\frac{L^2\cdot\ch_0(F)}{L\cdot\ch_1^{\beta_0}(F)}=\nu_{\frac{1}{n},\beta_0}(F)<\nu_{\frac{1}{n},\beta_0}(G_n)\leq\nu_{0,\beta_0}(F)-\frac{1}{k_F}-\frac{1}{2n^2}\cdot\frac{L^2\cdot\ch_0(G_n)}{L\cdot\ch_1^{\beta_0}(G_n)}
\]
which gives
\[
-\frac{L^2\cdot\ch_0(G_n)}{L\cdot\ch_1^{\beta_0}(G_n)}>\frac{2}{k_F}n^2-\frac{L^2\cdot\ch_0(F)}{L\cdot\ch_1^{\beta_0}(F)} \;\;\Longrightarrow \;\; -\frac{L^2\cdot\ch_0(G_n)}{L\cdot\ch_1^{\beta_0}(G_n)}\to+\infty \text{ as }n\to\infty
\]

But on the other hand, for every $n$ we have $L\cdot\ch_1^{\beta_0}(G_n)>0$ and hence
\begin{align*}
0\leq\odisc(G_n)
&=\left(L\cdot\ch_1^{\beta_0}(G_n)\right)^2\left(1-2\frac{L^2\cdot\ch_0(G_n)}{L\cdot\ch_1^{\beta_0}(G_n)}\cdot\nu_{0,\beta_0}(G_n)\right),
\end{align*}
which yields
\[
0\leq1-2\frac{L^2\cdot\ch_0(G_n)}{L\cdot\ch_1^{\beta_0}(G_n)}\cdot\nu_{0,\beta_0}(G_n)<1-2\frac{L^2\cdot\ch_0(G_n)}{L\cdot\ch_1^{\beta_0}(G_n)}\cdot\nu_{0,\beta_0}(F)
\]
for every $n$ such that $-\frac{L^2\cdot\ch_0(G_n)}{L\cdot\ch_1^{\beta_0}(G_n)}>0$.
Since $\nu_{0,\beta_0}(F)<0$, this contradicts the limit above and concludes the proof of \eqref{enum:SFNonEmpty}.

To prove \eqref{enum:BoundedBelow}, let $E\subseteq F$ be a subobject with $\nu_{0,\beta_0}(E)=\nu_{0,\beta_0}(F)$ and $\odisc(E)\geq0$.
Then
\[
0\leq\odisc(E)=\left(L\cdot\ch_1^{\beta_0}(E)\right)^2\left(1-2\frac{L^2\cdot\ch_0(E)}{L\cdot\ch_1^{\beta_0}(E)}\cdot\nu_{0,\beta_0}(F)\right)
\]
implies, under the assumption $\nu_{0,\beta_0}(F)<0$, that
\[
\frac{L^2\cdot\ch_0(E)}{L\cdot\ch_1^{\beta_0}(E)}\geq\frac{1}{2\nu_{0,\beta_0}(F)}\qedhere
\]
\end{proof}


The elements $E\in S_F$ satisfy $\frac{L^2\cdot\ch_0(E)}{L\cdot\ch_1^{\beta_0}(E)}\in \frac{1}{(bL\cdot\ch_1^{\beta_0}(F))!}\bZ$.
Hence by \autoref{SF}.\eqref{enum:BoundedBelow}, we can consider an element $E\in S_F$ with minimum $\frac{L^2\cdot\ch_0}{L\cdot\ch_1^{\beta_0}}$ among the objects of $S_F$;
by noetherianity of $\Coh^{\beta_0}(X)$, we may assume as well that $E$ is maximal with this property.

\begin{rem}
A priori, it is not obvious that $E$ must be unique.
This will follow from the proof of \autoref{HNalpha2}, where we will see that $E$ is the first step in the Bridgeland limit HN filtration of $F$ (recall that HN filtrations are unique). 
\end{rem}

\begin{lem}\label{maxDestObj}
Under the assumptions of \autoref{SF}, let $E\in S_F$ be an element with minimum $\frac{L^2\cdot\ch_0}{L\cdot\ch_1^{\beta_0}}$ among the objects of $S_F$ and which is maximal with this property.
Then, there exists $\alpha_0>0$ such that $E$ is $\sigma_{\alpha,\beta_0}$-semistable for every $\alpha\in[0,\alpha_0)$.
\end{lem}
\begin{proof}
The statement is clear for $\alpha=0$: $\sigma_{0,\beta_0}$-semistability of $E$ (actually of any element in $S_F$) trivially follows from that of $F$.

Now consider $\alpha>0$ such that $E$ is not $\sigma_{\alpha,\beta_0}$-semistable, and let $G\subsetneq E$ be a maximal destabilizing subobject.
Note that $0<L\cdot\ch_1^{\beta_0}(G)$ (since $E$ is $\sigma_{0,\beta_0}$-semistable) and $L\cdot\ch_1^{\beta_0}(G)<L\cdot\ch_1^{\beta_0}(E)$.

Moreover, $\nu_{0,\beta_0}(G)<\nu_{0,\beta_0}(E)$.
Indeed, an equality would imply (since $\nu_{\alpha,\beta_0}(G)>\nu_{\alpha,\beta_0}(E)$) that $G$ has smaller $\frac{L^2\cdot\ch_0}{L\cdot\ch_1^{\beta_0}}$ than $E$, contradicting our hypothesis on $E$.
Therefore
\[
\ch_2^{\beta_0}(G)<\ch_2^{\beta_0}(E)\cdot\frac{L\cdot\ch_1^{\beta_0}(G)}{L\cdot\ch_1^{\beta_0}(E)}<0
\]
(note that $\ch_2^{\beta_0}(E)<0$ because $\nu_{0,\beta_0}(E)=\nu_{0,\beta_0}(F)<0$). 

Now, since $\odisc(G)\geq0$ and $\ch_2^{\beta_0}(G)<0$, we have $\left(L\cdot\ch_1^{\beta_0}(G)\right)^2\geq2\left(L^2\cdot\ch_0(G)\right)\ch_2^{\beta_0}(G)$, so
\begin{equation*}
1\geq\frac{2L^2\cdot\ch_0(G)}{L\cdot\ch_1^{\beta_0}(G)}\cdot\nu_{0,\beta_0}(G)
\end{equation*}
and 
\begin{equation*}
\frac{-L^2\cdot\ch_0(G)}{L\cdot\ch_1^{\beta_0}(G)}\leq\frac{-1}{2\nu_{0,\beta_0}(G)}<\frac{-1}{2\nu_{0,\beta_0}(E)}
 <-\frac{(L\cdot\ch_1^{\beta_0}(E))^2}{\frac{2}{b}\ch_2^{\beta_0}(E)},
\end{equation*}
where in the last inequality we have used that $L\cdot\ch_1^{\beta_0}(E)>\frac{1}{b}$.
Therefore,
\begin{align*}
 \nu_{0,\beta_0}(E)-\frac{\alpha^2L^2\cdot\ch_0(E)}{2L\cdot\ch_1^{\beta_0}(E)}&=\nu_{\alpha,\beta_0}(E)<\nu_{\alpha,\beta_0}(G)=\nu_{0,\beta_0}(G)-\frac{\alpha^2L^2\cdot\ch_0(G)}{2L\cdot\ch_1^{\beta_0}(G)}\\
 &<\nu_{0,\beta_0}(G)-\frac{\alpha^2(L\cdot\ch_1^{\beta_0}(E))^2}{\frac{4}{b}\ch_2^{\beta_0}(E)}\leq\nu_{0,\beta_0}(E)-\frac{1}{k_E}-\frac{\alpha^2(L\cdot\ch_1^{\beta_0}(E))^2}{\frac{4}{b}\ch_2^{\beta_0}(E)},
\end{align*}
which gives the inequality
\[
\frac{\alpha^2}{2}\left(\frac{(L\cdot\ch_1^{\beta_0}(E))^2}{-\frac{2}{b}\ch_2^{\beta_0}(E)}+\frac{L^2\cdot\ch_0(E)}{L\cdot\ch_1^{\beta_0}(E)}\right)>\frac{1}{k_E}
\]

Using $\odisc(E)\geq0$, one can check that the factor multiplying $\frac{\alpha^2}{2}$ is positive.
Since this factor and $k_E$ only depend on $E$, this yields a lower bound for those $\alpha$ for which $E$ is not $\sigma_{\alpha,\beta_0}$-semistable.
\end{proof}

\begin{proof}[Proof of \autoref{HNalpha2}]
Consider the subobject $E\subseteq F$ defined below the proof of \autoref{SF}. 
It will be the first step of the Bridgeland limit filtration of $F$ at $\beta_0$. 
Indeed, an inductive process (applied to $F/E$) yields a chain of subobjects of $F$, which is finite by the noetherianity of $\Coh^{\beta_0}(X)$.

This chain is a HN filtration for $F$, valid for all $\sigma_{\alpha,\beta_0}$ with $\alpha>0$ small enough.
Indeed, on the one hand the semistability of the HN factors is guaranteed by \autoref{maxDestObj}, and on the other hand the inequalities of tilt slopes follow from the properties imposed to the subobjects taken at each step.\qedhere
\end{proof}

As a first consequence of \autoref{HNalpha2}, we obtain that $\sigma_{0,\beta_0}$-(semi)stability keeps some properties from Bridgeland stability:

\begin{cor}\label{disc2}
Let $F\in\Coh^{\beta_0}(X)$ be an object.
\begin{enumerate}[{\rm (1)}]
 \item\label{enum:DeltaPos} If $F$ is $\sigma_{0,\beta_0}$-semistable, then $\odisc(F)\geq0$.
 \item\label{enum:OpenStab} (Openness of stability) If $F$ is $\sigma_{0,\beta_0}$-stable with $\nu_{0,\beta_0}(F)\neq0$, then there exists $\alpha_0>0$ such that $F$ is $\sigma_{\alpha,\beta_0}$-stable for every $\alpha\in[0,\alpha_0)$.
 \item\label{enum:RegionSSt} If $F$ is $\sigma_{0,\beta_0}$-semistable with $\nu_{0,\beta_0}(F)\neq0$, there exists a region of Bridgeland stability conditions in the $(\alpha,\beta)$-plane for which $F$ is semistable.
\end{enumerate}
\end{cor}
\begin{proof}
Note that property \eqref{enum:DeltaPos} is trivially satisfied when $\ch_2^{\beta_0}(F)=0$.
If $\ch_2^{\beta_0}(F)\neq0$, according to \autoref{HNalpha2} the $\sigma_{0,\beta_0}$-semistable object $F$ has a Bridgeland limit HN filtration
\[
0=F_0\hookrightarrow F_1 \hookrightarrow \ldots \hookrightarrow F_{r-1}\hookrightarrow F_r=F
\]
valid for all $\sigma_{\alpha,\beta_0}$ with sufficiently small $\alpha>0$.
Of course each HN factor has $\odisc(F_k/F_{k-1})\geq 0$, and by construction of the filtration the equalities $\nu_{0,\beta_0}(F_k/F_{k-1})=\nu_{0,\beta_0}(F)$ hold.
Hence, the numbers $Z_{0,\beta_0}(F_k/F_{k-1})$ are in a ray of the complex plane and \cite[Lemma~A.7]{BMS} gives $\odisc(F)\geq\odisc(F_{r-1})\geq\ldots\geq\odisc(F_1)\geq 0$.

To prove \eqref{enum:OpenStab}, note that the result is trivial when $L\cdot\ch_1^{\beta_0}(F)=0$ (in this case, the $\sigma_{0,\beta_0}$-stability of $F$ is equivalent to $F$ being a simple object of $\Coh^{\beta_0}(X)$).

If $F$ is $\sigma_{0,\beta_0}$-stable with $L\cdot\ch_1^{\beta_0}(F)>0$, no subobject $E\in S_F\setminus\set{F}$ (notation as in \autoref{SF}) destabilizes $F$ for $\alpha>0$.
Indeed, since $F$ is $\sigma_{0,\beta_0}$-stable we must have $Z_{0,\beta_0}(F/E)=0$, so
\[
\nu_{\alpha,\beta_0}(E)<\nu_{\alpha,\beta_0}(F)<+\infty=\nu_{\alpha,\beta_0}(F/E)
\]
for every $\alpha>0$. 

By the construction of \autoref{HNalpha2}, all the potential $\sigma_{\alpha,\beta_0}$ destabilizers of $F$ (for $\alpha$ sufficiently small) are in $S_F\setminus\set{F}$; hence the strict inequalities actually imply that $F$ is $\sigma_{\alpha,\beta_0}$-stable (for $\alpha$ sufficiently small), which proves \eqref{enum:OpenStab}.

In \eqref{enum:RegionSSt}, there is nothing to prove if $F$ is $\sigma_{\alpha,\beta_0}$-semistable for small values of $\alpha>0$ (in particular, this covers the case where $L\cdot\ch_1^{\beta_0}(F)=0$).

For the rest of proof, we assume without loss of generality that $\nu_{0,\beta_0}(F)>0$ thanks to the duality functor (recall \autoref{dualstab}).
If $F_1,F_2/F_1,\ldots,F/F_{r-1}$ denote the factors of the Bridgeland limit HN filtration of $F$ at $\beta_0$, by construction we have:
\[
\nu_{0,\beta_0}(F_1)=\nu_{0,\beta_0}(F_2/F_1)=\ldots=\nu_{0,\beta_0}(F/F_{r-1})=\nu_{0,\beta_0}(F).
\]
Thus by Bertram's Nested Wall \autoref{structurewalls}, each factor $F_i/F_{i-1}$ defines the same (numerical) wall $W$ for $F$:  this wall is a semicircle whose left intersection point with the line $\alpha=0$ is $(0,\beta_0)$.

Since the factors $F_1,F_2/F_1,\ldots,F/F_{r-1}$ are $\sigma_{\alpha,\beta_0}$-semistable when $\alpha\geq0$ is small enough, it turns out that they are Bridgeland semistable along the wall $W$.
Hence $F$ is also semistable along $W$, since it is a (succesive) extension of semistable objects with the same slope.
\end{proof}

Now we are ready to improve \autoref{HNalpha2}, showing the existence of Bridgeland limit HN filtrations for objects without HN factors of vanishing tilt slope.
\begin{proof}[Proof of \autoref{existenceIntr}.\eqref{exist:Br}]
The result follows from induction on the length of the HN filtration of $F$ with respect to $\sigma_{0,\beta_0}$; the initial case is nothing but \autoref{HNalpha2}.

If $0=F_0\hookrightarrow F_1 \hookrightarrow \ldots \hookrightarrow F_{r-1}\hookrightarrow F_r=F$ is the HN filtration of $F$ with respect to $\sigma_{0,\beta_0}$, by induction hypothesis we may assume that both $F_{r-1}$ and $F/F_{r-1}$ admit a Bridgeland limit HN filtration at $\beta_0$.

Then, we can glue these filtrations to form that of $F$.
Indeed, if $F_{r-1}/A$ and $B/F_{r-1}$ respectively denote the last and the first limit HN factors of $F_{r-1}$ and $F/F_{r-1}$, the inequality
\[
\nu_{0,\beta_0}(F_{r-1}/A)=\nu_{0,\beta_0}(F_{r-1}/F_{r-2})>\nu_{0,\beta_0}(F/F_{r-1})=\nu_{0,\beta_0}(B/F_{r-1})
\]
guarantees this gluing, since we can take $\alpha_0$ small enough so that $\nu_{\alpha,\beta_0}(F_{r-1}/A)>\nu_{\alpha,\beta_0}(B/F_{r-1})$ for every $\alpha\in(0,\alpha_0)$.
\end{proof}

\subsection{Weak limit HN filtrations}\label{subsec:LocalExpr2}

Even if we are interested in semistability at the line $\alpha=0$, Bridgeland limit HN filtrations allow us to work in the $(\alpha,\beta)$-plane of Bridgeland stability conditions, where the wall-crossing phenomenon is well understood.

Following this strategy, already used in the proof of \autoref{disc2}.\eqref{enum:RegionSSt}, now we want to study HN filtrations at $\sigma_{0,\beta}$ for (rational) values of $\beta$ close to $\beta_0$:

\begin{defn}\label{weakLimitHN}
Given $F\in \Coh^{\beta_0}(X)$, if there exists $\epsilon>0$ such that $F$ has the same HN filtration with respect to all the weak stability conditions $\sigma_{0,\beta}$ with $\beta\in(\beta_0,\beta_0+\epsilon)\cap\bQ$, we will call this HN filtration the \emph{right weak limit HN filtration} of $F$ at $\beta_0$.
Its factors will be called \emph{right weak limit HN factors}.

Analogously we define the \emph{left weak limit HN filtration} of $F$ at $\beta_0$.
\end{defn}

\begin{lem}\label{HNbeta}
If $F\in\Coh^{\beta_0}(X)$ is an object with $\nu_{0,\beta_0}^+(F)<0$, then $F$ admits a right weak limit HN filtration at $\beta_0$ all of whose right weak limit factors have tilt slope $\nu_{0,\beta_0}<0$.
\end{lem}
\begin{proof}
Denote by
\[
0=F_0\hookrightarrow F_1\hookrightarrow\ldots\hookrightarrow F_{r-1} \hookrightarrow F_r=F
\]
the Bridgeland limit HN filtration of $F$ constructed in \autoref{existenceIntr}.\eqref{exist:Br}.
We will see that this is the desired right weak limit HN filtration for $F$.
Recall that when $\alpha=0$ (possibly) some of the HN factors get identified, if they have the same slope.

By assumption $\nu_{0,\beta_0}(F_i/F_{i-1})<0$ for every $i$, so the point $(0,\beta_0)$ lies on the right-hand side of the hyperbolas $\hyp{i}$ of the objects $F_i/F_{i-1}$.
In particular, if we study the locus in the $(\alpha,\beta)$-plane where these HN factors $F_i/F_{i-1}$ become non-semistable, we find that any such wall contains a segment of the line $\beta=\beta_0$ in its interior (see \autoref{fig3Hyp}).
\begin{figure}[ht]
\definecolor{rvwvcq}{rgb}{0.08235294117647059,0.396078431372549,0.7529411764705882}
\definecolor{dbwrru}{rgb}{0.8588235294117647,0.3803921568627451,0.0784313725490196}
\definecolor{sexdts}{rgb}{0.1803921568627451,0.49019607843137253,0.19607843137254902}
\definecolor{wrwrwr}{rgb}{0.3803921568627451,0.3803921568627451,0.3803921568627451}
\begin{tikzpicture}[line cap=round,line join=round,>=triangle 45,x=.8cm,y=.8cm]
\clip(-5,-.01) rectangle (5,5);
\draw [line width=1pt,color=wrwrwr,domain=-5:4] plot(\x,{(-0-0*\x)/-1});
\draw [shift={(0.4162301458038069,0)},line width=1pt,color=sexdts] plot[domain=0:3.141592653589793,variable=\t]({1*2.607348606082409*cos(\t r)+0*2.607348606082409*sin(\t r)},{0*2.607348606082409*cos(\t r)+1*2.607348606082409*sin(\t r)});
\draw [shift={(0.921772299877103,0)},line width=1pt,color=dbwrru] plot[domain=0:3.141592653589793,variable=\t]({1*2.569246438920924*cos(\t r)+0*2.569246438920924*sin(\t r)},{0*2.569246438920924*cos(\t r)+1*2.569246438920924*sin(\t r)});
\draw [shift={(2.3176191520859395,0)},line width=1pt,color=rvwvcq] plot[domain=0:3.141592653589793,variable=\t]({1*0.8511021754003596*cos(\t r)+0*0.8511021754003596*sin(\t r)},{0*0.8511021754003596*cos(\t r)+1*0.8511021754003596*sin(\t r)});
\draw[line width=1pt,dash pattern=on 4pt off 4pt,color=dbwrru, smooth,samples=100,domain=0:10] plot[domain=0:1.35,variable=\t] ({4.48-2.44149*cosh(\t)},{2.42337*sinh(\t)});
\draw[line width=1pt,dash pattern=on 5pt off 5pt,color=rvwvcq, smooth,samples=100,domain=0:10] plot[domain=0:1.72,variable=\t] ({5.95-3.2101*cosh(\t)},{1.60728*sinh(\t)});
\draw[line width=1pt,dash pattern=on 5pt off 5pt,color=sexdts, smooth,samples=100,domain=0:10] plot[domain=0:1.9,variable=\t] ({3.84-1.5648*cosh(\t)},{1.33978*sinh(\t)});
\draw [line width=1pt,color=wrwrwr] (2.843259810871818,0) -- (2.843259810871818,4.4);
\begin{footnotesize}
\draw[color=sexdts] (-2.7,.3) node {$W_2$};
\draw[color=dbwrru] (-1.2,.3) node {$W_3$};
\draw[color=rvwvcq] (1.1,.3) node {$W_1$};
\draw[color=dbwrru] (.2,4.2) node {$\hyp{3}$};
\draw[color=rvwvcq] (-1.9,3.2) node {$\hyp{1}$};
\draw[color=sexdts] (-1.3,3.7) node {$\hyp{2}$};
\draw[color=wrwrwr] (3.55,3.5) node {$\beta=\beta_0$};
\end{footnotesize}
\end{tikzpicture}
\caption{One hyperbola and at most one semicircular wall for each HN factor $F_i/F_{i-1}$}
\label{fig3Hyp}
\end{figure}
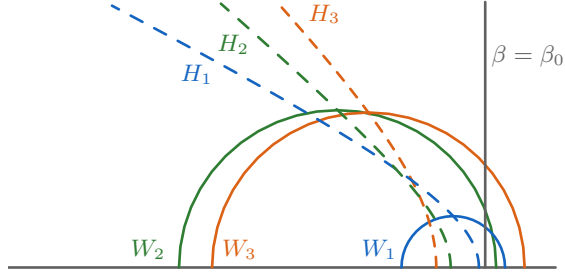

This proves that, for all $\beta$ in a sufficiently small right neighborhood of $\beta_0$, the objects $F_i/F_{i-1}$ are $\sigma_{0,\beta}$-semistable.
Shrinking if necessary this neighborhood, we have that the inequalities $\nu_{0,\beta}(F_i/F_{i-1})>\nu_{0,\beta}(F_{i+1}/F_{i})$ are preserved, no matter if these HN factors were merged in the HN filtration for $\sigma_{0,\beta_0}$.

Finally, \autoref{heart} guarantees that, possibly after another shrinking, the chain of inclusions
\[
0=F_0\hookrightarrow F_1\hookrightarrow\ldots\hookrightarrow F_{r-1} \hookrightarrow F_r=F
\]
also holds in $\Coh^\beta(X)$, for all $\beta$ in this neighborhood.
Summarizing, we have proved that this chain is the right weak limit HN filtration at $\beta_0$.
\end{proof}

When the object is semistable, the same holds for positive tilt slope:

\begin{lem}\label{HNbeta2}
If $F\in\Coh^{\beta_0}(X)$ is $\sigma_{0,\beta_0}$-semistable with $\nu_{0,\beta_0}(F)>0$, then $F$ admits a right weak limit HN filtration at $\beta_0$.
\end{lem}
\begin{proof}
We start by assuming $L\cdot\ch_1^{\beta_0}(F)>0$, so that $\beta=\beta_0$ is not a vertical wall for $F$.
According to \autoref{disc2}.\eqref{enum:RegionSSt} and its proof, the object $F$ is Bridgeland semistable along the (uniquely determined) semicircle $W'$ satisfying: $W'$ is centered at the $\beta$-axis, its top point lies on the hyperbola $\hyp{F}$ and its left intersection point with the $\beta$-axis is $(0,\beta_0)$.
This is true regardless of whether $W'$ is a numerical wall for $F$ or not.

Let $p=(\overline{\alpha},\overline{\beta})$ denote the top point of $W'$, i.e.~ its intersection point with $\hyp{F}$.
Local finiteness for Bridgeland stability conditions (see, e.g., \cite[Proposition~3.3.(b)]{BM:localP2}) ensures us that, for some $\epsilon'>0$, the HN filtration of $F$ is constant for all the stability conditions $\sigma_{\overline{\alpha},\beta}$ with $\beta\in(\overline{\beta}-\epsilon',\overline{\beta})$.
This implies that, in a small annulus inside $W'$, the HN filtration of $F$ stays constant 
(see \autoref{fig:annulus}).

\begin{figure}[ht]
\definecolor{ffffff}{rgb}{1,1,1}
\definecolor{sexdts}{rgb}{0.1803921568627451,0.49019607843137253,0.19607843137254902}
\definecolor{wrwrwr}{rgb}{0.3803921568627451,0.3803921568627451,0.3803921568627451}
\begin{tikzpicture}[line cap=round,line join=round,>=triangle 45,x=.8cm,y=.8cm]
\clip(-4.5,-0.1) rectangle (4.,4.5);
\draw [shift={(0.4162301458038069,0)},line width=1pt,color=sexdts,fill=sexdts,fill opacity=0.17] plot[domain=0:3.141592653589793,variable=\t]({1*2.607348606082409*cos(\t r)+0*2.607348606082409*sin(\t r)},{0*2.607348606082409*cos(\t r)+1*2.607348606082409*sin(\t r)});
\draw [line width=1pt,color=wrwrwr] (-2.1911184602786022,0) -- (-2.1911184602786022,8.949943566611623);
\draw [shift={(0.4162301458038069,0)},line width=1pt,color=ffffff,fill=ffffff,fill opacity=1] plot[domain=0:3.141592653589793,variable=\t]({1*2.27546707940022*cos(\t r)+0*2.27546707940022*sin(\t r)},{0*2.27546707940022*cos(\t r)+1*2.27546707940022*sin(\t r)});
\draw [line width=1pt,color=wrwrwr,domain=-6.483358732323508:7.800475128998744] plot(\x,{(-0-0*\x)/-1});
\draw[line width=1pt,dash pattern=on 5pt off 5pt,color=sexdts, smooth,samples=100,domain=0:10] 
plot[domain=0:2.2,variable=\t] ({3.84-1.5648*cosh(\t)},{1.33978*sinh(\t)});
\filldraw[color=sexdts] (0.4162301458038069,2.607348606082409) circle (2.5pt); 
\begin{footnotesize}
\draw[color=sexdts] (2.76,1.7) node {$W'$};
\draw[color=sexdts] (-0.5,4) node {$\hyp{F}$};
\draw[color=sexdts] (0.47,3) node {$p$};
\draw[color=wrwrwr] (-3.,4.) node {$\beta=\beta_0$};
\end{footnotesize}
\end{tikzpicture}
\caption{}
\label{fig:annulus}
\end{figure}

This constant filtration is the right weak limit HN filtration of $F$ at $\beta_0$, as claimed.

It only remains to check the case when $L\cdot\ch_1^{\beta_0}(F)=0$.
The strategy is similar.
In this case, $F$ is Bridgeland semistable along the whole vertical wall $\beta=\beta_0$.
Local finiteness for Bridgeland stability conditions gives that, inside $\{\alpha>0,\beta\geq\beta_0\}$, the HN filtration of $F$ remains constant for all the stability conditions in a certain open neighborhood of $\beta=\beta_0$.

Using that objects with $\odisc=0$ never get destabilized, it is not difficult to check that this holds in a ``tubular'' neighborhood of the form $(0,+\infty)\times[\beta_0,\beta_0+\epsilon)$, for a certain $\epsilon>0$.
This gives the desired HN filtration of $F$ with respect to $\sigma_{0,\beta}$, when $\beta$ lies in a right neighborhood of $\beta_0$.
\end{proof}

In view of \autoref{HNbeta} and \autoref{HNbeta2}, we would like to conclude the right case of \autoref{existenceIntr}.\eqref{exist:Wk}.
Nevertheless, if our object $F$ has more than one HN factor at $\sigma_{0,\beta_0}$ with positive slope, the way of gluing their right weak limit HN filtrations is not trivial.
Roughly, this is caused because not all the weak limit HN factors in a limit HN filtration approach the same slope as $\beta\to\beta_0^+$, since $\sigma_{0,\beta_0}$ is not a proper stability condition.

In order to solve this problem, we collect some information in the following lemma:

\begin{lem}\label{kernel}
Let $F\in\Coh^{\beta_0}(X)$ be a $\sigma_{0,\beta_0}$-semistable object with $\nu_{0,\beta_0}(F)>0$.
Then:
\begin{enumerate}[{\rm (1)}]

    \item\label{enum:minimal} There exists a (possibly trivial) subobject $F'\subset F$ in $\Coh^{\beta_0}(X)$ such that $Q=F/F'\in \ker(Z_{0,\beta_0})$.
    Moreover, $F'$ can be taken minimal satisfying this property.
    We call it a \emph{core subobject} of $F$.

    \item\label{enum:discont} If $Q\neq0$, then its tilt slope is discontinuous with respect to $\beta$, i.e,
    \[
    \nu_{0,\beta_0}(Q)=+\infty, \qquad\lim_{\beta\to\beta_0^+}\nu_{0,\beta}(Q)=0.
    \]
    \item\label{enum:slopeF'} If $F'\neq0$, then the tilt slope $\nu_{0,\beta}$ of every right weak limit HN factor of $F'$ has limit $\nu_{0,\beta_0}(F')$ as $\beta\to\beta_0^+$.

    \item\label{enum:rightF} The right weak limit HN filtration of $F$ at $\beta_0$ 
    consists of the right weak limit HN filtration of $F'$ together with the quotient $Q=F/F'$.
    In particular, the core subobject $F'$ is unique.
    
\end{enumerate}
\end{lem}

\begin{proof}
The existence of $F'$ is clear, since we allow $F'=F$.
If $F'\neq F$, then by \autoref{Giesekerkernel} $Q[-1]$ is a twisted $(L,-\frac{1}{2}K_X)$-Gieseker semistable vector bundle with $\mu_L=\beta_0$ and $\odisc=0$.
Therefore we can consider $F'$ minimal satisfying $Q=F/F'\in \ker(Z_{0,\beta_0})$, because if
 \[
 \ldots\subset F_2' \subset F_1' \subset F
 \]
is a chain of subobjects with this property, then one has $0\leq\ldots<\odisc(F_2')<\odisc(F_1')<\odisc(F)$. 

Now, a simple computation shows \eqref{enum:discont}.
In order to prove \eqref{enum:slopeF'}, note that
\autoref{HNbeta2} guarantees the existence of a right weak limit HN filtration for $F'$ at $\beta_0$. 

If $L\cdot\ch_1^{\beta_0}(F')>0$, by construction this is the HN filtration in an annulus inside the wall for $F'$ passing through $(0,\beta_0)$.
It is easy to check, under the assumption of minimality on $F'$, that no right weak limit HN factor of $F'$ has $\beta=\beta_0$ as a vertical wall.
Hence every right weak limit HN factor has $L\cdot\ch_1^{\beta_0}>0$, which shows the continuity of its tilt slope in a neighborhood of $(0,\beta_0)$.

If $L\cdot\ch_1^{\beta_0}(F')=0$, then every right weak limit HN factor of $F'$ either has tilt slope $\nu_{0,\beta}=+\infty$ for every $\beta$, or has $\odisc>0$ and $\beta=\beta_0$ as a vertical wall.
In both cases, the tilt slope $\nu_{0,\beta}$ approaches $\nu_{0,\beta_0}(F')=+\infty$ as $\beta\to\beta_0^+$.
This finishes the proof of \eqref{enum:slopeF'}.

Finally, observe that $Q\in\Coh^\beta(X)$ is semistable for every  $\sigma_{\alpha,\beta}$ with $\alpha\geq0$ and $\beta\geq\beta_0$, thanks to the description provided by \autoref{Giesekerkernel}.
Thus \eqref{enum:rightF} is a consequence of  \eqref{enum:slopeF'} (all the weak limit HN factors of $F'$ approach the slope $\nu_{0,\beta_0}(F')=\nu_{0,\beta_0}(F)>0$ as $\beta\to\beta_0^+$) and \eqref{enum:discont}.
\end{proof}

Note that if the object $F'$ is non-trivial (i.e.~$F'\neq0,F$), then $F$ may be $\sigma_{0,\beta_0}$-stable and $\sigma_{0,\beta}$-nonsemistable for every $\beta\to\beta_0^+$. 
We can visualize the situation in \autoref{figDiscont}. 
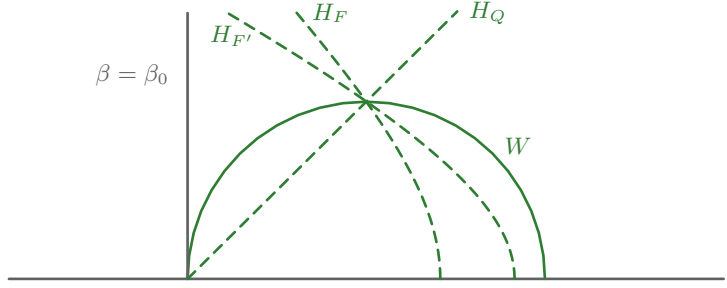
\begin{figure}[ht]
\definecolor{ffffff}{rgb}{1,1,1}
\definecolor{sexdts}{rgb}{0.1803921568627451,0.49019607843137253,0.19607843137254902}
\definecolor{wrwrwr}{rgb}{0.3803921568627451,0.3803921568627451,0.3803921568627451}
\begin{tikzpicture}[line cap=round,line join=round,>=triangle 45,x=2.35cm,y=2.35cm]
\draw [line width=1pt,color=sexdts] plot[domain=0:3.141592653589793,variable=\t]({cos(\t r)},{sin(\t r)});
\draw [line width=1pt,color=wrwrwr] (-2,0) -- (2,0);
\draw [line width=1pt,color=wrwrwr] (-1,0) -- (-1,1.5);
\draw[line width=1pt,dash pattern=on 4pt off 3pt,color=sexdts, smooth,samples=100,domain=0:10] 
plot[domain=-1:0.509461,variable=\t]({\t},{\t+1});
\draw[line width=1pt,dash pattern=on 4pt off 3pt,color=sexdts, smooth,samples=100,domain=0:10] 
plot[domain=0:1.2,variable=\t]({-exp(\t)-exp(-\t)+2.83},{(exp(\t)-exp(-\t))/2});
\draw[line width=1pt,dash pattern=on 4pt off 3pt,color=sexdts, smooth,samples=100,domain=0:10] 
plot[domain=0:1.2,variable=\t]({(-exp(\t)-exp(-\t))/2+1.41421356},{(exp(\t)-exp(-\t))/2});

\begin{footnotesize}
\draw[color=sexdts] (0.85,.75) node {$W$};
\draw[color=sexdts] (0.68,1.5) node {$\hyp{Q}$};
\draw[color=sexdts] (-0.2,1.5) node {$\hyp{F}$};
\draw[color=sexdts] (-0.75,1.37) node {$\hyp{F'}$};
\draw[color=wrwrwr] (-1.31,1.15) node {$\beta=\beta_0$};
\end{footnotesize}
\end{tikzpicture}
\caption{The wall $W$ is defined by $0\to F'\to F\to Q\to 0$.
}
\label{figDiscont}
\end{figure}

\begin{prop}\label{HNbeta3}
If $F\in\Coh^{\beta_0}(X)$ is an object satisfying $\nu_{0,\beta_0}^-(F)>0$, then $F$ admits a right weak limit HN filtration at $\beta_0$.
Furthermore, all of its right weak limit HN factors satisfy
\[\lim_{\beta\to \beta_0^+} \nu_{0,\beta}\geq 0\]
with equality for at most the last weak limit HN factor.
\end{prop}

\begin{proof}
Let $0=F_0\hookrightarrow F_1 \hookrightarrow \ldots \hookrightarrow F_s=F$ be the HN filtration of $F$ at $\sigma_{0,\beta_0}$.
We construct inductively a chain of subobjects $G_1\subset\ldots\subset G_s$ (with $G_i\subset F_i$) as follows: $G_1=F_1'$ is the core subobject of $F_1$ (throughout this proof we use the notation of $F'$ to denote the core of $F$ as defined in \autoref{kernel}.\eqref{enum:minimal}), and for $i>1$ we define $G_i$ in such a way that $(F_i/G_{i-1})'=G_i/G_{i-1}$.

We will prove, by induction on $s$, that $F$ admits a right weak limit HN filtration whose last limit HN factor is $F/G_s$, and that every weak limit HN factor of $G_s$ has positive limit tilt slope.
Note that the case $s=1$ follows from \autoref{kernel}.

For the induction step, one could assume that the statement is true for $F_{s-1}$.
Nevertheless, it is possible that the right weak limit filtrations of $F_{s-1}$ and $F_s/F_{s-1}$ do not glue directly, since it may happen
\[
\lim_{\beta\to\beta_0^+}\nu_{0,\beta}(F_{s-1}/G_{s-1})=0<\lim_{\beta\to\beta_0^+}\nu_{0,\beta}(\text{first limit HN factor of $F_s/F_{s-1}$}).
\] 

Then, the strategy consists on replacing $F_{s-1}$ by $G_{s-1}$.
We are allowed to do this because the HN filtration of $G_{s-1}$ with respect to $\sigma_{0,\beta_0}$ has length $\leq s-1$.
Indeed, $G_{s-1}$ has the same HN polygon as $F_{s-1}$, as a consequence of the equality $Z_{0,\beta_0}(G_i)=Z_{0,\beta_0}(F_i)$ for every $i$ (see \cite[Section~3]{BayerDef} or \cite[Section 4]{MS} for the definition and the properties of the HN polygon).

Now, on the one hand by induction hypothesis the right weak limit HN filtration of $G_{s-1}$ exists, and it satisfies
 \[
 \lim_{\beta\to\beta_0^+}\nu_{0,\beta}(\text{last weak limit HN factor of $G_{s-1}$})=\nu_{0,\beta_0}(G_{s-1}/G_{s-2})
 \]
 
On the other hand, we claim that $F_s/G_{s-1}$ admits a right weak limit HN filtration; this is not obvious at all, since $F_s/G_{s-1}$ is not $\sigma_{0,\beta_0}$-semistable (it contains the subobject $F_{s-1}/G_{s-1}$ with slope $\nu_{0,\beta_0}=+\infty$) and we cannot apply directly \autoref{HNbeta2}.
 
So to see this, we consider the wall $W$ in the $(\alpha,\beta)$-plane defined by the short exact sequence of $\Coh^{\beta_0}(X)$
\[
0\to F_{s-1}/G_{s-1}\to  F_s/G_{s-1} \to F_s/F_{s-1}\to 0
\]
The left intersection point of $W$ with $\alpha=0$ is $(0,\beta_0)$; the picture is similar to that of \autoref{figDiscont}, simply replacing $\hyp{Q},\hyp{F},\hyp{F'}$ by $\hyp{F_{s-1}/G_{s-1}},\hyp{F_s/G_{s-1}},\hyp{F_s/F_{s-1}}$ (respectively).
 
Note that $F_s/G_{s-1}$ is $\sigma_{\alpha,\beta}$-semistable for all the Bridgeland stability conditions $(\alpha,\beta)\in W$ with $\alpha>0$, since it is an extension of $\sigma_{\alpha,\beta}$-semistable objects of the same slope. 
Then, as in the proof of \autoref{HNbeta2}, local finiteness ensures that the HN filtration of $F_s/G_{s-1}$ is constant for all the stability conditions in a small annulus inside $W$.
This gives the desired right weak limit HN filtration for $F_s/G_{s-1}$.

Every factor of this right weak limit HN filtration for $F_s/G_{s-1}$ satisfies
\begin{align*}
 0&<\lim_{\beta\to\beta_0^+}\nu_{0,\beta}(\text{every limit HN factor of $F_s/G_{s-1}$})=\nu_{0,\beta_0}(F_s/G_{s-1})=\nu_{0,\beta_0}(F_s/F_{s-1})\\
 &<\nu_{0,\beta_0}(F_{s-1}/F_{s-2})=\nu_{0,\beta_0}(G_{s-1}/G_{s-2})=\lim_{\beta\to\beta_0^+}\nu_{0,\beta_0}(\text{last limit HN factor of $G_{s-1}$})
\end{align*}
with the only possible exception of the last weak limit HN factor for $F_s/G_{s-1}$, which may have tilt slope of limit 0.
Using the induction hypothesis, one easily checks that this last weak limit HN factor for $F_s/G_{s-1}$ is precisely $F_s/G_s$.

This allows to glue the right weak limit filtrations of $G_{s-1}$ and $F_s/G_{s-1}$ to produce that of $F_s$, which proves the assertion.
\end{proof}

Now we are ready to prove the main existence result of this subsection, namely the existence of right weak limit HN filtration for objects without HN factors of vanishing tilt slope (see \autoref{LeftWkFiltr} for the left filtrations).
\begin{proof}[Proof of \autoref{existenceIntr}.\eqref{exist:Wk}. Right case.]
Denote by $0=F_0\hookrightarrow F_1 \hookrightarrow \ldots \hookrightarrow F_s \hookrightarrow F_{s+1} \hookrightarrow \ldots \hookrightarrow F_r=F$ the HN filtration of $F$ with respect to $\sigma_{0,\beta_0}$, so that
\[
\nu_{0,\beta_0}(F_1)>\ldots>\nu_{0,\beta_0}(F_s/F_{s-1})>0>\nu_{0,\beta_0}(F_{s+1}/F_s)>\ldots>\nu_{0,\beta_0}(F/F_{r-1})
\]

On the one hand, by \autoref{HNbeta} $F/F_s$ admits a right weak limit HN filtration.
It turns out that, if $R$ is a weak limit HN factor of $F/F_s$ (corresponding to a weak limit HN factor of $F_i/F_{i-1}$ for some $i=s+1,\ldots,r$), then
\[
\lim_{\beta\to\beta_0^+}\nu_{0,\beta}(R)=\nu_{0,\beta_0}(F_i/F_{i-1})<0
\]

On the other hand, according to \autoref{HNbeta3}, $F_s$ has a right weak limit HN filtration, with all its weak limit HN factors satisfying $\lim_{\beta\to\beta_0^+}\nu_{0,\beta}\geq0$.

A standard glueing of the right weak limit filtrations for $F_s$ and $F/F_s$ finishes the proof.
\end{proof}

\begin{rem}\label{LeftWkFiltr}
Under the same hypothesis on $F\in\Coh^{\beta_0}(X)$, the existence of a left weak limit HN filtration at $\beta_0$ follows from the preservation of stability by the derived dual (see \autoref{dualstab}).

The main difference is that for $\beta\to\beta_0^-$ the object $F$ may have two nontrivial cohomologies with respect to $\Coh^\beta(X)$, namely
\[
\cH^{-1}_\beta(F)=\cH^{-1}(F_1),\;\;\cH^0_\beta(F)=F/(\cH^{-1}(F_1)[1])
\]
where $F_1$ is a first step in the HN filtration of $F$ with respect to $\sigma_{0,\beta_0}$, satisfying $\nu_{0,\beta_0}(F_1)=+\infty$.

The left weak limit filtration of $F$ is thus obtained by applying $\blank^\vee[1]$ to the right weak limit filtration at $-\beta_0$ of $\cH^0_{-\beta_0}(F^\vee[1])=(F/F_1)^\vee[1]$ and $\cH^1_{-\beta_0}(F^\vee[1])=F_1^\vee[2]$.
\end{rem}

\subsection{Local piecewise polynomial expressions} \label{subsec:LocalExpr3}
An immediate consequence of the right weak limit HN filtrations constructed in \autoref{existenceIntr}.\eqref{exist:Wk} is:
\begin{cor}\label{polynomial1}
If $F\in\Coh^{\beta_0}(X)$ is an object with all its HN factors with respect to $\sigma_{0,\beta_0}$ having slope $\nu_{0,\beta_0}\neq0$, then there exists $\epsilon>0$ such that the functions $\rf^0_{F,L},\rf^1_{F,L}$ are (piecewise) polynomial along the interval $(-\beta_0-\epsilon,-\beta_0)$.
\end{cor}

More explicitly, let $0=F_0\hookrightarrow F_1 \hookrightarrow \ldots \hookrightarrow F_s \hookrightarrow F_{s+1} \hookrightarrow \ldots \hookrightarrow F_r=F$ be the HN filtration of $F$ with respect to $\sigma_{0,\beta_0}$, so that
\[
\nu_{0,\beta_0}(F_1)>\ldots>\nu_{0,\beta_0}(F_s/F_{s-1})>0>\nu_{0,\beta_0}(F_{s+1}/F_s)>\ldots>\nu_{0,\beta_0}(F/F_{r-1})
\]

Consider the chain $G_1\subset\ldots\subset G_s$ in $\Coh^{\beta_0}(X)$ (with $G_i\subset F_i$) inductively defined by the rules $G_1=F_1'$ and $G_i/G_{i-1}=(F_i/G_{i-1})'$ (where $F'$ denotes the core of $F$ as defined in \autoref{kernel}.\eqref{enum:minimal}).
Then, for all (rational) $x$ in a left neighborhood of $-\beta_0$ we have
\[
\rf^0_{F,L}(x)=\ch_2^{-x}(G_s),\;\;\rf^1_{F,L}(x)=-\ch_2^{-x}(F/G_s).
\]


\begin{rem}
Along the interval where the right weak limit HN filtration of $F$ at $\beta_0$ remains constant, the functions $\rf^0_{F,L},\rf^1_{F,L}$ may still change their polynomial expression.
This happens if the last right weak limit HN factor of $G_s$ acquires tilt slope 0.
\end{rem}


Observe that the chain $G_1\subset\ldots\subset G_s$ of the proof of \autoref{polynomial1} encodes the Chern degree functions of $F$ in a left neighborhood of $-\beta_0$.
For easy reference, we will use the following terminology (that we adopt for objects $F$ with a HN factor of tilt slope 0 as well):

\begin{defn}
Let $F\in\Coh^{\beta_0}(X)$ and let $F_1,\ldots,F_s/F_{s-1}$ be the HN factors of $F$ having slope $\nu_{0,\beta_0}>0$, that is, $s=s(F)$ is the switching index of $F$.
The \emph{core filtration} of $F$ at $\beta_0$ is the chain $G_1\subset\ldots\subset G_s$ in $\Coh^{\beta_0}(X)$ (with $G_i\subset F_i$) inductively defined by $G_1=F_1'$ and $G_i/G_{i-1}=(F_i/G_{i-1})'$ where $F'$ denotes the core of $F$ as defined in \autoref{kernel}.\eqref{enum:minimal}.
\end{defn}

Now we want to conclude that for every $F\in\Coh^{\beta_0}(X)$ the function $\rf^0_{F,L}$ admits the (piecewise) polynomial expression $\ch_2^{-x}(G_s)$ in a left neighborhood of $-\beta_0$, where $G_s$ is the last object appearing in the core filtration of $F$ at $\beta_0$.
For this, we need to treat the case of HN factors with slope $\nu_{0,\beta_0}=0$; we will not find a weak limit HN filtration for them, but the following result will be enough for our purposes.

\begin{prop}\label{slope0}
If $F\in\Coh^{\beta_0}(X)$ is $\sigma_{0,\beta_0}$-semistable with $\nu_{0,\beta_0}(F)=0$, then $\rf^0_{F,L}(-\beta)=0$ for every rational number $\beta>\beta_0$ such that $F\in\Coh^\beta(X)$.
\end{prop}

\begin{proof}
Suppose, for the sake of a contradiction, that $\wbeta>\beta_0$ is a rational number with $F\in\Coh^{\wbeta}(X)$ and $\rf^0_{F,L}(-\wbeta)\neq0$.
This means that $F$ has a subobject $E\subset F$ in $\Coh^{\wbeta}(X)$ satisfying 
$\ch_2^{\wbeta}(E)>0$.

We may assume that $E$ is the first HN factor $F_1$ of $F$ with respect to $\sigma_{0,\wbeta}$; to see this, we only need to exclude that $\ch_2^{\wbeta}(F_1)=0$.
And indeed, since $\nu_{0,\wbeta}(F_1)>0$ the equality $\ch_2^{\wbeta}(F_1)=0$ would imply that $F_1\in\ker(Z_{0,\wbeta})$;
thus by \autoref{Giesekerkernel} $\cH^{-1}(F_1)$ would be a (twisted Gieseker semistable) sheaf of slope $\wbeta$.
But then the inclusion $\cH^{-1}(F_1)\subset\cH^{-1}(F)$ would contradict the hypothesis $F\in\Coh^{\beta_0}(X)$ (recall \autoref{heart}).

Replacing $E$ by the first step of its Bridgeland limit HN filtration at $\wbeta$ (\autoref{existenceIntr}.\eqref{exist:Br}), we may also assume that $E$ is $\sigma_{\alpha,\wbeta}$-semistable for some positive values of $\alpha$.

Consider the distinguished triangle $E\to F\to Q$ in $\Db(X)$ inducing the inclusion $E\subset F$ in $\Coh^{\wbeta}(X)$.
On the one hand, note that $E\in\Coh^{\beta_0}(X)$ as well;
indeed, the inequality $\mu_L^+(\cH^{-1}(E))\leq\beta_0$ follows from $\cH^{-1}(E)\subset\cH^{-1}(F)$ (recall \autoref{heart}).
Moreover, by \autoref{hyperbola} $\nu_{0,\wbeta}(E)>0$ is equivalent to $\wbeta<p_E$; hence $\beta_0<p_E$ holds, which gives $\nu_{0,\beta_0}(E)>0$.

On the other hand, it is possible that $Q\notin\Coh^{\beta_0}(X)$, so $Q$ may have two nontrivial cohomologies with respect to the heart $\Coh^{\beta_0}(X)$, namely $\cH^{-1}_{\beta_0}(Q)$ and $\cH^0_{\beta_0}(Q)$.
In particular, $\cH^{-1}_{\beta_0}(Q)$ is a subsheaf of $\cH^{-1}(Q)$ such that $\mu_+(\cH_{\beta_0}^{-1}(Q))\in (\beta_0,\wbeta]$.

Therefore the distinguished triangle yields an exact sequence
\[
0\to \cH^{-1}_{\beta_0}(Q)\to E\to F\to \cH^0_{\beta_0}(Q)\to 0
\]
in $\Coh^{\beta_0}(X)$.
Recall that $F$ is $\sigma_{0,\beta_0}$-semistable with $\nu_{0,\beta_0}(F)=0$, and $\nu_{0,\beta_0}(E)>0$.

If $E$ is $\sigma_{0,\beta_0}$-semistable, we already have the desired contradiction.
In fact, in such a case $\Hom(E,F)=0$ and thus $\cH^{-1}_{\beta_0}(Q)=E$; but this would imply $E\notin\Coh^{\wbeta}(X)$ by our previous description of $\cH^{-1}_{\beta_0}(Q)$.

Therefore, to finish the proof it suffices to check that $E$ may be assumed to be $\sigma_{0,\beta_0}$-semistable.
This is essentially due to the support property for Bridgeland stability conditions in the $(\alpha,\beta)$-plane.

Recall that $E$ is $\sigma_{\alpha,\wbeta}$-semistable for small enough values of $\alpha>0$.
Hence, if $E$ is not $\sigma_{0,\beta_0}$-semistable, there exists a short exact sequence $E_1\hookrightarrow E\twoheadrightarrow R_1$ destabilizing $E$ along a wall $W_1$ in the $(\alpha,\beta)$-plane, whose left point $(0,\wbeta_1)$ in the $\beta$-axis satisfies $\beta_0<\wbeta_1<\wbeta$:

\begin{figure}[H]
\definecolor{ffffff}{rgb}{1,0,0}
\definecolor{sexdts}{rgb}{0.1803921568627451,0.49019607843137253,0.19607843137254902}
\definecolor{wrwrwr}{rgb}{0.3803921568627451,0.3803921568627451,0.3803921568627451}
\begin{tikzpicture}[line cap=round,line join=round,>=triangle 45,x=1.6cm,y=1.6cm]
\draw [line width=1pt,color=ffffff] plot[domain=0:3.141592653589793,variable=\t]({cos(\t r)},{sin(\t r)});
\draw [line width=1pt,color=wrwrwr] (-4,0) -- (2,0);
\draw [line width=1pt,color=wrwrwr] (-2.8,0) -- (-2.8,2.35);
\draw [line width=1pt,color=wrwrwr] (-1,0) -- (-1,2.35);
\draw [line width=1pt,color=wrwrwr] (-0.72,0) -- (-0.72,2.35);
\draw[line width=1pt,dash pattern=on 4pt off 3pt,color=sexdts, smooth,samples=100,domain=0:10] 
plot[domain=0:1.6,variable=\t]({-exp(\t)-exp(-\t)+2.83},{(exp(\t)-exp(-\t))/2});
\draw[line width=1pt,dash pattern=on 4pt off 3pt,color=sexdts, smooth,samples=100,domain=0:10] 
plot[domain=0:1.6,variable=\t]({(-exp(\t)-exp(-\t))/2+1.41421356},{(exp(\t)-exp(-\t))/2});
\draw[line width=1pt,dash pattern=on 4pt off 3pt,color=sexdts, smooth,samples=100,domain=0:10] 
plot[domain=0:1.54,variable=\t]({exp(\t)+exp(-\t)-2*1.41421356},{(exp(\t)-exp(-\t))/2});
\draw[line width=1pt,dash pattern=on 4pt off 3pt,color=sexdts, smooth,samples=100,domain=0:10] 
plot[domain=0:1.6,variable=\t]({3*(exp(\t)+exp(-\t))/4-4.3},{(exp(\t)-exp(-\t))/2});

\begin{footnotesize}
\draw[color=sexdts] (2,2) node {$\hyp{E_1}$};
\draw[color=sexdts] (-1.35,2.35) node {$\hyp{E}$};
\draw[color=sexdts] (-1.95,2.35) node {$\hyp{R_1}$};
\draw[color=sexdts] (-0.25,2.35) node {$\hyp{F}$};
\draw[color=wrwrwr] (-0.72,-0.15) node {$\wbeta$};
\draw[color=wrwrwr] (-2.8,-0.15) node {$\beta_0$};
\draw[color=wrwrwr] (-1,-0.15) node {$\wbeta_1$};
\draw[color=sexdts] (0.4,-0.15) node {$p_E$};
\draw[color=ffffff] (1,0.65) node {$W_1$};

\end{footnotesize}
\end{tikzpicture}
\caption{}
\label{fig:newdest}
\end{figure}

We can assume, without loss of generality, that $E_1$ is the first step of the Bridgeland limit HN filtration of $E$ at $\wbeta_1$.
Hence $E_1$ is $\sigma_{\alpha,\wbeta_1}$-semistable for small enough values of $\alpha>0$.

At this point, observe that by the support property (in the form of \autoref{structurewalls}.\eqref{item:Disc0nodest}) we have $\odisc(E_1)+\odisc(R_1)<\odisc(E)$ (in particular, $\odisc(E_1)<\odisc(E)$).
Moreover, we have an inclusion $E_1\subset F$ in $\Coh^{\wbeta}(X)$ (since $E_1\subset E$ holds along the whole wall $W_1$).
Let $E_1\to F\to Q_1$ be the distinguished triangle in $\Db(X)$ defining this inclusion.

Then, the same arguments as before show that $E_1\in\Coh^{\beta_0}(X)$ and give an exact sequence
 \[
 0\to \cH^{-1}_{\beta_0}(Q_1)\to E_1\to F\to \cH^0_{\beta_0}(Q_1)\to 0
 \]
in $\Coh^{\beta_0}(X)$.
Moreover, $\nu_{0,\beta_0}(E_1)>0$ since $(0,\beta_0)$ lies on the left-hand side of $\hyp{E_1}$.

If $E_1$ were $\sigma_{0,\beta_0}$-semistable, by reasoning as in the case of $E$ $\sigma_{0,\beta_0}$-semistable we would obtain a contradiction.
Otherwise, we destabilize $E_1$ along a wall $W_2$ via a short exact sequence $0\to E_2\to E_1\to R_2\to 0$ with the same properties.

This finishes the proof, since this process must stop after a finite number of destabilizations thanks to the inequalities
\[
0\leq\ldots<\odisc(E_2)<\odisc(E_1)<\odisc(E)\qedhere
\]
\end{proof}


\begin{cor}\label{polynomial2}
If $F\in\Coh^{\beta_0}(X)$, there exists $\epsilon>0$ such that the functions $\rf_{F,L}^0$ and $\rf_{F,L}^1$ are (piecewise) polynomial along the interval $(-\beta_0-\epsilon,-\beta_0)$.
\end{cor}
\begin{proof}
We assume that $F$ has a HN factor with respect to $\sigma_{0,\beta_0}$ of slope $\nu_{0,\beta_0}=0$, otherwise we are in the situation of \autoref{polynomial1}.
Thus let
\[
0=F_0\hookrightarrow F_1 \hookrightarrow \ldots \hookrightarrow F_{s-1} \hookrightarrow F_s \hookrightarrow F_{s+1} \hookrightarrow \ldots \hookrightarrow F_r=F
\]
be the HN filtration of $F$ with respect to $\sigma_{0,\beta_0}$, so that $\nu_{0,\beta_0}(F_{s+1}/F_s)=0$, in particular, $s=s(F)$ is the switching index of $F$.
We claim that $\rf^0_{F/F_s,L}(-\beta)=0$ for all $\beta$ in a certain right neighborhood of $\beta_0$. 
This is a consequence of \autoref{sesequence}, when applied to the short exact sequence in $\Coh^\beta(X)$
\[
0\to F_{s+1}/F_s\to F/F_s\to F/F_{s+1}\to 0,
\]
taking into account that 
$\rf^0_{F_{s+1}/F_s,L}(-\beta)=0$ (by \autoref{slope0}) and $\rf^0_{F/F_{s+1},L}(-\beta)=0$ (by \autoref{HNbeta}).
Indeed, we can apply \autoref{sesequence} because $F/F_{s+1}$ has no subobject belonging to $\ker(Z_{0,\beta})$: otherwise, $\cH^{-1}(F/F_{s+1})$ would have a subsheaf of slope $\beta$, contradicting $F/F_{s+1}\in\Coh^{\beta_0}(X)$.

Now, for all $\beta>\beta_0$ in a certain right neighborhood of $\beta_0$, we use the vanishing $\rf^0_{F/F_s,L}(-\beta)=0$ to apply \autoref{sesequence} again, in this case with the short exact sequence $0\to F_s\to F\to F/F_s\to 0$. 
We obtain the equality $\rf^0_{F,L}(-\beta)=\rf^0_{F_s,L}(-\beta)$ in a right neighborhood of $\beta_0$. 

This proves the result for $\rf^0_{F,L}$, since $\rf^0_{F_s,L}$ is (piecewise) polynomial in (a shrinking of) this neighborhood by \autoref{polynomial1}.
The assertion for $\rf^1_{F,L}$ is simply a consequence of the relation $\rf^1_{F,L}(x)=\rf^0_{F,L}(x)-\ch_2^{-x}(F)$.
\end{proof}

\begin{proof}[Proof of \autoref{main}.\eqref{item:main1}]
The left polynomial expression for $\rf^k_{E,L}$ at $x_0$ is a consequence of \autoref{polynomial2} applied to the cohomology objects $\cH^k_{-x_0}(E)$ and $\cH^{k-1}_{-x_0}(E)$.
The right polynomial expression can be obtained from the left polynomial expression of $\rf^{2-k}_{E^\vee,L}$ at $-x_0$, thanks to the Serre duality of \autoref{rk:BasicProp}.\eqref{enum:SerreDua}.
\end{proof}

\begin{rem}\label{rem:IrratWithBrLimit}
If $\beta_0\in\bR\setminus\bQ$ is an irrational number at which $F$ admits a Bridgeland limit filtration, then the same conclusion of \autoref{polynomial2} holds.
Indeed, under this assumption $F$ admits a HN filtration
\[
0=F_0\hookrightarrow F_1 \hookrightarrow \ldots \hookrightarrow F_{s-1} \hookrightarrow F_s \hookrightarrow F_{s+1} \hookrightarrow \ldots \hookrightarrow F_r=F
\]
with respect to $\sigma_{0,\beta_0}$, so that $\nu_{0,\beta_0}(F_{s+1}/F_s)=0$ (we take $F_{s+1}=F_s$ if $F$ has no HN factor of tilt slope 0). 
Then:
\begin{enumerate}[{(1)}]
    \item The same arguments of \autoref{HNbeta} provide a right weak limit filtration for $F/F_{s+1}$.
    \item The arguments of \autoref{HNbeta2} together with induction on $s$, provide a right weak limit filtration for $F_s$.
    Indeed, $Z_{0,\beta_0}$ is a stability function on $\Coh^{\beta_0}(X)$ since $\beta_0\notin\bQ$;
    hence all the technical construction in \autoref{kernel} and \autoref{HNbeta3} can be avoided.
    In particular, the core filtration of $F$ at $\beta_0$ directly gives $F_1\subset\ldots\subset F_s$.
\end{enumerate}

Combining these right weak limit filtrations with \autoref{slope0} (also valid under the assumption $\beta_0\in\bR\setminus\bQ$) applied to $F_{s+1}/F_s$, we obtain the polynomial expression $\rf^0_{F,L}(x)=\ch_2^{-x}(F_s)$ for all $x$ in a left neighborhood of $-\beta_0$.

In all the cases we know, objects have a Bridgeland limit filtration also at irrational numbers (see for instance the examples in \autoref{sec:chdGieseker}), but we do not know how to prove this in general.
This would imply that the Chern degree functions are piecewise polynomial, which cannot be directly deduced from \autoref{main}.\eqref{item:main1}.
\end{rem}

\section{Continuity and differentiability of the functions}
\label{sec:ExtensionToR}

\subsection{Extension as continuous real functions}
In this section we extend the Chern degree functions $\rf_{F,L}^k$ to continuous functions of real variable.
A similar result first appeared in \cite[section~4]{BPS:LinearSystems} for the \emph{continuous rank functions}, and was later generalized to the study of \emph{cohomological rank functions} in \cite[section~3]{JP}. 

We essentially follow this second approach, namely: one bounds the derivative of the functions, and then argues by integration.
Whereas our arguments to express the functions around rational numbers (\autoref{sec:LocalExpr}) were much longer than the ones in \cite{JP}, the control of the derivatives for this part is easier in the stability framework.

Following the strategy of \autoref{sec:LocalExpr}, we first consider the case of objects in the hearts $\Coh^\beta(X)$:

\begin{thm}\label{extension1}
Let $F\in\Coh^\beta(X)$ for some $\beta\in\bR$.
Then, the functions $\rf_{F,L}^0$ and $\rf_{F,L}^1$ extend to continuous functions on the interval $I_F=(-\mu_-(\cH^0(F)),-\mu_+(\cH^{-1}(F))]$.
\end{thm}
\begin{proof}
Notice that $I_F$ is (minus) the interval of \autoref{heart} delimiting where $F$ belongs to the heart; we have reversed signs for coherence with the definition of the functions.

First of all, we claim that we may restrict ourselves to the case where $I_F$ is bounded (i.e.~the numbers $\mu_+(\cH^{-1}(F))$ and $\mu_-(\cH^0(F))$ are both finite). 

This follows from the Serre vanishing of  \autoref{rk:BasicProp}.\eqref{enum:SerreVan}.
Indeed, on the one hand, if $I_F=(-\mu_-(\cH^0(F)),+\infty)$ is unbounded from the right, then $F$ is a coherent sheaf since $\cH^{-1}(F)$ is always torsion-free.
By Serre vanishing, there exists $x_0\in\bQ$ so that
 \[
 \rf^1_{F,L}(x)=0,\quad \rf^0_{F,L}(x)=\ch_2^{-x}(F)
 \]
for every rational $x\geq x_0$.
The extension of the functions is thus clear along the whole $[x_0,+\infty)$, so the problem is reduced to extend along $(-\mu_-(\cH^0(F)),x_0]$.
 
On the other hand, if $I_F=(-\infty,-\mu_+(\cH^{-1}(F))]$ is unbounded from the left, then $\cH^0(F)$ is torsion (or $0$).
This implies that the complex $F^\vee$ has at most two cohomology sheaves, which will be its cohomologies with respect to the heart $\Coh^\beta(X)$ for all $\beta\ll0$:
 \[
 \cH^1_{\beta}(F^\vee)=\cH^1(F^\vee), \quad \cH^2_{\beta}(F^\vee)=\cH^2(F^\vee).
 \]
 If $x_0\in\bQ$ is a bound ensuring Serre vanishing for the functions of the sheaves $\cH^1(F^\vee)$ and $\cH^2(F^\vee)$, by Serre duality \autoref{rk:BasicProp}.\eqref{enum:SerreDua} we have
 \[
 \rf^i_{F,L}(x)=\rf^{2-i}_{F^\vee,L}(-x)=\rf^0_{\cH^{2-i}(F^\vee),L}(-x)=\ch_2^x(\cH^{2-i}(F^\vee))
 \]
 for $i=0,1$ and every rational $x\leq-x_0$.
 Hence we only need to extend the functions along $(-x_0,-\mu_+(\cH^{-1}(F))]$.

Now take $\beta_0\in\bQ$ such that $F\in\Coh^{\beta_0}(X)$ (equivalently, $-\beta_0\in I_F\cap\bQ$), and let
\[
0=F_0\hookrightarrow F_1 \hookrightarrow \ldots \hookrightarrow F_{s-1} \hookrightarrow F_s \hookrightarrow F_{s+1} \hookrightarrow \ldots \hookrightarrow F_r=F
\]
be the HN filtration of $F$ with respect to $\sigma_{0,\beta_0}$, where $s=s(F)$ is the switching index of $F$.

Consider the subobject $G_s\subset F_s$, where $G_1\subset\ldots\subset G_s$ is the core filtration of $F$ at $\beta_0$.
Then, for all $x\in\bQ$ in a left neighborhood of $-\beta_0$ the function $\rf^0_{F,L}$ is polynomially expressed as
\[
\rf^0_{F,L}(x)=\ch_2^{-x}(G_s).\]


Therefore, 
$L\cdot\ch_1^{\beta_0}(G_s)$ is the left derivative of $\rf_{F,L}^0$ at $-\beta_0$.
Since this is the rank of $G_s$ in the heart $\Coh^{\beta_0}(X)$, it turns out that $0\leq D^-\rf_{F,L}^0(-\beta_0)\leq L\cdot\ch_1^{\beta_0}(F)$.

Summing up, the function $\rf^0_{F,L}$ has a left (and right) derivative at every $x\in I_F\cap\bQ$ (actually at every $x\in I_F\cap U$, where $U\subset\bR$ is an open subset containing $\bQ$), and they coincide almost everywhere.
Moreover, the derivatives are non-negative and bounded from above.
By integration, it follows that $\rf^0_{F,L}$ extends to a continuous function on the whole interval $I_F$.

The result for $\rf^1_{F,L}$ follows directly by defining $\rf_{F,L}^1(x)=\rf_{F,L}^0(x)-\ch_2^{-x}(F)$ for $x\in I_F$.
\end{proof}


\begin{rem}
It follows from the proof that the function $\rf^0_{F,L}$ (resp.~$\rf^1_{F,L}$) is non-decreasing (resp.~non-increasing) along the interval $I_F$.
This (a posteriori) explains \autoref{slope0}.
\end{rem}


Now we are ready to give a proof of \autoref{main}.\eqref{item:main2} in full generality:

\begin{proof}[Proof of \autoref{main}.\eqref{item:main2}]
Using the definition of $\rf_{E,L}^k$ this becomes an immediate consequence of \autoref{extension1}, provided that $\rf_{E,L}^k$ is continuous at the (finitely many) points $x=-\beta$ where the cohomologies $\cH^i_{\beta}(E)$ change.

And this continuity is guaranteed by \autoref{main}.\eqref{item:main1}, since such points of change are rational (they correspond to $\mu_L$-slopes of Harder-Narasimhan factors of cohomology sheaves of $E$).
\end{proof}

\subsection{Critical points}\label{sec:critical}
Let $\beta_0\in\bQ$.
Now we want to study when the functions $\rf^k_{F,L}$ attached to an object $F\in\Db(X)$ have a \emph{critical point} at $x=-\beta_0$, namely the functions are not of class $\cC^\infty$ at $-\beta_0$.

For simplicity, we will assume $F\in\Coh^{\beta_0}(X)$; 
our explicit analysis (see \autoref{critical} below) covers the case where the cohomologies of $F$ change at $-\beta_0$.
Still, in the general case, one also should take care of the eventual situation where $-\beta_0$ is a critical point for both $\rf^0_{\cH_{\beta_0}^k(F),L}$ and $\rf^1_{\cH_{\beta_0}^{k-1}(F),L}$, but $\rf^k_{F,L}$ is regular at $-\beta_0$.
Our approach may also detect this phenomenon, but the precise description of the general situation is tiresome and not so enlightening.

Let $0=F_0\hookrightarrow F_1\hookrightarrow\ldots\hookrightarrow F_r=F$ be the HN filtration of $F$ with respect to $\sigma_{0,\beta_0}$, satisfying
\[+\infty=\nu_{0,\beta_0}(F_1)>\ldots>\nu_{0,\beta_0}(F_s/F_{s-1})>0=\nu_{0,\beta_0}(F_{s+1}/F_s)>\ldots>\nu_{0,\beta_0}(F/F_{r-1}),
\]
that is, $s=s(F)$ is the switching index of $F$.
We write $F_1=0$ (resp.~$F_s=F_{s+1}$) if $F$ has no HN factor of tilt slope $+\infty$ (resp.~tilt slope $0$).
Then:

\begin{itemize}[\textbullet]
    \item As dictated by \autoref{heart}, $F\in\Coh^\beta(X)$ for all small enough $\beta>\beta_0$.
    This means that $F$ will only have two nonzero functions in a left neighborhood of $-\beta_0$, namely $\rf^0_{F,L}$ and $\rf^1_{F,L}$.
    
    \item For big enough $\beta<\beta_0$, $F$ may have two nontrivial cohomologies in $\Coh^\beta(X)$.
    More precisely, recall that by \autoref{Giesekerkernel} $\cH^{-1}(F_1)$ is a $\mu_L$-semistable sheaf of slope $\beta_0$, and $\cH^0(F_1)$ is a $0$-dimensional sheaf.
    Therefore, we have distinguished triangles
    \[
    \xymatrix@!C=10pt{
    0\ar[rr] && \cH^{-1}(F_1)[1]\ar[rr]\ar[dl]&&F_1\ar[dl]\ar[rr]&&F\ar[dl]\\
    &\cH^{-1}(F_1)[1]\ar[ul]&&\cH^0(F_1)\ar[ul]&&F/F_1\ar[ul]
    }
    \]

    \noindent with $\cH^{-1}(F_1),\cH^0(F_1),F/F_1\in\Coh^\beta(X)$ for $\beta<\beta_0$.
    This tells us that $\cH^{-1}_\beta(F)=\cH^{-1}(F_1)$ and $\cH^0_\beta(F)=F/(\cH^{-1}(F_1)[1])$ (where this quotient is taken in $\Coh^{\beta_0}(X)$).
    
    \noindent Moreover, note that the only nonzero Chern degree functions of $F$ in a right neighborhood of $-\beta_0$ may be $\rf^{-1}_{F,L}$, $\rf^{0}_{F,L}$ and $\rf^{1}_{F,L}$.
\end{itemize}

Now we consider the polynomial expressions for the functions in a left and a right neighborhood of $-\beta_0$, that we found in \autoref{sec:LocalExpr}.
Explicitly, there exists $\epsilon>0$ such that

\begin{equation}\label{eqn:notacioR}
\begin{split}
   \rf^{-1}_{F,L}(x)&=\left\{
    \begin{array}{c l}
     0 & -\beta_0-\epsilon\leq x\leq-\beta_0\\
     -\ch_2^x\left(\cH^{-1}(F_1)^\vee[1]/P\right) & -\beta_0\leq x\leq-\beta_0+\epsilon\\
    \end{array}
    \right.\\
   \rf^0_{F,L}(x)&=\left\{
    \begin{array}{c l}
     \ch_2^{-x}(G_s) & -\beta_0-\epsilon\leq x\leq-\beta_0\\
     \ch_2^x(P)-\ch_2^x\left((F/F_1)^\vee[1]/R_{s+1}\right)+\text{length}(\cH^0(F_1)) & -\beta_0\leq x\leq-\beta_0+\epsilon\\
    \end{array}
    \right.
\\
   \rf^1_{F,L}(x)&=\left\{
    \begin{array}{c l}
     -\ch_2^{-x}(F/G_s) & -\beta_0-\epsilon\leq x\leq-\beta_0\\
     \ch_2^x(R_{s+1}) & -\beta_0\leq x\leq-\beta_0+\epsilon\\
    \end{array}
    \right.
\end{split}
\end{equation}

where:

\begin{itemize}[\textbullet]
    \item The chain $G_1\subset\ldots\subset G_s$ (with $G_i\subset F_i$) is the core filtration of $F$ at $\beta_0$.
    
    \item $P=\left(\cH^{-1}(F_1)^\vee[1]\right)'$ is the core subobject of $\cH^{-1}(F_1)^\vee[1]$ at $-\beta_0$.
    
    \item The chain $R_{r-1}\subset\ldots\subset R_{s+1}$ in $\Coh^{-\beta_0}(X)$ (with $R_i\subset (F/F_i)^\vee[1]$) is the core filtration of $(F/F_1)^\vee[1]$ at $-\beta_0$.
    It is inductively constructed by letting $R_{r-1}=\left((F/F_{r-1})^\vee[1]\right)'$ and  $R_i/R_{i+1}=\left((F/F_i)^\vee[1]/R_{i+1}\right)'$.
\end{itemize}

A critical point for $\rf^{-1}_{F,L}$ arises whenever the object $\cH^{-1}(F_1)^\vee[1]/P$ is nonzero.
Since the regularity of $\rf^0_{F,L}$ can be deduced from that of $\rf^{-1}_{F,L}$ and $\rf^{1}_{F,L}$, we are left to compare the polynomial expressions for $\rf^1_{F,L}$.
To this end, we will use that $\ch_2^x(R_{s+1})=-\ch_2^{-x}(R_{s+1}^\vee[1])$ and we will exhibit a chain of morphisms connecting $R_{s+1}^\vee[1]$ with $F/G_s$.

By construction, there is a short exact sequence
\[
0\to R_{s+1}\to (F/F_{s+1})^\vee[1]\to Q\to 0
\]
in $\Coh^{-\beta_0}(X)$ with $Q\in\ker(Z_{0,-\beta_0})$; hence the object $Q$ is of the form $Q=S[1]$ for a $\mu_L$-semistable vector bundle $S$ with $\mu_L(S)=-\beta_0$ and $\odisc(S)=0$.
Dualizing, this yields a short exact sequence
\[
0\to S^\vee\to F/F_{s+1}\to R_{s+1}^\vee[1]\to 0
\]
in $\Coh^\beta(X)$, for big enough values of $\beta<\beta_0$.
We have obtained the following sequence of morphisms 
\[
F/G_s\to F/F_s\to F/F_{s+1}\to R_{s+1}^\vee[1]
\]
for which there exists $\epsilon'>0$ satisfying: $F/G_s\to F/F_s$ is a surjection in $\Coh^\beta(X)$ for $\beta\in[\beta_0,\beta_0+\epsilon')$, $F/F_s\to F/F_{s+1}$ is a surjection in $\Coh^\beta(X)$ for $\beta\in(\beta_0-\epsilon',\beta_0+\epsilon')$ and $F/F_{s+1}\to R_{s+1}^\vee[1]$ is a surjection in $\Coh^\beta(X)$ for $\beta\in(\beta_0-\epsilon',\beta_0)$.

With all this information, we get the following result:

\begin{prop}\label{critical} Let $F\in \Coh^{\beta_0}(X)$ and we keep the notation of \eqref{eqn:notacioR}.
\begin{enumerate}[{\rm (1)}]
    \item\label{cond1} The function $\rf^{-1}_{F,L}$ has a critical point at $x=-\beta_0$ if and only if $P\subsetneq \cH^{-1}(F_1)^\vee[1]$ in $\Coh^{-\beta_0}(X)$.
    In such a case, the function $\rf^{-1}_{F,L}$ is of class $\cC^1$ at $x=-\beta_0$.
    
    \item The function $\rf^1_{F,L}$ has a critical point at $x=-\beta_0$ if and only if $F/G_s\neq R_{s+1}^\vee[1]$. 
    This is equivalent to one of the following conditions holding:
    \begin{enumerate}[{\rm (a)}]
        \item\label{cond:a} $F_s\neq F_{s+1}$, namely $F$ has a HN factor (with respect to $\sigma_{0,\beta_0}$) of slope $\nu_{0,\beta_0}=0$.
        \item\label{cond:b} $G_s\neq F_s$
        \item\label{cond:c} $F/F_{s+1}\neq R_{s+1}^\vee[1]$
    \end{enumerate}
    Furthermore, $\rf^1_{F,L}$ is of class $\cC^1$ at $x=-\beta_0$ unless condition \eqref{cond:a} holds.
    
    \item The function $\rf^0_{F,L}$ has a critical point at $x=-\beta_0$ if and only if $\rf^{-1}_{F,L}$ or $\rf^1_{F,L}$ has a critical point.
    In such a case, $\rf^0_{F,L}$ is of class $\cC^1$ at $x=-\beta_0$ unless \eqref{cond:a} holds.
\end{enumerate}
\end{prop}

\begin{rem} Intuitively, we may think that the functions have critical points at $x=-\beta_0$ if and only if the hyperbola of one of its (possibly weak limit) HN factors at $\beta_0$ passes through the point $(0,\beta_0)$ of the $(\alpha,\beta)$-plane.
    
    \noindent For instance, condition \eqref{cond:b} is equivalent to $F_s$ having a left weak limit HN factor with tilt slope $\nu_{0,\beta}$ tending to 0 as $\beta\to\beta_0^+$.
\begin{enumerate}[{\rm (1)}]
    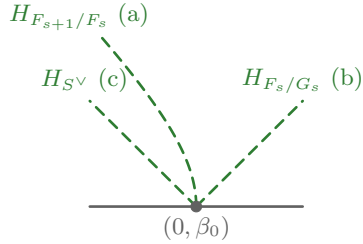
\begin{figure}[H]
\definecolor{ffffff}{rgb}{1,0,0}
\definecolor{sexdts}{rgb}{0.1803921568627451,0.49019607843137253,0.19607843137254902}
\definecolor{wrwrwr}{rgb}{0.3803921568627451,0.3803921568627451,0.3803921568627451}
\begin{tikzpicture}[line cap=round,line join=round,>=triangle 45,x=1.4cm,y=1.4cm]
\draw [line width=1pt,color=wrwrwr] (-1,0) -- (1,0);
\draw [line width=1pt,dash pattern=on 4pt off 3pt,color=sexdts] (0,0) -- (1,1);
\draw [line width=1pt,dash pattern=on 4pt off 3pt,color=sexdts] (0,0) -- (-1,1);
\draw[line width=1pt,dash pattern=on 4pt off 3pt,color=sexdts, smooth,samples=100,domain=0:10] 
plot[domain=0:1.25,variable=\t]({(-exp(\t)-exp(-\t))/2+1},{(exp(\t)-exp(-\t))/2});
\filldraw[color=wrwrwr] (0,0) circle (2pt); 

\begin{footnotesize}
\draw[color=sexdts] (-1.1,1.8) node {$\hyp{F_{s+1}/F_s}$ (a)};
\draw[color=sexdts] (1,1.2) node {$\hyp{F_s/G_s}$ (b)};
\draw[color=sexdts] (-1.05,1.2) node {$\hyp{S^\vee}$ (c)};
\draw[color=wrwrwr] (0,-0.2) node {$(0,\beta_0)$};
\end{footnotesize}
\end{tikzpicture}
\caption{Hyperbolas through $(0,\beta_0)$ producing a critical point for $\rf^0_{F,L}$ and $\rf^1_{F,L}$.}
\label{fig:hypcritical}
\end{figure}

    \item The case \eqref{cond1} is certainly exceptional, in the sense that it requires a (nontrivial) condition on one of the (finitely many) points where the cohomologies of $F$ with respect to the hearts $\Coh^\beta(X)$ change.
    
    \noindent On the other hand, whereas the case \eqref{cond:a} is naturally described in terms of $\sigma_{0,\beta_0}$-stability, condition \eqref{cond:b} (resp.~condition \eqref{cond:c}) requires the nontriviality of the core filtration of $F_s$ (resp.~$(F/F_{s+1})^\vee[1]$) at $\beta_0$ (resp.~$-\beta_0$).
    In particular, \eqref{cond:b} and \eqref{cond:c} require that $\sigma_{0,\beta_0}$ is not a Bridgeland stability condition.
    
    \item The possibilities \eqref{cond:a}, \eqref{cond:b} and \eqref{cond:c} producing critical points in the function $\rf^1_{F,L}$ may not be mutually exclusive. 
    Indeed,
    we will see below an example (\autoref{case:n=5}) where \eqref{cond:a} and \eqref{cond:c} hold simultaneously.
    It would be interesting to know whether conditions \eqref{cond:b} and \eqref{cond:c} may hold at the same time or not.
\end{enumerate}
Furthermore, the same study can be applied when $\beta_0$ is an irrational number, at which $F$ admits a Bridgeland limit HN filtration (recall \autoref{rem:IrratWithBrLimit}).

In such a case, $F\in\Coh^\beta(X)$ for some values $\beta<\beta_0$ as well, so $\rf^{-1}_{F,L}$ is identically 0 in an open neighborhood of $-\beta_0$.
Moreover, possibilities \eqref{cond:b} and \eqref{cond:c} must be excluded since $Z_{0,\beta_0}$ (resp.~$Z_{0,-\beta_0}$) is a stability function on $\Coh^{\beta_0}(X)$ (resp.~$\Coh^{-\beta_0}(X)$) when $\beta_0\notin\bQ$.
Therefore, the only possibility for a critical point is \eqref{cond:a}, which gives a point where the functions $\rf^0_{F,L}$ and $\rf^1_{F,L}$ are not differentiable.
\end{rem}

\section{The case of abelian surfaces}\label{sec:CaseCRF}

\subsection{Cohomological rank functions on abelian varieties} Let $(A,L)$ be a $g$-dimensional polarized abelian variety over $\bK$.
In their paper \cite{JP} Jiang and Pareschi attach to every $F\in\Db(A)$ and $i\in\bZ$ a \emph{cohomological rank function} $h^i_{F,L}:\bQ\to\bQ_{\geq0}$ defined as follows: given a rational number $x_0=\frac{a}{b}$ with $b>0$, then
\[
h^i_{F,L}(x_0):=\frac{1}{b^{2g}}h^i(\mu_b^*F\otimes M^{ab}\otimes\alpha)
\]
for general $\alpha\in\Pic^0(A)$, where $\mu_b:A\to A$ is the multiplication-by-$b$ isogeny and $M\in\Pic(A)$ is any ample line bundle representing the polarization $L$.

Since $\mu_b^*L=b^2L$ (hence $\mu_b^*(x_0L)=abL$) and $\deg\mu_b=b^{2g}$, the number $h^i_{F,L}(x_0)$ gives a meaning to the (hyper)cohomological rank $h^i(F\otimes L^{x_0})$ of $F$ twisted with the general representative of the fractional polarization $x_0L$.

\begin{rem}The definition in \cite{JP} is given under the assumption $\Char\bK=0$, but the same definition works in arbitrary characteristic as observed in \cite[Section~2]{Caucci}.
\end{rem}

The main results of Jiang and Pareschi about these functions can be summarized as:

\begin{thm}\label{JPmain}
Let $F\in\Db(A)$ be an object and $i\in\bZ$.
Then:
\begin{enumerate}[{\rm (1)}]
    \item\label{enum:JP1} \cite[Corollary~2.6]{JP} For every $x_0\in\bQ$, there exists a left (resp.~right) neighborhood of $x_0$ where the function $h^i_{F,L}$ is given by an explicit polynomial.
    \item\label{enum:JP2} \cite[Theorem~3.2]{JP} If $\Char\bK=0$, the function $h^i_{F,L}$ extends to a continuous function of real variable $h^i_{F,L}:\bR\to\bR_{\geq0}$.
\end{enumerate}
\end{thm}

It is expected (\cite[Remark~2.8]{JP}) that these real functions are piecewise polynomial; in other words, that their critical points do not accumulate towards an irrational number.

In the case of elliptic curves, cohomological rank functions admit a precise description in terms of $\mu_L$-stability.
The main point is that $\mu_L$-semistable coherent sheaves have \emph{trivial} functions, that is, the support of any of its functions is disjoint with the support of all the other functions.
This is certainly well known to the experts, but we include a proof since we could not find a reference.

\begin{prop}\label{CRFelliptic}
Let $(E,L)$ be an elliptic curve endowed with a polarization of degree 1. 
\begin{enumerate}[{\rm (1)}]
 \item\label{enum:elliptic1} If $F\in\Coh(E)$ is a $\mu_L$-semistable coherent sheaf, then it has trivial functions
 \[
 h^0_{F,L}(x)=\left\{
 \begin{array}{c l}
 0 & \chi_{F,L}(x)\leq 0\\
 \chi_{F,L}(x) & \chi_{F,L}(x)\geq0\\
 \end{array}
 \right.
 \;\;\;\;\;
 h^1_{F,L}(x)=\left\{
 \begin{array}{c l}
 -\chi_{F,L}(x) & \chi_{F,L}(x)\leq 0\\
 0 & \chi_{F,L}(x)\geq0\\
 \end{array}
 \right.
 \]
 where $\chi_{F,L}(x)=\rk(F)\cdot x+\deg(F)$ is the Hilbert polynomial of $F$ with respect to $L$.

 \item\label{enum:elliptic2} Let $0=F_0\hookrightarrow F_1\hookrightarrow\ldots\hookrightarrow F_r=F$ be the HN filtration of a coherent sheaf $F$.
 Then the functions of $F$ can be recovered from those of its HN factors:
 \[
 h^0_{F,L}(x)=\displaystyle\sum_{k=1}^r\left(h^0_{F_k/F_{k-1},L}(x)\right)=\displaystyle\sum_{\chi_{F_k/F_{k-1},L}(x)\geq0}\left(\chi_{F_k/F_{k-1},L}(x)\right)
 \]
 \[
 h^1_{F,L}(x)=\displaystyle\sum_{k=1}^r\left(h^1_{F_k/F_{k-1},L}(x)\right)=\displaystyle\sum_{\chi_{F_k/F_{k-1},L}(x)\leq0}\left(-\chi_{F_k/F_{k-1},L}(x)\right)
 \]

\item\label{enum:elliptic3} For any $ F\in\Db(E)$ and $i\in\bZ$, 
\[
h^i_{F,L}(x)=h^0_{\cH^i(F),L}(x)+h^1_{\cH^{i-1}(F),L}(x).
\]
\end{enumerate}

\end{prop}
\begin{proof}
Item \eqref{enum:elliptic1} is clear if $F$ is torsion, since in that case $h^0_{F,L}(x)=\chi_{F,L}(x)=\text{length}(F)$ for every $x\in\bQ$.
Hence we may assume that $F$ is a vector bundle.

Let $x=\frac{a}{b}\in\bQ$.
If $\chi_{F,L}(x)<0$ (resp.~$\chi_{F,L}(x)>0$), then we consider a non-decreasing (resp.~non-increasing) sequence $\{x_n=\frac{a_n}{b_n}\}_n\subset\bQ$ converging to $x$, such that for every $n$ the multiplication isogeny $\mu_{b_n}:E\to E$ is an étale morphism\footnote{Of course, if $b$ is not divisible by $\Char\bK$ (e.g., if $\Char\bK=0$) one can take $x_n=x$ for every $n$}.

We claim that for every degree 1 line bundle $M$ and every $n\in\bN$, one has
\[
0=\Hom(M^{-a_nb_n},\mu_{b_n}^*F)=H^0(\mu_{b_n}^*F\otimes M^{a_nb_n})
\]
\[
\text{(resp.~}0=\Hom(\mu_{b_n}^*F,M^{-a_nb_n})=\Ext^1(M^{-a_nb_n},\mu_{b_n}^*F)^*=H^1(\mu_{b_n}^*F\otimes M^{a_nb_n})^*\;)
\]

Indeed, $\mu_{b_n}^*F$ is $\mu_L$-semistable (we can apply \cite[Lemma~3.2.2]{HuybrechtsLehn}, since $\mu_{b_n}$ is a separable isogeny) as well as $M^{-a_nb_n}$.
Thus the claim follows from the inequality $\mu_L(M^{-a_nb_n})=-a_nb_n>b_n^2\mu_L(F)=\mu_L(\mu_{b_n}^*F)$ (resp.~$\mu_L(M^{-a_nb_n})=-a_nb_n<b_n^2\mu_L(F)=\mu_L(\mu_{b_n}^*F)$).

Therefore $h^0_{F,L}(x_n)=0$ (resp.~$h^1_{F,L}(x_n)=0$) for every $n$, which by \autoref{JPmain}.\eqref{enum:JP1} implies $h^0_{F,L}(x)=0$ (resp.~$h^1_{F,L}(x)=0$).
This proves \eqref{enum:elliptic1}.

Item \eqref{enum:elliptic2} follows by induction on the length $r$ of the HN filtration of $F$, the initial case being nothing but \eqref{enum:elliptic1}.
Let $x=\frac{a}{b}\in\bQ$.
For the induction step one uses the long exact sequence in cohomology associated to
\[
0\to \mu_b^*F_{r-1}\otimes M^{ab}\to \mu_b^*F_r\otimes M^{ab}\to \mu_b^*(F_r/F_{r-1})\otimes M^{ab}\to 0
\]
for every line bundle $M$ of degree 1, together with the observation that
\[
\chi_{F_k/F_{k-1},L}(x)>(<)0\Longleftrightarrow x>(<)-\mu_L(F_k/F_{k-1})
\]
for any $k\in\{1,\ldots,r\}$ and the inequalities $-\mu_L(F_1)<\ldots<-\mu_L(F_r/F_{r-1})$.

For the proof of \eqref{enum:elliptic3}, write $x=\frac{a}{b}\in\bQ$ and let $M$ be any line bundle of degree $1$.
Considering the distinguished triangle in $\Db(E)$ obtained by truncation of $\mu_b^*F\otimes M^{ab}$
\[
\mu_b^*(\tau_{\leq i-1}F)\otimes M^{ab}\to \mu_b^*F\otimes M^{ab}\to \mu_b^*(\tau_{\geq i}F)\otimes M^{ab}
\]
and the corresponding long exact sequence of hypercohomology groups
\begin{gather*}
\ldots\to\mathbb{H}^{i-1}(\mu_b^*(\tau_{\geq i}F)\otimes M^{ab})\to \mathbb{H}^{i}(\mu_b^*(\tau_{\leq i-1}F)\otimes M^{ab})\to \mathbb{H}^{i}(\mu_b^*F\otimes M^{ab})\to \\
\to\mathbb{H}^{i}(\mu_b^*(\tau_{\geq i}F)\otimes M^{ab})\to\mathbb{H}^{i+1}(\mu_b^*(\tau_{\leq i-1}F)\otimes M^{ab})\to\ldots
\end{gather*}
the result becomes a consequence of the following immediate equalities:
\begin{gather*}
\mathbb{H}^{i-1}(\mu_b^*(\tau_{\geq i}F)=0,\;\;\mathbb{H}^{i}(\mu_b^*(\tau_{\leq i-1}F)\otimes M^{ab})=H^1(\mu_b^*(\cH^{i-1}F)\otimes M^{ab}),\\
\mathbb{H}^{i}(\mu_b^*(\tau_{\geq i}F)\otimes M^{ab})=H^0(\mu_b^*(\cH^{i}F)\otimes M^{ab}),\;\;
\mathbb{H}^{i+1}(\mu_b^*(\tau_{\leq i-1}F)\otimes M^{ab})=0\qedhere
\end{gather*}
\end{proof}

\subsection{Proof of \autoref{mainCRF}} 
In this part we will prove that the Chern degree functions $\rf^k_{F,L}$ attached to any object $F\in\Db(X)$ on a polarized surface $(X,L)$ recover, in the case where $X$ is an abelian surface, the cohomological rank functions $h^k_{F,L}$ of Jiang and Pareschi.

The key point of the proof are the following two lemmas.
The first one is the analogue to the fact that a coherent sheaf on an elliptic curve only may have $h^0$ and $h^1$ as nonzero functions. 

\begin{lem}\label{heartlemma}
If $F\in\Coh^\beta(X)$ for a number $\beta\in\bQ$, then $h^i_{F,L}(-\beta)=0$ for every $i\neq0,1$.
\end{lem}

\begin{proof}
Since $F$ is a complex with at most two nontrivial cohomology sheaves (namely $\cH^{-1}(F)$ and $\cH^0(F)$), it turns out that $h^i_{F,L}(-\beta)=0$ for every $i\notin\set{-1,0,1,2}$.
Moreover, we have
\[
h^{-1}_{F,L}(-\beta)=h^0_{\cH^{-1}(F),L}(-\beta),\;\;h^{2}_{F,L}(-\beta)=h^2_{\cH^0(F),L}(-\beta)
\]
so it suffices to check that $h^2_{E,L}(-\beta)=0$
whenever $E\in\cT_\beta$, and $h^0_{G,L}(-\beta)=0$ whenever $G\in\cF_\beta$.

To prove the first vanishing, we consider a non-decreasing sequence of rational numbers $\beta_n=\frac{a_n}{b_n}$ converging to $\beta$, such that $\mu_{b_n}$ is a separable isogeny for every $n$.
Let $0=E_0\hookrightarrow E_1\hookrightarrow\ldots\hookrightarrow E_r=E$ be the HN filtration of $E$ with respect to $\mu_L$-stability.
Since torsion sheaves are always $\mu_L$-semistable (they have slope $\mu_L=+\infty$), it follows from \cite[Lemma~3.2.2]{HuybrechtsLehn} that
\[
0=\mu_{b_n}^*E_0\hookrightarrow \mu_{b_n}^*E_1\hookrightarrow\ldots\hookrightarrow \mu_{b_n}^*E_r=\mu_{b_n}^*E
\]
is a HN filtration for $\mu_{b_n}^*E$ with respect to $\mu_L$-stability, for every $n$.
Observe that we have
\[
\mu_L(\mu_{b_n}^*(E_i/E_{i-1}))=b_n^2\mu_L(E_i/E_{i-1})>b_n^2\cdot\beta_n=a_nb_n
\]
for every $i\in\{1,\ldots,r\}$, thanks to the condition $E\in\cT_{\beta_n}$ inherited from $E\in\cT_{\beta}$.
Therefore
\[
0=\Hom(\mu_{b_n}^*(E_i/E_{i-1}),L^{a_nb_n})=\Ext^2(L^{a_nb_n},\mu_{b_n}^*(E_i/E_{i-1}))^*=H^2(\mu_{b_n}^*(E_i/E_{i-1})\otimes L^{-a_nb_n})^*
\]
for every $i$, and the equality $h^2_{E,L}(-\beta_n)=0$ follows from an easy induction on the length $r$ of the HN filtration.
This is enough to prove $h^2_{E,L}(-\beta)=0$, thanks to \autoref{JPmain}.\eqref{enum:JP1}.

For the second vanishing, we will check that $h^0_{G,L}(x)=0$ for every rational number $x=\frac{c}{d}$ with $x<-\beta$.
Again, this is more than enough for our purposes thanks to \autoref{JPmain}.\eqref{enum:JP1}.

We approach $x$ by a non-decreasing sequence $x_n=\frac{c_n}{d_n}$ of rational numbers such that the multiplication $\mu_{d_n}$ is étale.
As before, the HN filtration $0=G_0\hookrightarrow G_1\hookrightarrow\ldots\hookrightarrow G_r=G$ of $G$ in $\mu_L$-stability induces the HN filtration $0=\mu_{d_n}^*G_0\hookrightarrow \mu_{d_n}^*G_1\hookrightarrow\ldots\hookrightarrow \mu_{d_n}^*G_r=\mu_{d_n}^*G$
of $\mu_{d_n}^*G$, for every $n$.

The hypothesis $G\in\cF_\beta$ says that
\[
\mu_L(\mu_{d_n}^*(G_i/G_{i-1}))=d_n^2\mu_L(G_i/G_{i-1})\leq d_n^2\beta<-d_n^2x_n=-c_nd_n
\]
for every $i$, which implies
\[
0=\Hom(L^{-c_nd_n},\mu_{d_n}^*(G_i/G_{i-1}))=H^0(\mu_{d_n}^*(G_i/G_{i-1})\otimes L^{c_nd_n})
\]
The equality $h^0_{G,L}(x_n)=0$ is again obtained by induction on $r$, and then $h^0_{G,L}(x)=0$ follows. 
\end{proof}

The second lemma is the analogue of \autoref{CRFelliptic}.\eqref{enum:elliptic1}, namely that at a fixed point, at most one function is nonzero for a semistable sheaf on an elliptic curve:

\begin{lem}\label{vanishinglemma}
If $F\in\Coh^\beta(X)$ ($\beta\in\bQ$) is $\sigma_{0,\beta}$-semistable, then $h^1_{F,L}(-\beta)=0$ (resp.~$h^0_{F,L}(-\beta)=0$) if $\nu_{0,\beta}(F)\geq0$ (resp.~if $\nu_{0,\beta}(F)\leq0$).
\end{lem}
\begin{proof}
First of all, observe that the same arguments of \cite[Proposition~6.1]{BMS} yield that $\mu_b^*F\in\Coh^{b^2\beta}(X)$ and it is a $\sigma_{0,b^2\beta}$-semistable object, for every $b\in\bZ_{>0}$ such that $\mu_b$ is a separable isogeny\footnote{The condition $\beta\in\bQ$ is required at this point, to ensure the existence of HN filtrations with respect to $\sigma_{0,\beta}$ and $\sigma_{0,d^2\beta}$.}.
Moreover, we have
\[
\nu_{0,b^2\beta}(\mu_b^*F)=b^2\nu_{0,\beta}(F)
\]
as follows from $\ch(\mu_b^*F)=(\ch_0(F),b^2\ch_1(F),b^4\ch_2(F))$.

If $\nu_{0,\beta}(F)\geq0$, we will prove that $h^1_{F,L}(x)=0$ for every $x\in\bQ$ with $x>-\beta$; then, $h^1_{F,L}(-\beta)=0$ will follow again from \autoref{JPmain}.\eqref{enum:JP1}.
To this end, we approach $x$ by a non-increasing sequence $x_n=\frac{c_n}{d_n}$ such that $\mu_{d_n}$ is separable. 

Observe that the condition $\frac{c_n}{d_n}>-\beta$ reads $-c_nd_n<d_n^2\beta$.
Therefore, $L^{-c_nd_n}[1]$ is $\sigma_{0,d_n^2\beta}$-semistable with $\nu_{0,d_n^2\beta}(L^{-c_nd_n}[1])<0\leq\nu_{0,d_n^2\beta}(\mu_{d_n}^*F)$, which gives
\[
0=\Hom(\mu_{d_n}^*F,L^{-c_nd_n}[1])=\Ext^1(\mu_{d_n}^*F,L^{-c_nd_n})=\Ext^1(L^{-c_nd_n},\mu_{d_n}^*F)^*=H^1(\mu_{d_n}^*F\otimes L^{c_nd_n})^*
\]
and thus $h^1_{F,L}(x_n)=0$.
It follows that $h^1_{F,L}(x)=0$, as desired.

If $\nu_{0,\beta}(F)\leq0$, following the same strategy it suffices to check that $h^0_{F,L}(x)=0$ for every rational $x=\frac{c}{d}$ with $x<-\beta$ and $\mu_d$ étale.
And indeed, $\frac{c}{d}<-\beta$ reads $-cd>d^2\beta$, so $L^{-cd}$ is $\sigma_{0,d^2\beta}$-semistable with $\nu_{0,d^2\beta}(L^{-cd})>0\geq\nu_{0,d^2\beta}(\mu_d^*F)$.
This implies
\[
0=\Hom(L^{-cd},\mu_d^*F)=H^0(\mu_d^*F\otimes L^{cd})
\]
and therefore $h^0_{F,L}(x)=0$.
\end{proof}

We are now ready to prove that, on abelian surfaces, cohomological rank functions are recovered by Chern degree functions:

\begin{proof}[Proof of \autoref{mainCRF}]
Given $F\in\Db(X)$ and $k\in\bZ$, we will prove the equality $\rf^k_{F,L}(-\beta)=h^k_{F,L}(-\beta)$ for every $\beta\in\bQ$.

We start with the basic case where $F\in\Coh^\beta(X)$ is $\sigma_{0,\beta}$-semistable.
Assume that $\nu_{0,\beta}(F)\geq0$
Then, on the one hand, by definition of the functions $\rf^k_{F,L}$, we have $\rf^k_{F,L}(-\beta)=0$ for every $k\neq0$ and $\rf^0_{F,L}(-\beta)=\ch_2^{\beta}(F)$.
On the other hand, \autoref{heartlemma} and \autoref{vanishinglemma} give $h^k_{F,L}(-\beta)=0$ for every $k\neq0$, and thus $h^0_{F,L}(-\beta)=\chi_{F,L}(-\beta)$.

Since the polynomial (in $x$) $\ch_2^{-x}(F)$ equals the Hilbert polynomial $\chi_{F,L}(x)$, it follows that $\rf^k_{F,L}(-\beta)=h^k_{F,L}(-\beta)$ for every $k$.

If $F\in\Coh^\beta(X)$ is $\sigma_{0,\beta}$-semistable with $\nu_{0,\beta}(F)\leq0$, then the same arguments yield $\rf^k_{F,L}(-\beta)=0=h^k_{F,L}(-\beta)$ for $k\neq1$ and $\rf^1_{F,L}(-\beta)=-\ch_2^{\beta}(F)=h^1_{F,L}(-\beta)$.

When $F\in\Coh^\beta(X)$ is an arbitrary object (not necessarily $\sigma_{0,\beta}$-semistable), the result follows by induction on the length of the HN filtration of $F$ with respect to $\sigma_{0,\beta}$, arguing similarly to the proof of \autoref{CRFelliptic}.\eqref{enum:elliptic2}.

To prove the result for a general $F\in\Db(X)$, we approach $-\beta$ by a non-decreasing sequence $-\beta_n=\frac{a_n}{b_n}$ of rational numbers such that the multiplication maps $\mu_{b_n}$ are étale.
If $\tau_{\leq k-1}^{\beta_n}$, $\tau_{\geq k}^{\beta_n}$ denote the truncation functors of the bounded t-structure defined by $\Coh^{\beta_n}(X)$, then one immediately checks (again, using \cite[Proposition~6.1.a]{BMS}) that
\[
\tau_{\leq k-1}^{b_n^2\beta_n}\circ\mu_{b_n}^*=\mu_{b_n}^*\circ\tau_{\leq k-1}^{\beta_n}, \;\;\;\tau_{\geq k}^{b_n^2\beta_n}\circ\mu_{b_n}^*=\mu_{b_n}^*\circ\tau_{\geq k}^{\beta_n}
\]
for every $n$.
Therefore, we have a distinguished triangle in $\Db(X)$
\[
\mu_{b_n}^*(\tau_{\leq k-1}^{\beta_n} F)\otimes L^{a_nb_n}\to \mu_{b_n}^*F\otimes L^{a_nb_n}\to \mu_{b_n}^*(\tau_{\geq k}^{\beta_n} F)\otimes L^{a_nb_n}
\]
Arguing with its associated long exact sequence of hypercohomology groups as in the proof of \autoref{CRFelliptic}.\eqref{enum:elliptic3} (and using that the assertion has already been proved for objects of $\Coh^{\beta_n}(X)$), we obtain $\rf^k_{F,L}(-\beta_n)=h^k_{F,L}(-\beta_n)$ for every $n$.
Then the equality $\rf^k_{F,L}(-\beta)=h^k_{F,L}(-\beta)$ is a consequence of \autoref{main}.\eqref{item:main1} and \autoref{JPmain}.\eqref{enum:JP2}.
\end{proof}

This description of cohomological rank functions on abelian surfaces establishes a clear analogy with the case of elliptic curves.
In particular, the proof shows that the cohomological rank functions of an object $F\in\Db(X)$ at $x=-\beta$ split into simpler pieces, corresponding to its HN factors with respect to $\sigma_{0,\beta}$.
The main difference is that, in the case of elliptic curves, the study via $\mu_L$-stability is actually global and proves that cohomological rank functions are piecewise polynomial, with all their critical points being rational.
In dimension $2$ the study is strictly local as we saw in \autoref{sec:LocalExpr}, which makes the situation much richer.

To finish this section, it is worth mentioning the following immediate consequence of this new presentation in terms of Chern degree functions, which is a refinement of \cite[Lemma~6.1]{JP} for the case of abelian surfaces:

\begin{cor}
If $(X,L)$ is a polarized abelian surface and $F\in\Db(X)$, any local polynomial expression of the cohomological rank function $h^k_{F,L}$ has integral coefficients.
\end{cor}

\section{Chern degree functions of Gieseker semistable sheaves}
\label{sec:chdGieseker}

In order to illustrate with examples the previous account, let us discuss briefly some properties of Chern degree functions of ($(L,-\frac{1}{2}K_X)$-twisted) Gieseker semistable sheaves.
We will mainly focus on Gieseker semistable sheaves on abelian surfaces, where these properties will become properties of the corresponding cohomological rank functions.

So we fix a polarized surface $(X,L)$ and a torsion-free Gieseker semistable sheaf $F$.
Note that $F\in\Coh^\beta(X)$ (resp.~$F[1]\in\Coh^\beta(X)$) for every $\beta<\mu_L(F)$ (resp.~$\beta\geq\mu_L(F)$), hence $\rf^2_{F,L}(x)=0$ (resp.~$\rf^0_{F,L}(x)=0$) for $x\geq-\mu_L(F)$ (resp.~$x\leq-\mu_L(F)$).

Moreover, if $\beta<\mu_L(F)$ then $F$ is $\sigma_{\alpha,\beta}$-semistable for all $\alpha\gg0$ (\autoref{Giesekerchamber}), so the problem of computing $\rf^0_{F,L}$ and $\rf^1_{F,L}$ in $(-\mu_L(F),+\infty)$ consists of studying how the trivial HN filtration of $F$ in the Gieseker chamber varies as we reach the line $\alpha=0$.

\subsection{Trivial Chern degree functions}\label{subsec:Trivialchd}
We will call a Chern degree function \emph{trivial} if its support is disjoint with the support of all the other functions attached to the same object.
In terms of stability, this reads as the simplest possible situation:

\begin{prop}\label{trivial}
The function $\rf^0_{F,L}$ is trivial if, and only if, $F$ is $\sigma_{\alpha,\beta}$-semistable for every $\alpha>0$ and $\beta<\mu_L(F)$.
In such a case,
\[
\rf^0_{F,L}(x)=\left\{
 \begin{array}{c l}
 0 & x\leq -p_F\\
 \ch_2^{-x}(F) & x\geq-p_F\\
 \end{array}
 \right.
\]
where 
$-p_F$ is the largest root of the Chern degree polynomial $\ch_2^{-x}(F)$ (recall also \autoref{hyperbola}).
\end{prop}

\begin{proof}
If the semistability assumption on $F$ is fulfilled, $F$ is $\sigma_{0,\beta}$-semistable for every $\beta<\mu_L(F)$ and thus $\rf^0_{F,L}(-\beta)$ is simply defined according to the sign of the tilt slope $\nu_{0,\beta}(F)$.

Conversely, assume that the function $\rf^0_{F,L}$ is trivial.
If $F$ is not $\sigma_{\alpha,\beta}$-semistable for every $\alpha>0$ and $\beta<\mu_L(F)$, there exists an actual wall $W$ (intersecting $H_F$ at its top point) along which $F$ destabilizes.
Let $0\to E\to F\to Q\to 0$ define this wall. 

Then, for every $\beta\in(p_F,p_E)$ $E$ is a subobject of $F$ in $\Coh^\beta(X)$ with $\nu_{0,\beta}(E)>0$, which gives $\rf^0_{F,L}(x)>0$ for $x\in(-p_E,-p_F)$.
This contradicts the triviality of $\rf^0_{F,L}$, since for $-\mu_L(F)<x<-p_F$ we have $\rf^1_{F,L}(x)>0$ due to $\ch_2^{-x}(F)<0$ and the relation $\rf^0_{F,L}(x)-\rf^1_{F,L}(x)=\ch_2^{-x}(F)$.
\end{proof}

\begin{ex}\label{extrivial}
Let us give some examples of trivial Chern degree function $\rf^0_{F,L}$.
\begin{enumerate}[{\rm (1)}]
    
    \item\label{enum:odisc} Gieseker semistable sheaves with $\odisc(F)=0$.
    These are the only examples of Gieseker semistable sheaves where the function $\rf^0_{F,L}$ is trivial and of class $\cC^1$ at their critical point $-p_F$, according to \autoref{critical} (see also \autoref{mainCritical}).

    \item Gieseker semistable sheaves with $\Delta(F)+C_L(L\cdot\ch_1(F))=0$, where $\Delta=(\ch_1)^2-2\ch_0\cdot\ch_2$ and $C_L$ is the constant of \cite[Lemma~3.3]{BMS}.
    These objects are more general than the objects considered in \eqref{enum:odisc}.
    For example, for abelian surfaces, where one can choose $C_L$ to be zero, the objects in \eqref{enum:odisc} only recover semihomogeneous vector bundles with determinant proportional to $L$, while here we are considering all semihomogeneous vector bundles.
    
    \item\label{enum:idealpoint} The ideal sheaf $\mathcal{I}_p$ of a point $p$ on a principally polarized abelian surface is well-known to have a trivial $h^0$ function, according to the analysis in \cite[section~8]{JP}.
    Conversely, as observed in \cite[section~4.1]{Meachan} $\mathcal{I}_p$ has no actual wall for $\beta<0$.
    
    \item\label{enum:AbelJacobi} Let $i:C\hookrightarrow X$ be an Abel-Jacobi embedding of a smooth curve $C$ of genus 2 inside its (principally polarized) Jacobian $X=JC$, and let $F=i_*M$ for a line bundle $M$ of odd degree on $C$.
    
    \noindent In this case $F$ is torsion, but the same analysis works as in the torsion-free case thanks to \cite{bayerli}.
    Then, as explained in \cite[Example~4.3]{JP} the function $h^0_{F,L}$ is trivial; according to \autoref{trivial}, it follows that $F$ is semistable on the whole $(\alpha,\beta)$-plane.
\end{enumerate}
\end{ex}

\subsection{Chern degree functions of semistable sheaves of low discriminant}\label{subsec:lowDisc}
There are other situations in which $\rf^0_{F,L}$, even if not trivial, can be explicitly described.
To this end we consider the \emph{minimal discriminant}, defined as the positive generator $m$ of the ideal of $\bZ$ generated by $\{\odisc(v)\mid v\in\Lambda, \odisc(v)>0\}$.

Note that the quantity $m$ essentially depends on the intersection pairing of $\NS(X)$.
For instance, if $\NS(X)=\bZ\cdot L$ (or more generally, if $L^2|L\cdot M$ for every $M\in\NS(X)$) then $m$ is a multiple of $L^2$.
Of course, in full generality one can only assert $m\in\bZ_{\geq1}$.


Let us assume now that $\odisc(F)=m$. 
Then either the function $\rf^0_{F,L}$ is trivial (as happens in \autoref{extrivial}.\eqref{enum:idealpoint}-\eqref{enum:AbelJacobi}), or $F$ destabilizes along an actual wall $W$ defined by an exact sequence $0\to E\to F\to Q\to 0$ with $\odisc(E)=0=\odisc(Q)$:

\begin{figure}[ht]
\definecolor{ffffff}{rgb}{1,1,1}
\definecolor{sexdts}{rgb}{0.1803921568627451,0.49019607843137253,0.19607843137254902}
\definecolor{wrwrwr}{rgb}{0.3803921568627451,0.3803921568627451,0.3803921568627451}
\begin{tikzpicture}[line cap=round,line join=round,>=triangle 45,x=2.20cm,y=2.20cm]
\draw [line width=1pt,color=red] plot[domain=0:3.141592653589793,variable=\t]({cos(\t r)},{sin(\t r)});
\draw [line width=1pt,color=wrwrwr] (-2,0) -- (2,0);
\draw[line width=1pt,dash pattern=on 4pt off 3pt,color=sexdts, smooth,samples=100,domain=0:10] 
plot[domain=-1:0.509461,variable=\t]({\t},{\t+1});
\draw[line width=1pt,dash pattern=on 4pt off 3pt,color=sexdts, smooth,samples=100,domain=0:10] 
plot[domain=0:0.71,variable=\t]({(-exp(\t)-exp(-\t))/2+1.118034},{(exp(\t)-exp(-\t))});

\draw[line width=1pt,dash pattern=on 4pt off 3pt,color=sexdts, smooth,samples=100,domain=0:10] 
plot[domain=-0.509461:1,variable=\t]({\t},{1-\t});

\begin{footnotesize}
\draw[color=red] (0.85,.75) node {$W$};
\draw[color=sexdts] (0.68,1.5) node {$\hyp{Q}$};
\draw[color=sexdts] (0.05,1.5) node {$\hyp{F}$};
\draw[color=sexdts] (-0.68,1.5) node {$\hyp{E}$};
\draw[color=wrwrwr] (-1,-0.15) node {$p_Q$};
\draw[color=wrwrwr] (1,-0.15) node {$p_E$};
\draw[color=wrwrwr] (0.14,-0.15) node {$p_F$};
\end{footnotesize}
\end{tikzpicture}
\caption{}
\label{fig:treem}
\end{figure}

In the latter case, $E$ and $Q$ can only be destabilized at their vertical walls $\beta=p_E$ and $\beta=p_Q$, so it follows that $E$ and $Q$ are the $\sigma_{0,\beta}$-HN factors of $F$ for all $\beta\in(p_Q,p_E)$.
Clearly $F$ is $\sigma_{0,\beta}$-semistable for $\beta\leq p_Q$ and $p_E\leq\beta<\mu_L(F)$, so whenever nontrivial the function $\rf^0_{F,L}$ reads
\[
\rf^0_{F,L}(x)=\left\{
 \begin{array}{c l}
 0 & x\leq -p_E\\
 \ch_2^{-x}(E) & -p_E\leq x\leq-p_Q\\
 \ch_2^{-x}(F) & x\geq-p_Q\\
 \end{array}
 \right.
\]
and, according to the description of \autoref{critical}, the function is $\cC^1$ at $-p_E$ and $-p_Q$.

\vspace{.4cm}

\begin{ex}
Assume that $(X,L)$ is a $(1,2)$-polarized abelian surface with $\NS(X)=\bZ\cdot L$.
Under these assumptions the linear system $|L|$ has exactly four base points, that are identified by the polarization map $\varphi_L:X\to\Pic^0(X)$.
In other words, for any point $p\in X$ there exists a unique $\alpha\in\Pic^0(X)$ such that $H^1(\cI_p\otimes L\otimes\alpha)\neq0$.

By Serre duality, there is a non-trivial extension $0\to L^{-1}\otimes\alpha^{-1}\to E\to\cI_p\to 0$,
which after rotation defines an actual wall for $\cI_p$. 
Hence, the cohomological rank function $h^0_{\cI_p,L}$ is
\[
h^0_{\cI_p,L}(x)=\left\{
 \begin{array}{c l}
 0 & x\leq \frac{1}{2}\\
 \chi_{E,L}(x)=4x^2-4x+1 & \frac{1}{2}\leq x\leq1\\
 \chi_{\cI_p,L}(x)=2x^2-1 & x\geq 1.\\
 \end{array}
 \right.
\]
\end{ex}

\vspace{.4cm}

Now assume $\odisc(F)=2m$.
If $\rf^0_{F,L}$ is not trivial, then $F$ destabilizes along a wall $W$.
We can choose a destabilizing sequence $0\to E\to F\to Q\to 0$ defining the HN filtration of $F$ in a small annulus just below $W$.
Then, we have to distinguish several possibilities.

If $\odisc(E)=\odisc(Q)=0$, $E$ and $Q$ are semistable for all the stability conditions in the interior of $W$.
Thus the function $\rf^0_{F,L}$ admits the same description as the one given in the case $\odisc(F)=m$, i.e.~ it has critical points $-p_E$ and $-p_Q$ (see for instance \autoref{case:n=2}).

The other possibility is that either $E$ or $Q$ has positive discriminant, say $\odisc(E)=m$ and $\odisc(Q)=0$.
Note that $Q$ is semistable in the whole interior of $W$.
If this is the case for $E$ as well, then  $\rf^0_{F,L}$ again has $-p_E$ and $-p_Q$ as its critical points (the function being not differentiable at $-p_E$).
Otherwise, $E$ destabilizes along a wall $W_E$ inside $W$, defined by a sequence $0\to E_1\to E\to E_2\to 0$ with $\odisc(E_1)=0=\odisc(E_2)$:

\begin{figure}[H]
\definecolor{ffffff}{rgb}{1,1,1}
\definecolor{sexdts}{rgb}{0.1803921568627451,0.49019607843137253,0.19607843137254902}
\definecolor{wrwrwr}{rgb}{0.3803921568627451,0.3803921568627451,0.3803921568627451}
\definecolor{rvwvcq}{rgb}{0.08235294117647059,0.396078431372549,0.7529411764705882}
\definecolor{dbwrru}{rgb}{0.8588235294117647,0.3803921568627451,0.0784313725490196}

\begin{tikzpicture}[line cap=round,line join=round,>=triangle 45,x=2.15cm,y=2.15cm]
\draw [line width=1pt,color=rvwvcq] plot[domain=0:3.141592653589793,variable=\t]({cos(\t r)},{sin(\t r)});
\draw [line width=1pt,color=wrwrwr] (-2,0) -- (2,0);
\draw[line width=1pt,dash pattern=on 4pt off 3pt,color=rvwvcq, smooth,samples=100,domain=0:10] 
plot[domain=-1:0.509461,variable=\t]({\t},{\t+1});
\draw[line width=1pt,dash pattern=on 4pt off 3pt,color=rvwvcq, smooth,samples=100,domain=0:10] 
plot[domain=0:0.71,variable=\t]({(-exp(\t)-exp(-\t))/2+1.118034},{(exp(\t)-exp(-\t))});
\draw[line width=1pt,dash pattern=on 4pt off 3pt,color=sexdts, smooth,samples=100,domain=0:10] 
plot[domain=0:1.2,variable=\t]({(-exp(\t)-exp(-\t))/2+1.41421356},{(exp(\t)-exp(-\t))/2});

\draw [line width=1pt,color=dbwrru] plot[domain=0:3.141592653589793,variable=\t]({0.521095305*cos(\t r)+0.2865875972},{0.521095305*sin(\t r)});

\draw[line width=1pt,dash pattern=on 4pt off 3pt,color=dbwrru, smooth,samples=100,domain=0:10] 
plot[domain=-0.2345077078:1.509461-0.2345077078,variable=\t]({\t},{\t+0.2345077078});

\draw[line width=1pt,dash pattern=on 4pt off 3pt,color=dbwrru, smooth,samples=100,domain=0:10] 
plot[domain=-1.509461+0.8076829022:0.8076829022,variable=\t]({\t},{-\t+0.8076829022});

\begin{footnotesize}
\draw[color=rvwvcq] (0.9,.7) node {$W$};
\draw[color=dbwrru] (-0.25,0.4) node {$W_E$};
\draw[color=rvwvcq] (0.52,1.6) node {$\hyp{Q}$};
\draw[color=rvwvcq] (-0.12,1.6) node {$\hyp{F}$};
\draw[color=dbwrru] (-0.72,1.6) node {$\hyp{E_1}$};
\draw[color=sexdts] (-0.4,1.6) node {$\hyp{E}$};
\draw[color=dbwrru] (1.3,1.6) node {$\hyp{E_2}$};
\draw[color=wrwrwr] (-1,-0.15) node {$p_Q$};
\draw[color=wrwrwr] (0.8076829022,-0.15) node {$p_{E_1}$};
\draw[color=wrwrwr] (-0.2345077078,-0.15) node {$p_{E_2}$};

\end{footnotesize}
\end{tikzpicture}
\caption{}
\label{tree2m}
\end{figure}

Clearly, both $E_1$ and $E_2$ are semistable in the whole interior of $W_E$.
With this information, it is easy to describe the HN filtrations of $F$ for all the  $\sigma_{0,\beta}$ with $\beta<\mu_L(F)$, and one obtains
\[
\rf^0_{F,L}(x)=\left\{
 \begin{array}{c l}
 0 & x\leq -p_{E_1}\\
 \ch_2^{-x}(E_1) & -p_{E_1}\leq x\leq-p_{E_2}\\
 \ch_2^{-x}(E) & -p_{E_2}\leq x\leq-p_Q\\
 \ch_2^{-x}(F) & x\geq-p_Q\\
 \end{array}
 \right.
\]
(with the function being differentiable at its three critical points).


More generally, one may try to apply this philosophy for an arbitrarily big $\odisc(F)$ as follows.
If $\rf^0_{F,L}$ is not trivial, then $F$ destabilizes along a wall $W$.
We keep track of its HN filtration
\[
0=F_0\hookrightarrow F_1\hookrightarrow\ldots\hookrightarrow F_{r-1}\hookrightarrow F_r=F
\]
for Bridgeland stability conditions in a (sufficiently small) annulus just below $W$, which necessarily satisfy $\odisc(F_1)+\odisc(F_2/F_1)+\ldots+\odisc(F/F_{r-1})<\odisc(F)$. 

Now, it is possible that some HN factors $F_i/F_{i-1}$ are not semistable in the whole region inside $W$, so they destabilize along a wall $W_i$.
For each such $F_i/F_{i-1}$, again we keep track of its HN filtration for stability conditions just below $W_i$
\[
0\hookrightarrow F_{i,1}/F_{i-1}\hookrightarrow\ldots\hookrightarrow F_{i,r_i-1}/F_{i-1}\hookrightarrow F_{i,r_i}/F_{i-1}=F_i/F_{i-1}
\]
which satisfies $\odisc(F_{i,1}/F_{i-1})+\odisc(F_{i,2}/F_{i,1})+\ldots+\odisc(F_{i,r_i}/F_{i,r_i-1})<\odisc(F_i/F_{i-1})$. 
Proceeding inductively, we finish in a finite number of steps thanks to the strict inequalities on discriminants.

The process yields a tree, in which the final vertices correspond to objects $G$ that are semistable in an open neighborhood of $(0,p_G)$ in the $(\alpha,\beta)$-plane.
By construction, we can consider a lexicographical order on the final vertices; 
\begin{defn}\label{defn:tree}
We say that the tree is \emph{well-ordered} if $G_1<G_2$ implies $p_{G_1}\geq p_{G_2}$ for any two final vertices $G_1$ and $G_2$.
\end{defn}

\begin{ex}\label{ex:tree}
Assume that $\odisc(F)=2m$.
In the stability situation of \autoref{tree2m} $F$ has the following tree, which is well-ordered since $p_{E_1}>p_{E_2}>p_Q$:

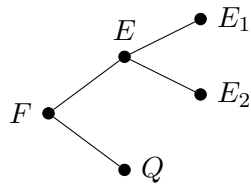
\begin{figure}[H]

\begin{tikzpicture}
      \tikzset{enclosed/.style={draw, circle, inner sep=0pt, minimum size=.15cm, fill=black}}

      \node[enclosed, label={left, yshift=0cm: $F$}] (F) at (0,0) {};
      \node[enclosed, label={above, xshift=0cm: $E$}] (E) at (1,0.75) {};
      \node[enclosed, label={right, yshift=0cm: $Q$}] (Q) at (1,-0.75) {};
      \node[enclosed, label={right, xshift=0cm: $E_1$}] (E1) at (2,1.25) {};
      \node[enclosed, label={right, yshift=0cm: $E_2$}] (E2) at (2,0.25) {};

      \draw (F) -- (E) node[midway, sloped, above] (edge1) {};
      \draw (E) -- (E1) node[midway, right] (edge2) {};
      \draw (E) -- (E2) node[midway, right] (edge3) {};
      \draw (F) -- (Q) node[midway, below] (edge4) {};
\end{tikzpicture}
\end{figure}
\end{ex}

If the tree of $F$ is well-ordered, it is not difficult to recover the Bridgeland limit filtration (and thus the $\sigma_{0,\beta}$-HN filtration) of $F$ at every $\beta<\mu_L(F)$.
The corresponding function $\rf^0_{F,L}$ is piecewise polynomial, and their critical points are precisely the points $-p_G$ for the final vertices $G$ of the tree.
Furthermore, using \autoref{critical} the non-differentiability of $\rf^0_{F,L}$ at the point $-p_G$ is characterized by the condition $\odisc(G)>0$.

Even if we do not expect every Gieseker semistable sheaf to have a well-ordered tree, in many concrete situations this is the case and thus the previous description of Chern degree functions applies.
The examples in the next subsection illustrate this phenomenon for principally polarized abelian surfaces.


\subsection{Finite subschemes on principally polarized abelian surfaces}\label{subsec:CRFGiesekerPPAS} 
In the sequel, $(X,L)$ will be a principally polarized complex abelian surface with $\NS(X)=\bZ\cdot L$.
Under these assumptions, the minimal discriminant is $m=4$ and $X$ is the Jacobian of a genus 2 curve $C$;
after embedding $C$ in $X$ by means of one of its Weierstrass points, we fix $C\in|L|$ as a symmetric theta divisor.

We want to compute $h^0_{\cI_T,L}$ for ideal sheaves $\mathcal{I}_T$ of 0-dimensional subschemes $T$ on $X$.
If $n=h^0(\cO_T)$ denotes the length of $T$, we have:  $\ch(\mathcal{I}_T)=(1,0,-n)$, $\chi_{\mathcal{I}_T,L}(x)=x^2-n$ (in particular, $p_{\cI_T}=-\sqrt{n}$) and $\odisc(\mathcal{I}_T)=4n$ (i.e.~ $\odisc(\mathcal{I}_T)$ is $n$ times the minimal discriminant).

In order to understand geometrically the destabilization of the ideal of a length two 0-dimensional subscheme $T$, we first need to control the translates of $C$ that contain $T$ \footnote{In the language of Jiang--Pareschi, this is the (scheme-theoretic) support of the sheaf $\varphi_L^*R^{2}\Phi_{\cP^{\vee}}((\mathcal{I}_T\otimes L)^{\vee})$, which controls the $h^0$-function in a right neighborhood of $x=1$.} via the following easy lemma.
\begin{lem}\label{gamma}
If $T\subset X$ is a 0-dimensional subscheme of length 2, then the locus $$\{\alpha\in\Pic^0(X):\;h^0(\mathcal{I}_T\otimes L\otimes\alpha)>0\}$$
parametrizing translates of $C$ containing $T$, with its natural scheme structure as support of the sheaf $R^{2}\Phi_{\cP^{\vee}}((\mathcal{I}_T\otimes L)^{\vee})$, is a 0-dimensional subscheme $\Gamma\subset\Pic^0(X)$ of length 2. 
\end{lem}

\begin{proof}
The reduced case is clear and the nonreduced one follows easily from \cite[Example~2.8]{GL}. \qedhere

\end{proof}

\begin{ex}[Case $n=2$]\label{case:n=2}
Let $T$ be a length two 0-dimensional subscheme.
Then the cohomological rank function is 
\[
h^0_{\mathcal{I}_T,L}(x)=\left\{
    \begin{array}{c l}
     0 & x \leq 1\\
     2(x-1)^2 & 1 \leq x \leq 2\\
     x^2-2 & x \geq 2.\\
    \end{array}
    \right.
\]

To illustrate the strategy outlined in the previous section, we consider the first possible wall $W$ for the Chern character $(1,0,-2)$ which has center $-\frac{3}{2}$ and radius $\frac{1}{2}$.
Indeed, it can be defined by the following combinations of Chern characters\footnote{One can use Schmidt's implementation \cite[Appendix]{schmidt} to find the wall candidates. 
}:
\[(1,-L,1)\hookrightarrow(1,0,-2)\twoheadrightarrow(0,L,-3), \;\;\;(2,-2L,2)\hookrightarrow(1,0,-2)\twoheadrightarrow(-1,2L,-4).\]
It is an actual wall for $\mathcal{I}_T$, and the first step in the HN filtration of $\mathcal{I}_T$ after crossing $W$ has Chern character $(2,-2L,2)$.
Let's see how to construct this subobject.

Let $\Gamma\subset\Pic^0(X)$ be the subscheme of \autoref{gamma}, and let $\pi:X\times \Pic^0(X)\longrightarrow X$ and $\sigma:X\times(-\Gamma)\longrightarrow X$ denote the first projection maps.
Then $E=\sigma_*(\pi^*(L^{-1})\otimes\cP_{|X\times(-\Gamma)})$ is a semihomogeneous vector bundle of rank 2 on $X$ with $\ch(E)=(2,-2L,2)$, coming with a natural epimorphism of sheaves $E \twoheadrightarrow \mathcal{I}_T$.
For instance, if $\Gamma$ is reduced (i.e.~ there are two distinct translates $C_1$, $C_2$ of $C$ containing $T$), then the short exact sequence attached to $E \twoheadrightarrow \mathcal{I}_T$ is nothing but the Koszul complex of the complete intersection $T=C_1\cap C_2$.

In any case, taking the short exact sequence of sheaves and rotating the triangle, we obtain a short exact sequence in $\Coh^{-\sqrt{2}}(X)$ defining the HN filtration of $\cI_T$ just below $W$.
Since both HN factors have discriminant 0, they are semistable in the whole interior of $W$ and our tree is well-ordered, and thus we obtain the claimed cohomological rank function for $\cI_T$ since $\chi_{E,L}(x)=2(x-1)^2$.
\end{ex}


For the cases $n\geq3$, we will use the following result in \cite{Meachan} describing the stability of $\mathcal{I}_T$ along the vertical line $\beta=-2$. 
From now on, a finite subscheme inside $X$ will be called \textit{collinear} if it is contained in a (single) translate of the symmetric theta divisor $C$.

\begin{lem}[{\cite[Lemma~3.3.6]{Meachan}}]\label{meachan}
Let $T\in\Hilb^n(X)$.
\begin{enumerate}[{\rm (1)}]
    \item If $n\neq5$: the object $\mathcal{I}_T$ is destabilized at the vertical line $\beta=-2$ if, and only if, $T$ contains a collinear subscheme of colength $m$, for some $0\leq m<\frac{n-2}{2}$.
    
    \noindent In such a case, the destabilizing subobject in $\Coh^{-2}(X)$ is $L^{-1}\otimes\mathcal{I}_{T'}\otimes\alpha$, for some $T'\in\Hilb^m(X)$ and $\alpha\in\Pic^0(X)$.
    
    \item If $n=5$: $\mathcal{I}_T$ can also be destabilized at $\beta=-2$ by $K$, where $K$ is a slope-stable locally free sheaf with $\ch(K)=(2,-3L,4)$.
    
    \noindent This destabilization takes place if, and only if, every subscheme of $T$ with length 4 is non-collinear and contains a unique collinear subscheme of length 3.
\end{enumerate}
\end{lem}

\begin{ex}[Case $n=4$]\label{case:n=4}
If $T$ is a length four $0$-dimensional subscheme, we consider two subcases.

If no translate of $C$ contains $T$, then $\cI_T$
has trivial cohomological rank function:
\[
    h^0_{\mathcal{I}_T,L}(x)=\left\{
    \begin{array}{c l}
     0 & x \leq 2\\
     x^2-4 & x \geq 2.\\
    \end{array}
    \right.
\]
Indeed, by \autoref{meachan} $\mathcal{I}_T$ remains semistable along the vertical line $\beta=-2$.
Since $p_{\cI_T}=-2$, this means that $\mathcal{I}_T$ is semistable in the whole region $\beta<0$.

Now assume that $T$ is collinear.
Then, the cohomological rank function of $\cI_T$ is \[
    h^0_{\mathcal{I}_T,L}(x)=\left\{
    \begin{array}{c l}
     0 & x \leq 1\\
     (x-1)^2 & 1 \leq x \leq 2\\
     (x-1)^2+(x-2)^2 & 2 \leq x \leq 3\\
     x^2-4 & x \geq 3.\\
    \end{array}
    \right.
\]

To simplify the notation and illustrate the strategy, let us prove this expression when $T=\set{p,q,r,s}$ is a reduced subscheme contained in a translate $C_1$ of $C$.
Indeed, consider any two points of $T$: there is another translate of $C$ containing them, unless $C_1$ has the same tangent direction at these points. 
Since the Gauss map of $C_1$ has degree 2, it follows that (possibly after reordering the points of $T$) there are two translates of $C$ passing through $p$ and $q$ (resp.~through $r$ and $s$)  simultaneously; being $C_1$ one of them, we denote by $C_2$ (resp.~$C_3$) the other one.
We take also $\alpha,\beta,\gamma\in\Pic^0(X)$ such that $C_1\in|L\otimes\alpha|$, $C_2\in|L\otimes\beta|$ and $C_3\in|L\otimes\gamma|$.

The Chern character $(1,0,-4)$ has a unique possible wall $W$, of center $-\frac{5}{2}$ and radius $\frac{3}{2}$.
Since $T$ is collinear, we can destabilize $\mathcal{I}_T$ (as predicted by \autoref{meachan}) via the exact sequence
\[
0 \longrightarrow L^{-1}\otimes\alpha^{-1}\overset{\cdot s_1}{\longrightarrow}
 \mathcal{I}_T\longrightarrow \cO_{C_1}(-p-q-r-s)\longrightarrow 0
\] 
where $s_1\in H^0(L\otimes\alpha)$ defines $C_1$.
This sequence gives the HN filtration of $\cI_T$ just below $W$.

The subobject $E=L^{-1}\otimes\alpha^{-1}$ is everywhere semistable since $\odisc(E)=0$.
The quotient $Q=\cO_{C_1}(-p-q-r-s)$ has $\odisc(Q)=4$, and destabilizes along a wall inside $W$ defined by a sequence
\[
0 \longrightarrow L^{-2}\otimes\beta^{-1}\otimes\gamma^{-1}\overset{\cdot {s_2s_3}_{|C_1}}{\longrightarrow}
 \cO_{C_1}(-p-q-r-s)\longrightarrow (L^{-3}\otimes\alpha^{-1}\otimes\beta^{-1}\otimes\gamma^{-1})[1]\longrightarrow 0
\]
where $s_2$ and $s_3$ are sections defining $C_2$ and $C_3$.

Both $Q_1=L^{-2}\otimes\beta^{-1}\otimes\gamma^{-1}$ and $Q_2=(L^{-3}\otimes\alpha^{-1}\otimes\beta^{-1}\otimes\gamma^{-1})[1]$ have $\odisc=0$, so it is not necessary to study further destabilizations.
We obtain a well-ordered tree with final vertices $E$, $Q_1$ and $Q_2$, which gives the desired cohomological rank function, since $\chi_{E,L}(x)=(x-1)^2$ and $\chi_{E,L}(x)+\chi_{Q_1,L}(x)=(x-1)^2+(x-2)^2$.
\end{ex}

\begin{ex}[Case $n=3$]\label{case:n=3}
If $T$ is a length three 0-dimensional subscheme, we consider first the case where $T$ is collinear, where the cohomological rank function of $\cI_T$ is 
\[
    h^0_{\mathcal{I}_T,L}(x)=\left\{
    \begin{array}{c l}
     0 & x \leq 1\\
     (x-1)^2 & 1 \leq x \leq 2\\
     x^2-3 & x \geq 2.\\
    \end{array}
    \right.
\]
Indeed, if $C_1\in|L\otimes\alpha|$ is the translate of $C$ containing $T$, then $\mathcal{I}_T$ destabilizes along the wall $W$ defined by
$$
0 \longrightarrow L^{-1}\otimes\alpha^{-1}\overset{\cdot s}{\longrightarrow}
 \mathcal{I}_T\longrightarrow \cO_{C_1}(-T)\longrightarrow 0
$$ 
where the section $s\in H^0(L\otimes\alpha)$ defines $C_1$.
This was already predicted by \autoref{meachan}.

The subobject $L^{-1}\otimes\alpha^{-1}$ is obviously semistable everywhere inside $W$; the same happens for the quotient $\cO_{C_1}(-T)$, since it is a line bundle of odd degree on a genus 2 Abel-Jacobi curve (recall \autoref{extrivial}.\eqref{enum:AbelJacobi}).
This completes the tree and we obtain the desired cohomological rank function.

If $T$ is not collinear, then the cohomological rank function of $\cI_T$ is 
\[
    h^0_{\mathcal{I}_T,L}(x)=\left\{
    \begin{array}{c l}
     0 & x \leq \frac{3}{2}\\
     4x^2-12x+9 & \frac{3}{2} \leq x \leq 2\\
     x^2-3 & x \geq 2.\\
    \end{array}
    \right.
\]

Indeed, according to \autoref{meachan} $\mathcal{I}_T$ remains semistable along the whole line $\beta=-2$.
The next possible wall $W'$ has center $-\frac{7}{4}$ and radius $\frac{1}{4}$; let us describe the destabilization of $\cI_T$ along this wall, under the assumption that $T=\set{p,q,r}$ is reduced and every pair of points on $T$ is contained in two distinct translates of $C$.

Let $C_1=C+t_1\in|L\otimes\alpha|$ and $\widetilde{C_1}=C+\widetilde{t_1}\in|L\otimes\widetilde{\alpha}|$ be the two translates of $C$ passing through the points $p$ and $q$.
Among all the translates of $C$ passing through $r$, we take two curves $C_2=C+t_2$ and $\widetilde{C_2}=C+\widetilde{t_2}$ such that $t_2-\widetilde{t_2}=t_1-\widetilde{t_1}$.
This is possible because the subtraction morphism $C\times C\longrightarrow S$ is surjective (the curve $C$ being non-degenerate).

Writing $C_2\in|L\otimes\beta|$ and $\widetilde{C_2}\in|L\otimes\beta\otimes\widetilde{\alpha}\otimes\alpha^{-1}|$ (according to the condition $t_2-\widetilde{t_2}=t_1-\widetilde{t_1}$), it is easy to check that $h^0(\mathcal{I}_X\otimes
L^2\otimes\beta\otimes\widetilde{\alpha})\geq2$ and thus $h^1(\mathcal{I}_X\otimes L^2\otimes\beta\otimes\widetilde{\alpha})\geq1$.
By Serre duality, the last condition reads $\Ext^1(\mathcal{I}_X,
L^{-2}\otimes\beta^{-1}\otimes\widetilde{\alpha}^{-1})\neq0$.
A (rotated) nontrivial extension gives a short exact sequence, destabilizing $\cI_T$ along $W'$.

By repeating this process (starting with the two curves through $p$ and $r$, and the two curves through $q$ and $r$) and taking direct sum, it is possible to destabilize $\mathcal{I}_T$ via a short exact sequence $0\to E\to \cI_T\to Q\to 0$ corresponding to the HN filtration of $\cI_T$ just below $W'$.

Since $\ch(E)=(4,-6L,9)$ and $\ch(Q)=(-3,6L,-12)$, both $E$ and $Q$ have $\odisc=0$ and thus the construction of the tree has no more steps. 
\end{ex}

\begin{ex}[Case $n=5$]\label{case:n=5}
The Chern character $(1,0,-5)$ has three possible walls intersecting the vertical line $\beta=-2$, corresponding to the three special geometric situations described in \autoref{meachan}:
\begin{itemize}[\textbullet]
    \item A wall $W_1$ of center $-3$ and radius $2$.
    Destabilization of $\cI_T$ takes place along $W_1$ if and only if $T$ is collinear.
    
    \item A wall $W_2$ of center $-\frac{5}{2}$ and radius $\frac{1}{2}\sqrt{5}$.
    $\cI_T$ gets destabilized along $W_2$ if, and only if, $T$ itself is not collinear but contains a collinear subscheme of length 4.
    
    \item A wall $W_3$ of center $-\frac{7}{3}$ and radius $\frac{2}{3}$.
    $\cI_T$ destabilizes along $W_3$ if, and only if, every length 4 subscheme is non-collinear and contains a unique length 3 collinear subscheme.

\end{itemize}

In these three special geometric situations, the usual method constructs the well-ordered tree of $\cI_T$ and hence recovers $h^0_{\cI_T,L}$.
We give some details of the case where $\cI_T$ gets destabilized along the wall $W_2$, since it presents a remarkable peculiarity and illustrates a special situation in the characterization of critical points (see \autoref{critical}).

Assume for simplicity that $T=\set{p_1,p_2,p_3,p_4,p_5}$ is reduced.
If $Y=\set{p_1,p_2,p_3,p_4}$ and $C_1\in|L\otimes\alpha|$ is the translate of $C$ containing $Y$, then
\[
0 \longrightarrow \mathcal{I}_{p_5}\otimes L^{-1}\otimes\alpha^{-1}\overset{\cdot s_1}{\longrightarrow}
 \mathcal{I}_T\longrightarrow \cO_{C_1}(-p_1-p_2-p_3-p_4)\longrightarrow 0
\]
is the exact sequence destabilizing $\mathcal{I}_T$ along the wall $W_2$.
It is also the HN filtration of $\cI_T$ for stability conditions just below $W_2$.
Since both HN factors have $\odisc=4$, we have to study possible further destabilizations:

\begin{itemize}[\textbullet]
    \item $E=\mathcal{I}_{p_5}\otimes L^{-1}\otimes\alpha^{-1}$ is semistable everywhere inside $W_2$, according to \autoref{extrivial}.\eqref{enum:idealpoint}.
    
    \item The sheaf $Q=\cO_{C_1}(-p_1-p_2-p_3-p_4)$ destabilizes as described in the case of 4 collinear points (see \autoref{case:n=4}).
    This destabilization takes place along a wall $W_Q$ of center $-\frac{5}{2}$ and radius $\frac{1}{2}$.
    Both the subobject $Q_1$ and the quotient $Q_2$ inducing this wall have $\odisc=0$, so we are done.
    
\end{itemize}

\autoref{fig:treeideal5}  shows the position of  walls and hyperbolas in the $(\alpha,\beta)$-plane.

\begin{figure}[ht]
\definecolor{ffffff}{rgb}{1,1,1}
\definecolor{sexdts}{rgb}{0.1803921568627451,0.49019607843137253,0.19607843137254902}
\definecolor{wrwrwr}{rgb}{0.3803921568627451,0.3803921568627451,0.3803921568627451}
\definecolor{rvwvcq}{rgb}{0.08235294117647059,0.396078431372549,0.7529411764705882}
\definecolor{dbwrru}{rgb}{0.8588235294117647,0.3803921568627451,0.0784313725490196}

\begin{tikzpicture}[line cap=round,line join=round,>=triangle 45,x=2.35cm,y=2.35cm]

\draw [line width=1pt,color=wrwrwr] (-4.5,0) -- (-0.5,0);

\draw[line width=1pt,dash pattern=on 4pt off 3pt,color=dbwrru, smooth,samples=100,domain=0:10] 
plot[domain=0:0.71,variable=\t]({2.236068*(-exp(\t)-exp(-\t))/2},{2.236068*(exp(\t)-exp(-\t))/2});

\draw[line width=1pt,dash pattern=on 4pt off 3pt,color=dbwrru, smooth,samples=100,domain=0:10] 
plot[domain=0:1.33,variable=\t]({(-exp(\t)-exp(-\t))/2-1},{(exp(\t)-exp(-\t))/2});

\draw[line width=1pt,dash pattern=on 4pt off 3pt,color=sexdts, smooth,samples=100,domain=0:10] 
plot[domain=0:1.75,variable=\t]({-2.5},{\t});

\draw[line width=1pt,dash pattern=on 4pt off 3pt,color=rvwvcq, smooth,samples=100,domain=0:10] 
plot[domain=-3.75:-2,variable=\t]({\t},{-\t-2});

\draw[line width=1pt,dash pattern=on 4pt off 3pt,color=rvwvcq, smooth,samples=100,domain=0:10] 
plot[domain=-3:-1.25,variable=\t]({\t},{\t+3});

\draw [line width=1pt,color=rvwvcq] plot[domain=0:3.141592653589793,variable=\t]({0.5*cos(\t r)-2.5},{0.5*sin(\t r)});

\draw [line width=1pt,color=dbwrru] plot[domain=0:3.141592653589793,variable=\t]({2.236068*cos(\t r)/2-2.5},{2.236068*sin(\t r)/2});

\begin{footnotesize}
\draw[color=rvwvcq] (-2.95,0.42) node {$W_Q$};
\draw[color=dbwrru] (-3.5,0.77) node {$W_2$};
\draw[color=rvwvcq] (-1.07,1.8) node {$\hyp{Q_2}$};
\draw[color=sexdts] (-2.38,1.8) node {$\hyp{Q}$};
\draw[color=rvwvcq] (-3.9,1.8) node {$\hyp{Q_1}$};
\draw[color=dbwrru] (-3.1,1.81) node {$\hyp{E}$};
\draw[color=dbwrru] (-2.76,1.8) node {$\hyp{\cI_T}$};
\end{footnotesize}
\end{tikzpicture}
\caption{}
\label{fig:treeideal5}
\end{figure}

Again our tree (with final vertices $E,Q_1,Q_2$) is well-ordered, but in this case we have an overlap $p_E=p_{Q_1}=-2$.
In the terminology of \autoref{critical}, this implies that $x=2$ is a critical point where conditions \eqref{cond:a} and \eqref{cond:c} are simultaneously fulfilled.

The cohomological rank function $h^0_{\cI_T,L}$ can be recovered as
\[
    h^0_{\mathcal{I}_T,L}(x)=\left\{
    \begin{array}{c l}
     0 & x \leq 2\\
     \chi_{E,L}(x)+\chi_{Q_1,L}(x)=2(x-1)(x-2) & 2 \leq x \leq 3\\
     x^2-5 & x \geq 3.\\
    \end{array}
    \right.
\]

Finally, if $T$ satisfies none of the special geometric conditions, the next possible wall $W$ for $\cI_T$ has center $-\frac{9}{4}$ and radius $\frac{1}{4}$.
According to \cite[Theorem~10.2]{Maciocia2}, the locus
\[
\widetilde{\Gamma}=\{\alpha\in\Pic^0(X):\;h^0(\mathcal{I}_T\otimes L^2\otimes\alpha)>0\}
\]
(with its natural scheme structure as support of the sheaf $R^{2}\Phi_{\cP^{\vee}}((\mathcal{I}_T\otimes L^2)^{\vee})$) is a 0-dimensional subscheme of length 5.
In other words, among all the translated linear systems $|L^2\otimes\alpha|$ there exist exactly five curves (counted with multiplicities) that contain $T$.

Similarly to the case $n=2$, one can use these five curves to destabilize $\cI_T$ along $W$.
The tree of $\cI_T$ consists of this single step and the cohomological rank function reads
\[
    h^0_{\mathcal{I}_T,L}(x)=\left\{
    \begin{array}{c l}
     0 & x \leq 2\\
     5(x-2)^2 & 2 \leq x \leq \frac{5}{2}\\
     x^2-5 & x \geq \frac{5}{2}.\\
    \end{array}
    \right.
\]
\end{ex}

\bibliography{refer}
\bibliographystyle{alphaspecial}

\end{document}